\pdfoutput=1
\documentclass[hidelinks, reqno, 10pt]{amsart}

\usepackage[numbers, square]{natbib}
\setcitestyle{citesep={,}}

\usepackage{graphicx}
\usepackage{subcaption}
\usepackage{tabularx}
\usepackage{booktabs}
\usepackage{array}
\usepackage{amsmath}
\usepackage{amsfonts}
\usepackage{amssymb}
\usepackage{amsthm}
\usepackage{thmtools}
\usepackage{tikz-cd}
\usepackage{yhmath}
\usepackage[scr = boondox]{mathalfa}	
\usepackage{bm} 
\usepackage{bbm} 
\usepackage[inline]{enumitem}
\usepackage{permute}
\usepackage{slashed}
\usepackage[dvipsnames]{xcolor}
\usepackage[pagebackref = true, colorlinks, linkcolor = Red, citecolor = Green, bookmarksdepth=2, linktocpage=true, hypertexnames=false]{hyperref}
\usepackage[capitalize]{cleveref}
\usepackage{comment}

\usepackage[T1]{fontenc}
\usepackage{makecell}

\usepackage{tikz}
\usetikzlibrary{calc}

\allowdisplaybreaks

\DeclareMathOperator{\interior}{int}
\DeclareMathOperator{\exterior}{ext}

\DeclareMathOperator{\Aut}{Aut}
\DeclareMathOperator{\Id}{Id}

\DeclareMathOperator{\supp}{supp}

\DeclareMathOperator{\SL}{SL}

\DeclareMathOperator{\SO}{SO}

\DeclareMathOperator{\PSL}{PSL}
\DeclareMathOperator{\PGL}{PGL}

\DeclareMathOperator{\Lip}{Lip}

\DeclareMathOperator{\proj}{proj}

\DeclareMathOperator{\T}{T}

\DeclareMathOperator{\diam}{diam}

\DeclareMathOperator{\len}{len}
\let\Pr\relax
\DeclareMathOperator{\Pr}{Pr}

\DeclareMathOperator{\Ad}{Ad}

\DeclareMathOperator{\Span}{span}

\DeclareMathOperator{\Fix}{Fix}

\DeclareMathOperator{\Lie}{Lie}
\DeclareMathOperator{\Li}{Li}

\newcommand{\C}{\mathbb{C}}
\newcommand{\N}{\mathbb{N}}
\newcommand{\bP}{\mathbb{P}}

\newcommand{\R}{\mathbb{R}}
\newcommand{\Z}{\mathbb{Z}}

\newcommand{\LieG}{\mathfrak{g}}
\newcommand{\LieA}{\mathfrak{a}}
\newcommand{\LieN}{\mathfrak{n}}
\newcommand{\LieK}{\mathfrak{k}}
\newcommand{\LieM}{\mathfrak{m}}
\newcommand{\LieP}{\mathfrak{p}}

\newcommand{\LieS}{\mathfrak{s}}
\newcommand{\LieSimple}{\mathfrak{S}}
\newcommand{\Fboundary}{\mathcal{F}_\Theta}
\newcommand{\iFboundary}{\mathcal{F}_{\involution(\Theta)}}
\newcommand{\Gboundary}{\partial\Gamma}
\newcommand{\involution}{\mathsf{i}}
\newcommand{\growthindicator}{\psi_\Gamma}

\newcommand{\limitset}{\Lambda_\Theta}
\newcommand{\ilimitset}{\Lambda_{\involution(\Theta)}}

\newcommand{\limitcone}{\mathcal{L}_\Theta}
\newcommand{\BMS}{m^\mathrm{BMS}_\sfv}

\newcommand{\jj}{j_0}

\newcommand{\sfv}{\mathsf{v}}
\newcommand{\sfw}{\mathsf{w}}
\newcommand{\ww}{\mathscr{w}}

\DeclareFontFamily{U}{mathb}{\hyphenchar\font45}
\DeclareFontShape{U}{mathb}{m}{n}{
<5> <6> <7> <8> <9> <10> gen * mathb
<10.95> mathb10 <12> <14.4> <17.28> <20.74> <24.88> mathb12
}{}
\DeclareSymbolFont{mathb}{U}{mathb}{m}{n}
\DeclareMathSymbol{\bigast}{2}{mathb}{"06}

\def\XXint#1#2#3{{\setbox0=\hbox{$#1{#2#3}{\int}$}
\vcenter{\hbox{$#2#3$}}\kern-.5\wd0}}

\theoremstyle{plain}
\newtheorem{theorem}{Theorem}[section]
\newtheorem{proposition}[theorem]{Proposition}
\newtheorem{lemma}[theorem]{Lemma}
\newtheorem{corollary}[theorem]{Corollary}

\newtheorem{questions}[theorem]{Questions}

\theoremstyle{definition}
\newtheorem{definition}[theorem]{Definition}

\theoremstyle{remark}
\newtheorem{remark}[theorem]{Remark}

\setlist[enumerate,1]{ref=(\arabic*)}
\setlist[enumerate,2]{ref=(\theenumi)(\alph*)}
\setlist[enumerate,3]{ref=(\theenumi)(\theenumii)(\roman*)}
\setlist[enumerate,4]{ref=(\theenumi)(\theenumii)(\theenumiii)(\Alph*)}

\newlist{alternative}{enumerate}{4}     
\setlist[alternative,1]{label=(\arabic*), ref=(\arabic*)}
\setlist[alternative,2]{label=(\alph*), ref=(\thealternativei)(\alph*)}
\setlist[alternative,3]{label=(\roman*), ref=(\thealternativei)(\thealternativeii)(\roman*)}
\setlist[alternative,4]{label=(\Alph*), ref=(\thealternativei)(\thealternativeii)(\thealternativeiii)(\Alph*)}

\Crefname{enumi}{Property}{Properties}
\Crefname{alternativei}{Alternative}{Alternatives}
\Crefname{subsection}{Subsection}{Subsections}

\begin{document}
\title[Exponential prime orbit theorems for Anosov subgroups]{Exponential prime orbit theorems for Anosov subgroups}

\author{Michael Chow}
\address{Department of Mathematics, Yale University, New Haven, CT 06511, USA}
\email{mikey.chow@yale.edu}

\author{Pratyush Sarkar}
\address{Department of Mathematics, UC San Diego, La Jolla, CA 92093, USA}
\curraddr{Department of Mathematics, University of Utah, Salt Lake City, UT 84112}
\email{p.sarkar@utah.edu}

\date{\today}

\begin{abstract}
Let $\Gamma$ be a Zariski dense Anosov subgroup of a connected semisimple real algebraic group---these are higher rank analogues of convex cocompact subgroups. Let us measure the Jordan projections with any linear form which is positive on the limit cone of $\Gamma$. We prove a corresponding counting theorem with a power saving error term for the conjugacy classes of loxodromic elements in $\Gamma$. The proof is based on interpreting the Jordan projections as periods of a natural flow associated to $\Gamma$ and proving exponential mixing. We also prove the existence of a spectral gap for the Selberg zeta function.
\end{abstract}

\maketitle

\setcounter{tocdepth}{1}
\tableofcontents

\renewcommand{\Fboundary}{\mathcal{F}_\Theta}
\renewcommand{\iFboundary}{\mathcal{F}_{\involution\Theta}}
\renewcommand{\limitset}{\Lambda_\Theta}
\renewcommand{\ilimitset}{\Lambda_{\involution\Theta}}
\renewcommand{\limitcone}{\mathcal{L}_\Theta}
\renewcommand{\growthindicator}{\psi_\Theta}
\newcommand{\BMSX}{m_{\mathcal{X}}}

\section{Introduction}
\label{sec:Introduction}
Prime orbit theorems have been studied for decades, perhaps since the work of Huber from 1961 (see \cite{Nau05} and references therein). Techniques to prove such theorems usually involve the spectral theory of Selberg or Ruelle zeta functions, spectral theory of transfer operators, or mixing of flows, to name a few. For instance, in Margulis's thesis from 1970 (see the translation \cite{Mar04}), he used mixing of the geodesic flow to prove the following prime geodesic theorem: on a Riemannian manifold $\mathcal{M}$ with negative sectional curvature, we have
\begin{align*}
	\#\mathcal{P}(T) \sim \frac{e^{\delta T}}{\delta T} \qquad \text{as $T \to +\infty$}
\end{align*}
where $\mathcal{P}(T)$ is the set of all primitive oriented geodesics in $\mathcal{M}$ of length at most $T > 0$, and $\delta > 0$ is the topological entropy of the geodesic flow on $\T^1(\mathcal{M})$. In particular, the above theorem applies to $\mathcal{M} = \Gamma\backslash \mathbb{H}^n$ for cocompact lattices $\Gamma < \SO(n, 1)^\circ$. In this setting, the above theorem can be formulated independently of the geometry by recalling that oriented geodesics in $\mathcal{M}$ are in one-to-one correspondence with conjugacy classes of loxodromic elements in $\Gamma$. Remaining in the hyperbolic manifold setting, the state of the art is that using effective versions of the aforementioned dynamical techniques, we have a prime geodesic theorem for convex cocompact (or even geometrically finite) discrete subgroups $\Gamma < \SO(n, 1)^\circ$, with a power saving error term \cite{Nau05,Sto11,MMO14,LP23}.

In the same vein, one of the motivations of this paper is to further develop the dynamical techniques and obtain a more general prime orbit theorem which simultaneously satisfies the following two features:
\begin{itemize}
\item the theorem applies to Anosov subgroups of higher rank semisimple real algebraic groups;
\item the theorem provides a power saving error term.
\end{itemize}
In present-day literature, Anosov subgroups are thought of as the natural generalization of convex cocompact subgroups to higher rank. They were first introduced by Labourie \cite{Lab06} and later generalized by Guichard--Wienhard \cite{GW12}, and includes many interesting geometric examples such as Schottky subgroups, and the images of Hitchin representations \cite{Lab06} into $\PSL_n(\R)$, strongly convex cocompact projective representations into $\PGL_n(\R)$ \cite{DGK17}, maximal representations into $\mathrm{PSp}_{2n}(\R)$ and $\mathrm{PO}(n, 2)$ \cite{BIL05,BIW24}, and Barbot representations into $\PSL_n(\R)$ \cite{Bar10}. Recently, there has been a lot of activity in understanding the dynamics related to Anosov subgroups (see for instance \cite{Sam15,BCLS15,LO20a,LO20b,ELO23,BLLO23,CS23,Sam24,KO24}).

For the sake of simplicity, we first state our prime orbit theorem with a power saving error term in terms of counting conjugacy classes of primitive loxodromic elements, and for the time being, without introducing any dynamical system. The appropriate dynamical system will be introduced in \cref{subsec:TheDynamicalSystemsAndMixing} and a more refined prime orbit theorem will be stated in \cref{subsec:EssentialSpectralGapAndPrimeOrbitTheorem}.

We first introduce the required objects; see \cref{sec:AnosovSubgroupsAndTheirTranslationFlows} for further definitions and details. Let $G$ be a connected semisimple real algebraic group, $A < G$ be a maximal real split torus, and fix a closed positive Weyl chamber $\LieA^+ \subset \LieA:= \Lie(A)$. Fix a nonempty subset $\Theta \subset \Pi \subset \LieA^*$ of the set of all simple roots for $(\LieG,\LieA^+)$ and define
\begin{align*}
\LieA_\Theta &:= \bigcap_{\alpha \in \Pi - \Theta} \ker\alpha \subset \LieA, & \LieA_\Theta^+ &:= \LieA_\Theta \cap \LieA^+.
\end{align*}
Let $\Gamma < G$ be a torsion-free Zariski dense \emph{$\Theta$-Anosov subgroup} (see \cref{subsec:TheDynamicalSystemsAndMixing}). Let $\lambda_\Theta: G \to \LieA_\Theta^+$ be the $\Theta$-Jordan projection and $\limitcone \subset \interior\LieA_\Theta^+ \cup \{0\}$ be the minimal closed cone containing $\lambda_\Theta(\Gamma)$, called the $\Theta$-limit cone of $\Gamma$.
For any linear form $\psi \in \LieA_\Theta^*$ which is positive on $\limitcone - \{0\}$ and $T > 0$, define
\begin{align*}
\mathcal{P}_\psi(T) := \{[\gamma] \in [\Gamma]: \text{$\gamma$ is primitive loxodromic, $\psi(\lambda_\Theta(\gamma)) \leq T$}\}
\end{align*}
where $[\Gamma]$ denotes the set of conjugacy classes in $\Gamma$.

\begin{theorem}
\label{thm:ExponentialPrimeOrbitTheoremSimple}
Let $\psi \in \LieA_\Theta^*$ be a linear form which is positive on $\limitcone - \{0\}$.
There exist $\delta > 0$, $\eta > 0$ and $C > 0$ such that for all $T > 0$, we have
	\begin{align*}
		\bigl|\#\mathcal{P}_\psi(T) - \Li\bigl(e^{\delta T}\bigr)\bigr| \leq C e^{(\delta - \eta)T}.
	\end{align*}
\end{theorem}

Here, $\Li: (2, +\infty) \to \R$ is the offset logarithmic integral function defined by
\begin{align}
	\label{eqn:Li function}
	\Li(x) = \int_2^x \frac{1}{\log(t)} \, dt = \frac{x}{\log(x)} + \frac{x}{\log(x)^2} + O\left(\frac{x}{\log(x)^3}\right) \qquad \text{for all $x > 2$}.
\end{align}

For $\#\Theta = 1$, the above theorem was obtained in \cite{DMS24} shortly before the completion of our work. Thus, the above theorem is new for $\#\Theta > 1$ and we obtain an independent proof for $\#\Theta = 1$.

\subsection{The dynamical systems and mixing}
\label{subsec:TheDynamicalSystemsAndMixing}
Let us now elaborate on the setting and introduce the principal dynamical systems of interest. We refer the reader to \cref{sec:AnosovSubgroupsAndTheirTranslationFlows} for further definitions and details.

Let $P_\Theta$ denote the standard parabolic subgroup of $G$ corresponding to $\Theta$ and $P_\Theta = S_\Theta A_\Theta N_\Theta$ be its Langlands decomposition where $S_\Theta$ centralizes $A_\Theta := \exp(\LieA_\Theta)$. The $\Theta$-Furstenberg boundary is $\Fboundary := G/P_\Theta$.
Denote the unique open $G$-orbit in $\Fboundary \times \iFboundary$, where $\involution:\LieA \to \LieA$ is the opposition involution, by $\Fboundary^{(2)}$.

We fixed $\Gamma < G$ to be a torsion-free Zariski dense \emph{$\Theta$-Anosov subgroup}. One definition which is useful from a dynamical systems point of view is the following: $\Gamma$ is a torsion-free Gromov hyperbolic group with continuous $\Gamma$-equivariant maps
\begin{align*}
\zeta_\Theta&: \Gboundary \to \Fboundary, & \zeta_{\involution\Theta}&: \Gboundary \to \iFboundary,
\end{align*}
from the Gromov boundary $\Gboundary$, such that both of the following hold:
\begin{enumerate}
\item $(\zeta_\Theta(x), \zeta_{\involution\Theta}(x))$ is a singular pair for all $x \in \Gboundary$;
\item $(\zeta_\Theta(x), \zeta_{\involution\Theta}(y)) \in \Fboundary^{(2)}$ for all distinct $x, y \in \Gboundary$.
\end{enumerate}
The images $\limitset := \zeta_\Theta(\Gboundary)$ and $\ilimitset := \zeta_{\involution\Theta}(\Gboundary)$ are the unique $\Gamma$-minimal subsets of $\Fboundary$ and $\iFboundary$. Define
\begin{align*}
\limitset^{(2)} := (\zeta_\Theta \times \zeta_{\involution\Theta})(\{(x, y) \in \Gboundary^2: x \neq y\}) \subset \Fboundary^{(2)}.
\end{align*}
Unless $\Theta = \Pi$ (in which case $S_\Pi$ is compact), the $\Gamma$-action is not necessarily properly discontinuous on $G/S_\Theta \cong \Fboundary^{(2)} \times \LieA_\Theta$
(see \cite[\S\,1.5]{GGKW17}).
However, the $\Gamma$-action is properly discontinuous on the $\Gamma$-invariant subset $\limitset^{(2)} \times \LieA_\Theta$. Let $\Omega := \Gamma \backslash  \bigl(\limitset^{(2)} \times \LieA_\Theta\bigr)$ be the corresponding locally compact Hausdorff metric space.

For each $\sfv \in \LieA_\Theta$, the translation action of $\R\sfv$ on the $\LieA_\Theta$-coordinate of $\limitset^{(2)} \times \LieA_\Theta$ induces the \emph{one-parameter diagonal flow} on $\Omega$ which we denote by $\{a_{t\sfv}\}_{t \in \R}$. Restricting to $\sfv \in \interior\limitcone$, we have a Bowen--Margulis--Sullivan measure $\BMS$ on $\Omega$. Altogether, we have the family of dynamical systems
\begin{align*}
\bigl\{\bigl(\Omega, \BMS, \{a_{t\sfv}\}_{t \in \R}\bigr)\bigr\}_{\sfv \in \interior\limitcone}.
\end{align*}
By arguments of the authors \cite{CS23} for general $\Theta$ when ignoring the $S_\Pi$-valued holonomy, and also of \cite[Appendix B]{Sam24}, the above dynamical systems are known to be \emph{locally mixing}: there exists $C_\sfv > 0$ such that for all $\phi_1, \phi_2 \in C_{\mathrm{c}}(\Omega)$, we have
\begin{align}
	\label{eqn:ThetaLocalMixing}
	\lim_{t \to +\infty} t^{\frac{\#\Theta - 1}{2}} \int_{\Omega} \phi_1(a_{t\sfv}x) \phi_2(x) \, d\BMS(x)
	= C_\sfv\BMS(\phi_1) \BMS(\phi_2).
\end{align}
See references therein for more details.
It is an interesting (and still open) question as to what the rate of mixing above should be in general. It is known to the authors that Dolgopyat's method cannot be carried out as the non-concentration property or its generalization (cf. \cite{SW21,CS22}) fails miserably whenever $G$ is of higher rank; and hence, the rate is expected to be subexponential in general when $\#\Theta > 1$. See the results below for the case $\#\Theta = 1$.

On the other hand, the situation is completely different for a related family of dynamical systems which we now describe. For each $\sfv \in \interior\limitcone$, the $\Gamma$-action on $\limitset^{(2)} \times \LieA_\Theta$ descends via the linear form $\psi_{\sfv} = \langle \growthindicator(\sfv)^{-1}\nabla\growthindicator(\sfv), \cdot \rangle$ to a $\Gamma$-action on $\limitset^{(2)} \times \R$ which is properly discontinuous and cocompact. It turns out that the projected compact Hausdorff metric spaces $\Gamma\backslash\bigl(\limitset^{(2)} \times \R\bigr)$ are all Lipschitz homeomorphic to each other and we may identify them as a single canonical object $\mathcal{X}$. As a result, applying the linear form $\psi_\sfv$, the family of dynamical systems $\bigl\{\bigl(\Omega, \BMS, \{a_{t\sfv}\}_{t \in \R}\bigr)\bigr\}_{\sfv \in \interior\limitcone}$ projects to another family of dynamical systems
\begin{align*}
\bigl\{\bigl(\mathcal{X}, \BMSX^{\sfv}, \{a^{\sfv}_t\}_{t \in \R}\bigr)\bigr\}_{\sfv \in \interior\limitcone}
\end{align*}
where
\begin{itemize}
\item $\mathcal{X}$ is a compact Hausdorff metric space over which $\Omega \to \mathcal{X}$ is a trivial $\R^{\#\Theta - 1}$-bundle,
\item $\BMSX^{\sfv}$ is a probability measure on $\mathcal{X}$ such that we have the product structure $\BMS = \BMSX^{\sfv} \otimes \mathrm{Leb}_{\R^{\#\Theta-1}}$ on $\Omega$,
\item $\{a^{\sfv}_t\}_{t \in \R}$ is the \emph{translation flow} on $\mathcal{X}$, whose topological entropy is $\delta_\sfv = \growthindicator(\sfv)$, induced by the one-parameter diagonal flow $\{a_{t\sfv}\}_{t \in \R}$ on $\Omega$ via $\psi_\sfv$.
\end{itemize}
See \cref{fig:DynamicalSystems} for a comprehensive diagram of the various spaces.

\begin{figure}[htbp]
\centering
\[
\begin{array}{ccc}
\left[
\begin{tikzcd}[column sep=2em, row sep=2em]
G \arrow[d] \\
G/S_\Theta \arrow[d, phantom, "\cong"] \\
\Fboundary^{(2)} \times \LieA_\Theta \arrow[d, "\psi_\sfv"] \arrow[r, phantom, "\supset"] & \limitset^{(2)} \times \LieA_\Theta \arrow[d, "\psi_\sfv"] \\
\Fboundary^{(2)} \times \R \arrow[r, phantom, "\supset"] & \limitset^{(2)} \times \R
\end{tikzcd}
\right]
& \xrightarrow{\Gamma\backslash\bullet} &
\left[
\begin{tikzcd}[column sep=2em, row sep=2em]
\Gamma\backslash G \arrow[d] \\
\Gamma\backslash G/S_\Theta \arrow[d, phantom, "\cong"] \\
\Gamma\backslash\bigl(\Fboundary^{(2)} \times \LieA_\Theta\bigr) \arrow[r, phantom, "\supset"] & \Omega \arrow[d] \\
{} & \mathcal{X}
\end{tikzcd}
\right]
\end{array}
\]
\caption{}
\label{fig:DynamicalSystems}
\end{figure}

Let us give some remarks on the metrics on $\Omega$ and $\mathcal{X}$; see \cref{sec:AnosovProperty} for details. We first put a metric on $\Fboundary$ induced from a bi-invariant Riemannian metric on a maximal compact subgroup of $G$ and then restrict it to $\limitset$. We may then put metrics on $\limitset^{(2)} \times \LieA_\Theta$ and $\limitset^{(2)} \times \R$ which are $\Gamma$-invariant and locally bi-Lipschitz equivalent to the product metrics. By $\Gamma$-invariance, they descend to metrics on $\Omega$ and $\mathcal{X}$.

\begin{remark}
\label{rem:Metrics}
We emphasize the following points (see \cref{rem:limitset homeomorphism,rem:MetricsMain}).
\begin{itemize}
\item Metrics defined in the above fashion are highly sensitive to the metric space used for the boundary coordinates.
\item In accordance with the above, a key point in our arguments is that it is advantageous to use the $\Theta$-limit set $\limitset$ rather than the Gromov boundary $\Gboundary$ to define the metric spaces $\Omega$ and $\mathcal{X}$.
\item Let $\Theta' \subset \Theta$ be a nonempty proper subset. Then $\Gamma$ is also a $\Theta'$-Anosov subgroup. With this viewpoint, we obtain metric spaces $\Omega'$ and $\mathcal{X}'$ as above, corresponding to $\Theta'$. They are homeomorphic to $\Omega$ and $\mathcal{X}$, respectively, but have crucial differences as metric spaces. Indeed, we have Lipschitz homeomorphisms $\Omega \to \Omega'$ and $\mathcal{X} \to \mathcal{X}'$, induced by the canonical projection $\Fboundary \to \mathcal{F}_{\Theta'}$, but their inverses need not be Lipschitz or even H\"older.
\end{itemize}
\end{remark}

So far in the literature on Anosov subgroups, translation flows (or similarly defined flows) have been studied more often as an auxiliary tool rather than a dynamical system of intrinsic interest. However, translation flows are natural dynamical systems which mimic (and in fact generalize) the geodesic flow for convex cocompact hyperbolic manifolds and they are H\"{o}lder conjugate to H\"{o}lder reparametrizations of the Gromov geodesic flow associated to $\Gamma$. In this vein, we prove that translation flows are \emph{exponentially} mixing for generic vectors $\sfv \in \interior\limitcone$. Indeed, we obtain a detailed characterization of the exponential mixing property for the whole family of translation flows.

\subsection{Exponential mixing}
The set of generic $\sfv \in \interior\limitcone$ for which we can prove exponential mixing of the translation flows can be described in terms of the $\Theta$-growth indicator $\growthindicator$. Denote
\begin{align*}
\ker\Theta := \bigcup_{\alpha \in \Theta} \ker\alpha.
\end{align*}
We say that $\sfv \in \interior\limitcone$ is \emph{$\growthindicator$-regular} if $\nabla\growthindicator(\sfv) \notin \ker\Theta$ and we define the \emph{exceptional cone}
\begin{align*}
\mathscr{E} := \{\sfv \in \interior\limitcone: \sfv \text{ is not } \growthindicator\text{-regular}\}.
\end{align*}
It is nonempty for the following reason. First, since $\growthindicator$ is strictly concave except along radial directions on $\interior\limitcone$ and vertically tangent on $\partial\limitcone$ (see \cref{thm:BasicProperties}\labelcref{itm:BasicProperties4}), we deduce that $\overline{\bigcup_{\sfv \in \interior\limitcone} \R_{> 0} \nabla\growthindicator(\sfv)}$ coincides with the dual cone $\limitcone^* \subset \LieA_\Theta$ of $\limitcone$. Then, $\limitcone \subset \interior\LieA_\Theta^+ \cup \{0\}$ (see \cref{thm:BasicProperties}\labelcref{itm:BasicProperties2}) and non-obtuseness of $\LieA_\Theta^+$ (see \cite[Chapter II, \S\,5, Proposition 2.48(e)]{Kna96}) implies that $\partial\LieA_\Theta^+ \subset \interior\limitcone^* \cup \{0\}$, whence nonemptyness follows. In fact, $\mathscr{E}$ is the image of $\ker\Theta$ restricted to $\interior\limitcone$ under a diffeomorphism on $\LieA_\Theta$ and bounds a connected open cone.
We remark that the translation flows are homothety equivariant, i.e., $a^{c \sfv}_t = a^{\sfv}_{ct}$ for all $t \in \R$ and $c > 0$. For $\alpha \in (0, 1]$, we denote by $C^{0, \alpha}(\mathcal{X})$ the space of real-valued $\alpha$-H\"{o}lder continuous functions on $\mathcal{X}$ equipped with the $\alpha$-H\"{o}lder norm $\|\cdot\|_{C^{0, \alpha}}$.

\begin{remark}
We view $G$ and $\Gamma$ as fixed groups and so here and throughout the paper, we view constants depending only on $G$ and $\Gamma$ as absolute.
\end{remark}

\begin{theorem}
\label{thm:ExponentialMixingOnX}
Let $\alpha \in (0,1]$. Let $\sfv \in \interior\limitcone - \mathscr{E}$. There exist
\begin{itemize}
\item $\eta_{\alpha, \sfv} = \frac{\alpha}{\alpha + 2} \cdot \growthindicator(\sfv) \cdot O\bigl(e^{-c\|\nabla\growthindicator(\sfv)\|}\bigr) > 0$ which is continuous in $\sfv$, for some absolute constant $c > 0$,
\item $C_\sfv$ (independent of $\alpha$) which is continuous and homothety-invariant in $\sfv$,
\end{itemize}
such that for all $\phi_1,\phi_2 \in C^{0, \alpha}(\mathcal{X})$, and $t > 0$, we have
\begin{align*}
\left|\int_{\mathcal{X}} \phi_1(a^{\sfv}_tx)\phi_2(x) \, d\BMSX^{\sfv}(x) - \BMSX^{\sfv}(\phi_1)  \BMSX^{\sfv}(\phi_2)\right| \leq C_{\sfv} e^{-\eta_{\alpha, \sfv}t} \|\phi_1\|_{C^{0, \alpha}} \|\phi_2\|_{C^{0, \alpha}}.
\end{align*}
\end{theorem}

\begin{remark}
\label{rem:GoodThetaTrick}
We make the following observations; see \cref{rem:Metrics,subsec:ReductionOfGradGrowthIndicatorInKernelToNotInKernel}.
\begin{enumerate}
\item For any compact subset $\mathcal{K} \subset \interior\limitcone - \mathscr{E}$, the exponential rate can be taken to be \emph{uniform} in $\sfv \in \mathcal{K}$.
\item If $\nabla\growthindicator(\sfv) \in \ker\Theta$, then we may exploit the fact that the linear form $\psi_{\sfv} = \langle \growthindicator(\sfv)^{-1}\nabla\growthindicator(\sfv), \cdot \rangle$ factors through an orthogonal projection map $\pi_{\LieA_{\Theta'}}|_{\LieA_\Theta}: \LieA_\Theta \to \LieA_{\Theta'}$ (defined in \cref{subsec:LieTheoreticPreliminaries}) for some nonempty proper subset $\Theta' \subset \Theta$:
\begin{align*}
\psi_\sfv &= \psi_{\sfv'} \circ \pi_{\LieA_{\Theta'}}|_{\LieA_\Theta}, & \psi_{\sfv'} &:= \langle \psi_{\Theta'}(\sfv')^{-1}\nabla\psi_{\Theta'}(\sfv'), \cdot \rangle,
\end{align*}
for some $\sfv' \in \interior\mathcal{L}_{\Theta'}$. We may choose $\Theta' \subset \Theta$ such that $\nabla\growthindicator(\sfv) \notin \ker\Theta'$ and hence also $\nabla\psi_{\Theta'}(\sfv') \notin \ker\Theta'$. Then, we can still deduce exponential mixing on the corresponding metric space $\mathcal{X}'$.
\item We warn the reader that although the above trick gives a dynamical system $\bigl(\mathcal{X}', m_{\mathcal{X}'}^{\sfv'}, \{a^{\sfv'}_t\}_{t \in \R}\bigr)$ which is exponentially mixing and we have the fact that $\bigl(\mathcal{X}, \BMSX^{\sfv}, \{a^{\sfv}_t\}_{t \in \R}\bigr)$ is conjugate to $\bigl(\mathcal{X}', m_{\mathcal{X}'}^{\sfv'}, \{a^{\sfv'}_t\}_{t \in \R}\bigr)$, it \emph{does not necessarily} imply that $\bigl(\mathcal{X}, \BMSX^{\sfv}, \{a^{\sfv}_t\}_{t \in \R}\bigr)$ is also exponentially mixing in the strict sense of the theorem. More precisely, the exponential error term depends on the regularity of the test functions and so we must be careful of the metric used; indeed, H\"older functions lift from $\mathcal{X}'$ to $\mathcal{X}$ but do not necessarily descend from $\mathcal{X}$ to $\mathcal{X}'$.

Nevertheless, it may be possible to prove that $\bigl(\mathcal{X}, \BMSX^{\sfv}, \{a^{\sfv}_t\}_{t \in \R}\bigr)$ is exponentially mixing in some cases by further investigating the $\Theta$-limit set $\limitset$ and $\Theta$-limit cone $\limitcone$ associated to $\Gamma$.
\end{enumerate}
\end{remark}

We obtain the following simple corollary when $\#\Theta = 1$ because in this case $\dim(\LieA_\Theta) = 1$ and so the family of translation flows collapses to a trivial one-dimensional family $\{\{a_{ct}\}_{t \in \R}\}_{c >0}$ coinciding with rescalings of the one-parameter diagonal flow.

\begin{theorem}
\label{thm:ExponentialMixingOnXSingletonTheta}
Suppose $\#\Theta = 1$. Let $\alpha \in (0,1]$. Then, there exist $\eta_\alpha > 0$ and $C > 0$ (independent of $\alpha$) such that for all $\phi_1,\phi_2 \in C^{0, \alpha}(\mathcal{X})$ and $t > 0$, we have
\begin{align*}
\left|\int_{\mathcal{X}} \phi_1(a_tx)\phi_2(x) \, d\BMSX(x) - \BMSX(\phi_1)  \BMSX(\phi_2)\right| \leq C e^{-\eta_\alpha t} \|\phi_1\|_{C^{0, \alpha}} \|\phi_2\|_{C^{0, \alpha}}.
\end{align*}
\end{theorem}

Note that the above theorem strengthens the local mixing theorem from \cref{eqn:ThetaLocalMixing} for $\#\Theta = 1$. It also generalizes the case that $G$ is of rank one in which case we get the geodesic flow for convex cocompact locally symmetric spaces and follows from the work of Stoyanov \cite{Sto11}; however, see \cref{pro:LNIC} or \cite{CS22} of the authors for the \emph{local non-integrability condition}, and the latter also for the generalization to the \emph{frame flow}.

The above theorem also follows from prior works by the following argument. For an irreducible projective Anosov subgroup $\Gamma < \PSL_n(\R)$ (in which case $\#\Theta = 1$), it is shown in \cite{DMS24} that the translation flow on $\mathcal{X}$ extends to a contact Axiom A flow on a smooth manifold $\mathcal{M}$ with $\mathcal{X} \subset \mathcal{M}$ as its basic set. They also prove a \emph{strong local non-integrability condition} and thereby deduce exponential mixing by the theorem of \cite{Sto11}. The general \cref{thm:ExponentialMixingOnXSingletonTheta} can be now be deduced via a Pl\"ucker representation; however, one needs to be careful of the compatibility of metrics and hence requires arguments as in the proof of \cref{thm:TranslationFlowConjugateToRhoFlow} in \cref{subsec:TranslationFlowIsMetricAnosovProof}.

Using the spectral bounds of transfer operators in \cref{thm:SpectralBoundOnTransferOperator}, we actually prove the more detailed theorem below from which the above theorems follow. Note that one can repeat a standard convolution argument as in \cite[Appendix]{KM96} using Lipschitz continuous bump functions to handle general $\alpha$-H\"older functions (see also \cite[Corollary 5.2]{MW12} and \cite[Theorem 3.1.4]{Sar19}). The behavior of the exponential decay rate $\eta_\sfv$ for $\sfv$ near the boundary of the limit cone $\partial\limitcone$ is expressed in terms of $\|\nabla\growthindicator(\sfv)\|$ and follows, by rescaling back, from \cref{thm:RescaledExponentialMixingOnXWithPR-Resonances} corresponding to any fixed neighborhood $\mathcal{N} \supset \ker\Theta$. Note that $\eta_\sfv$ is homogeneous of degree $1$ because so is $\growthindicator$. The complex numbers $\{\mu_{k, \sfv}\}_{k = 1}^{k_\sfv}$ which appear are the Pollicott--Ruelle resonances for the translation flow associated to $\sfv \in\interior\limitcone$. The concept of Pollicott--Ruelle resonances originated in \cite{Pol85,Pol86,Rue86,Rue87}. Here, $L(\mathcal{X})$ denotes the space of real-valued Lipschitz continuous functions on $\mathcal{X}$ equipped with the Lipschitz norm $\|\cdot\|_{\Lip}$.

\begin{theorem}
\label{thm:ExponentialMixingOnXWithPR-Resonances}
Let $\sfv \in \interior\limitcone - \mathscr{E}$. There exist $k_\sfv \in \N$ and
\begin{itemize}
\item $\eta_\sfv = \frac{\growthindicator(\sfv)}{\|\nabla\growthindicator(\sfv)\|} \cdot O\bigl(e^{-c\|\nabla\growthindicator(\sfv)\|}\bigr) = \growthindicator(\sfv) \cdot O\bigl(e^{-c'\|\nabla\growthindicator(\sfv)\|}\bigr) > 0$ which is continuous in $\sfv$, for some absolute constants $c, c' > 0$,
\item $C_\sfv > 0$ (independent of $\alpha$) which is continuous and homothety-invariant in $\sfv$,
\item a finite set of complex numbers $\{\mu_{k, \sfv}\}_{k = 1}^{k_\sfv} \subset (-\eta_\sfv, 0) + i[-1, 1]$ which come in conjugate pairs and are homogeneous of degree $1$ in $\sfv$,
\item a finite set of finite-rank positive semi-definite bilinear forms $\{\mathcal{B}_{k, \sfv}\}_{k = 1}^{k_\sfv}$ which are homothety-invariant in $\sfv$,
\end{itemize}
such that for all $\phi_1,\phi_2 \in L(\mathcal{X})$ and $t > 0$, we have
\begin{multline*}
\left|\int_{\mathcal{X}} \phi_1(a^{\sfv}_tx)\phi_2(x) \, d\BMSX^{\sfv}(x) - \left(\BMSX^{\sfv}(\phi_1)  \BMSX^{\sfv}(\phi_2) + \sum_{k = 1}^{k_\sfv} e^{\mu_{k, \sfv} t} \mathcal{B}_{k, \sfv}(\phi_1, \phi_2)\right)\right| \\
\leq C_{\sfv} e^{-\eta_\sfv t} \|\phi_1\|_{\Lip} \|\phi_2\|_{\Lip}.
\end{multline*}
\end{theorem}

\begin{remark}
Of course, $\eta_\sfv$ depends first and foremost on $\Gamma$ and $\Theta$, and also on $\sfv \in \interior\limitcone - \mathscr{E}$. Restricting to unit vectors, $\eta_\sfv$ tends to $0$ as $\sfv$ tends to vectors in $\partial\limitcone$ (in which case the decay is exponential in $\|\nabla\growthindicator(\sfv)\|$) and possibly also to vectors in $\mathscr{E}$.
\end{remark}

\subsection{Essential spectral gap and effective prime orbit theorem}
\label{subsec:EssentialSpectralGapAndPrimeOrbitTheorem}
Again using the powerful \cref{thm:SpectralBoundOnTransferOperator}, we also obtain the following applications.

For all $\sfv \in \interior\limitcone$ and $T > 0$, define
\begin{align*}
\mathcal{P}_\sfv :={}&\{\gamma: \text{$\gamma$ is a primitive closed $\{a^\sfv_t\}_{t \in \R}$-orbit in $\mathcal{X}$}\}, \\
\mathcal{P}_\sfv(T) :={}&\{\gamma \in \mathcal{P}_\sfv: \ell_\sfv(\gamma) \leq T\},
\end{align*}
where $\ell_\sfv(\gamma)$ denotes the primitive period of $\gamma$. We may drop the subscript for the former and write $\mathcal{P}$ since the set of closed orbits coincide for all translation flows. Moreover, there is a canonical one-to-one correspondence
\begin{align*}
\mathcal{P} \leftrightarrow [\Gamma_{\mathrm{prim}}],
\end{align*}
where the latter is the set of conjugacy classes of $\Gamma$ corresponding to the subset of primitive loxodromic elements $\Gamma_{\mathrm{prim}} \subset \Gamma$. Note that in our setting, all nontorsion elements of $\Gamma$ are loxodromic (see \cref{thm:BasicProperties}). Under the above one-to-one correspondence, we also have
\begin{align*}
\mathcal{P}_\sfv(T) \leftrightarrow \{[\gamma] \in [\Gamma_{\mathrm{prim}}]: \psi_\sfv(\lambda_\Theta(\gamma)) \leq T\}.
\end{align*}
Recall that $\psi_{\sfv} = \langle \delta_\sfv^{-1}\nabla\growthindicator(\sfv), \cdot \rangle$ where $\delta_\sfv = \growthindicator(\sfv)$ is the topological entropy of the translation flow $\{a^\sfv_t\}_{t \in \R}$.

For all $\sfv \in \interior\limitcone$, recall the Selberg zeta function $Z_\sfv: (\delta_\sfv, +\infty) + i\R \to \C$ \cite{Sel56} defined by
\begin{align*}
Z_\sfv(\xi) = \prod_{k = 0}^\infty \prod_{\gamma \in \mathcal{P}} \bigl(1 - e^{-(\xi + k)\ell_\sfv(\gamma)}\bigr) \qquad \text{for all $\Re(\xi) > \delta_\sfv$}
\end{align*}
which converges absolutely using the noneffective prime orbit theorem in \cite[Eq. (1.8)]{CF23} (cf. \cite[Chapter 2, Definition 4.1]{Hej76} and the subsequent remark). By the work of Pollicott \cite{Pol86}, it then extends meromorphically to the domain $(\delta_\sfv - \eta_\sfv, +\infty) + i\R$ for some $\eta_\sfv > 0$ with a simple zero at $\delta_\sfv$. He also characterizes $\eta_\sfv$ in terms of the pressure functional, the expansion/contraction rate from the metric Anosov property, and the topological entropy.

Let us introduce a convenient terminology. For any nonempty subset $\Theta' \subset \Theta$, denote $(\interior\limitcone)_{\Theta'} := \{\sfv \in \interior\limitcone: \nabla\growthindicator(\sfv) \in \LieA_{\Theta'} - \ker\Theta'\}$. Then, we have the stratification $\mathscr{E} = \bigsqcup_{\varnothing \neq \Theta' \subsetneq \Theta} (\interior\limitcone)_{\Theta'}$. We say that a real-valued function $\eta$ on $\limitcone$ or $\interior\limitcone$ is \emph{continuous according to strata} if for all nonempty subsets $\Theta' \subset \Theta$, the restriction $\eta|_{(\interior\limitcone)_{\Theta'}}$ is continuous except possibly with points of discontinuities in $\bigsqcup_{\varnothing \neq \Theta'' \subsetneq \Theta'}(\interior\limitcone)_{\Theta''}$.

\begin{theorem}
\label{thm:EssentialSpectralGap}
Let $\sfv \in \interior\limitcone$. There exists $\eta_\sfv = \growthindicator(\sfv) \cdot O\bigl(e^{-c\|\nabla\growthindicator(\sfv)\|}\bigr) > 0$ which is continuous according to strata in $\sfv$, for some absolute constant $c > 0$, such that the Selberg zeta function $Z_\sfv$ has a simple zero at $\delta_\sfv$ and only finitely many other zeros on the half-plane $(\delta_\sfv - \eta_\sfv, +\infty) + i\R$.
\end{theorem}

The above theorem has the following equivalent formulation. For all $\sfv \in \interior\limitcone$, recall the Ruelle zeta function $\zeta_\sfv: (\delta_\sfv, +\infty) + i\R \to \C$ defined by
\begin{align*}
\zeta_\sfv(\xi) = \frac{Z_\sfv(\xi + 1)}{Z_\sfv(\xi)} = \prod_{\gamma \in \mathcal{P}} \bigl(1 - e^{-\xi\ell_\sfv(\gamma)}\bigr)^{-1} \qquad \text{for all $\Re(\xi) > \delta_\sfv$}
\end{align*}
which then extends meromorphically as above.

\begin{theorem}
\label{thm:EssentialSpectralGapRuelle}
Let $\sfv \in \interior\limitcone$. There exists $\eta_\sfv = \growthindicator(\sfv) \cdot O\bigl(e^{-c\|\nabla\growthindicator(\sfv)\|}\bigr) > 0$ which is continuous according to strata in $\sfv$, for some absolute constant $c > 0$, such that the Ruelle zeta function $\zeta_\sfv$ has a simple pole at $\delta_\sfv$ and only finitely many other poles on $(\delta_\sfv - \eta_\sfv, +\infty) + i\R$.
\end{theorem}

We say that the dynamical system $\bigl(\mathcal{X}, \BMSX^{\sfv}, \{a^{\sfv}_t\}_{t \in \R}\bigr)$ has an \emph{essential spectral gap} of size $\eta_\sfv > 0$. We emphasize that $\kappa_\Theta$ is an explicit function which dictates the behavior of the decay of the essential spectral gap near the boundary of the limit cone $\partial\limitcone$ whereas we have no information on the zeros of the Selberg zeta function $Z_\sfv$.

We have the following more refined version of the prime orbit theorem with a power saving error term stated in \cref{thm:ExponentialPrimeOrbitTheoremSimple}.

\begin{theorem}
\label{thm:ExponentialPrimeOrbitTheorem}
Let $\sfv \in \interior\limitcone$. There exist
\begin{itemize}
\item $\eta_\sfv = \growthindicator(\sfv) \cdot O\bigl(e^{-c\|\nabla\growthindicator(\sfv)\|}\bigr) > 0$ which is continuous according to strata in $\sfv$, for some absolute constant $c > 0$,
\item $C_\sfv$ which is continuous according to strata and homothety-invariant in $\sfv$,
\end{itemize}
such that for all $T > 0$, we have
\begin{align*}
\bigl|\#\mathcal{P}_\sfv(T) - \Li\bigl(e^{\delta_\sfv T}\bigr)\bigr| \leq C_\sfv e^{(\delta_\sfv - \eta_\sfv)T}.
\end{align*}
\end{theorem}

The above theorem effectivizes the prime orbit theorem of Sambarino \cite[Theorem 7.8]{Sam14b} and its generalization by Chow--Fromm \cite[Eq. (1.8)]{CF23} (cf. \cite[Corollary 1.4]{CF23}) where the authors obtain the cruder main term $\frac{e^{\delta_\sfv T}}{\delta_\sfv T}$. In fact, another plausible approach to prove \cref{thm:ExponentialPrimeOrbitTheorem} is by effectivizing the work of \cite{CF23} with \cref{thm:ExponentialMixingOnX} as input.

Prior effective results in the Anosov setting are those of Pollicott--Sharp \cite{PS24} and Delarue--Monclair--Sanders \cite{DMS24}. In \cite{PS24}, the authors study Hitchin surface subgroups $\Gamma < \PSL_n(\R)$, i.e., the image of a Hitchin representation of a surface group into $\PSL_n(\R)$, viewed as a projective Anosov subgroup (in which case $\#\Theta = 1$). They obtain a meromorphic extension of the Selberg zeta function to the entire complex plane, and for a more restricted subclass of Hitchin surface subgroups, they obtain an essential spectral gap and prime orbit theorem with a power saving error term. In \cite{DMS24}, the authors generalize such results to the general irreducible projective Anosov setting; however, they do not need irreducibility for the result on meromorphic extension to the entire complex plane. Our theorems hold for the general Zariski dense $\Theta$-Anosov setting; however, we do not address the question of meromorphic extension to the entire complex plane.

\begin{remark}
In the above two theorems, we are able to extract positivity of the essential spectral gap and the power saving exponent in the prime orbit theorem even if $\nabla\growthindicator(\sfv) \in \ker\Theta$, by using the trick from \cref{rem:GoodThetaTrick}---that is, the above two theorems only require the existence of \emph{some} associated dynamical system with appropriate periods which is exponentially mixing, and they are not sensitive to the metric on the dynamical system. However, we are not able to gain finer control, i.e., continuity of the of the essential spectral gap function and the power saving exponent function at points in a lower dimensional stratum. We think the below related questions are interesting.
\end{remark}

\begin{questions}
We raise the following questions:
\begin{enumerate}
\item Are the optimal essential spectral gap for the Selberg zeta function and the optimal power saving exponent in the prime orbit theorem continuous in $\sfv \in \interior\limitcone$ (not just according to strata)?
\item Are the above constants bounded below by a positive constant uniformly in $\sfv \in \interior\limitcone$?
\item How does the behavior depend on the geometry of $\Gamma$, i.e., on the $\Theta$-limit set and $\Theta$-limit cone?
\item How does the behavior vary with respect to deformations of $\Gamma$? (cf. \cite{DO24})
\end{enumerate}
\end{questions}

We think that for the first two questions, the \emph{generic} behavior is that the answer is negative.

\subsection{On the proof of \texorpdfstring{\cref{thm:ExponentialMixingOnX}}{\autoref{thm:ExponentialMixingOnX}}}
\label{subsec:OnTheProofOfMainTheorem}
For general $\Theta$-Anosov subgroups, we have a family of translation flows $\{\{a^\sfv_t\}_{t \in \R}\}_{\sfv \in \interior\limitcone}$ on $\mathcal{X}$ and we begin by showing that they are metric Anosov. We note that when $\Theta \ne \Pi$, the space $\limitset^{(2)} \times \R$ does not naturally inherit a $\Gamma$-invariant metric from the Riemannian metric on $G$ and so the metric on $\mathcal{X}$ must be chosen and constructed carefully. Moreover, we also show that the strong (un)stable laminations in $\mathcal{X}$ of a translation flow are projections of the horospherical foliations of $G$ (\cref{thm:TranslationFlowIsMetricAnosov}). Our proofs of these properties, following \cite{CS23}, use the fact that one can always construct a projective Anosov representation of $\Gamma$ using the Pl\"ucker representation and then use the results of Bridgeman--Canary--Labourie--Sambarino \cite{BCLS15} on the metric Anosov property in this case. It is likely that the arguments in \cite{BCLS15} can be extended to general Anosov subgroups to give more direct proofs of these properties but we do not pursue this. An immediate consequence of the metric Anosov property is that there exist Markov sections for the translation flows which are compatible with the corresponding strong (un)stable laminations. The constructed framework is an important tool which do not appear in previous works and for which reason our approach is fruitful in the general $\Theta$-Anosov setting.

Thanks to the Markov sections, we may freely introduce symbolic dynamics and thermodynamic formalism. In particular, we define transfer operators $\mathcal{L}_{\xi\tau^\sfv}: C\bigl(U, \C\bigr) \to C\bigl(U, \C\bigr)$ defined by
\begin{align*}
\mathcal{L}_{\xi\tau^\sfv}(H)(u) = \sum_{u' \in \sigma^{-1}(u)} e^{\xi\tau^\sfv(u')}H(u')
\end{align*}
where $\tau^\sfv: U \to \R$ is the first return time map associated to the translation flow $\{a^\sfv_t\}_{t \in \R}$ on $\mathcal{X}$. We wish to then execute Dolgopyat's method to obtain strong spectral bounds for the transfer operators. It is well-known that such spectral bounds can be used to derive:
\begin{enumerate}
\item exponential mixing;
\item essential spectral gap for the Selberg zeta function;
\item prime orbits theorem with a power saving error term.
\end{enumerate}
Building on the works of Dolgopyat \cite{Dol98}, Avila--Gou\"ezel--Yoccoz \cite{AGY06}, and Liverani \cite{Liv04}, Stoyanov \cite{Sto11} implemented Dolgopyat's method for Axiom A flows assuming a crucial \emph{local non-integrability condition (LNIC)}. However, in our setting, the translation flow on $\mathcal{X}$ does not \emph{a priori} extend to an obvious Axiom A flow and hence we opt to work with weaker properties; however, see the comments below on \cite{DMS24}. In any case, we need to prove LNIC (\cref{pro:LNIC}). We wish to use Lie theoretic techniques for this. However, we need to carefully deal with some topological issues due to the fractal nature of the limit set $\limitset$ and the $\Gamma$-action not being properly discontinuous on $\Fboundary^{(2)} \times \R$ in general. Following \cite{SW21}, the relationship between the strong (un)stable laminations of the translation flow and the horospherical foliations allows us to lift and extend the rectangles of the Markov section to open sets and use auxiliary smooth extensions of the Poincar\'{e} map and the first return time map. This is a rather crude approach to use the smooth structure but it is sufficient for the purposes of generalizing the Lie theoretic techniques in \cite{SW21,CS23} to prove a reverse Lipschitz version of LNIC on the original Markov section which is of fractal nature. In the process, it becomes clear that the aforementioned derivation holds for all $\sfv$ but those in a certain exceptional set $\mathscr{E} \subset \interior\limitcone$. Fortunately, Dolgopyat's method is robust enough that the above ingredients suffice to be able to run the required arguments (for all $\sfv \in \interior\limitcone - \mathscr{E}$).

Throughout the paper, we also take a unified viewpoint and carefully investigate the \emph{dependence} of the family of translation flows on the parameter $\sfv \in \interior\limitcone$. This requires additional important properties. The first one is the existence of a \emph{compatible family} of Markov sections for all $\sfv \in \interior\limitcone$. This allows us to identify the corresponding rectangles with each other and therefore use the same domain $U$ to define the transfer operators and only vary the first return time maps $\tau^\sfv$ which is smooth in $\sfv \in \interior\limitcone$. We also show some bounds for $\tau^\sfv$ which is used to carry through the dependence on $\sfv$ throughout the proofs.

Let us briefly compare our approach above with that of Pollicott--Sharp \cite{PS24} and Delarue--Monclair--Sanders \cite{DMS24}. In \cite{PS24}, the authors also use transfer operators but for a Bowen--Series type symbolic coding rather than a symbolic coding from a Markov section of a flow. In \cite{DMS24}, one of the main result is that indeed \emph{there exists} an extension of the translation flow on $\mathcal{X}$ to a contact Axiom A flow on a smooth manifold $\mathcal{M} := \Gamma\backslash \tilde{\mathcal{M}}$ with $\mathcal{X} \subset \mathcal{M}$ as its basic set \cite[Theorem A]{DMS24}. For this objective, the $\Gamma$-action not being properly discontinuous on $\Fboundary^{(2)} \times \R$ in general is a fundamental obstruction in \cite{DMS24} as well. Nevertheless, the authors prove that the $\Gamma$-action is properly discontinuous on a translation flow-invariant open subset $\tilde{\mathcal{M}} \subset \Fboundary^{(2)} \times \R$ containing $\limitset^{(2)} \times \R$.

\subsection{Organization of the paper}
\Cref{sec:AnosovSubgroupsAndTheirTranslationFlows} covers background on semisimple real algebraic groups, Anosov subgroups, and translation flows. In \cref{sec:AnosovProperty}, we prove that the translation flow is metric Anosov for a suitable metric and describe the laminations in terms of the $\Theta$-horospherical subgroups to establish Markov sections compatible with the Lie structure. In \cref{sec:TransferOperators}, we make some reductions and introduce transfer operators. \Cref{sec:DolgopyatsMethodAndProofs} contains the formulation of a key theorem for a Lipschitz version of Dolgopyat's method. We then outline the derivations for exponential mixing for Lipschitz continuous functions, essential spectral gap for the Selberg zeta function, and prime orbit theorem with power saving error term. In \cref{sec:LNIC} we prove the crucial LNIC. Finally, in \cref{sec:ProofOfTheorem}, we prove the aforementioned key theorem.

\subsection*{Acknowledgements}
We thank our advisor Hee Oh for her encouragements. We thank Rafael Potrie for useful correspondence regarding metric Anosov flows. We also thank Institut des Hautes \'{E}tudes Scientifiques (IH\'{E}S) and the Fields Institute for Research in Mathematical Sciences for their hospitality and facilitating conversations. P.S. acknowledges support by an AMS-Simons Travel Grant. The paper was also substantially revised while P.S. was visiting the Mathematical Sciences Research Institute (MSRI), now becoming the Simons Laufer Mathematical Sciences Institute (SLMath), which is supported by the National Science Foundation (Grant No. DMS-2424139).

\section{Anosov subgroups and their translation flows}
\label{sec:AnosovSubgroupsAndTheirTranslationFlows}

\subsection{Lie theoretic preliminaries}
\label{subsec:LieTheoreticPreliminaries}
Let $G$ be a connected semisimple real algebraic group with identity element $e$ and Lie algebra $\LieG$. Let $B:\LieG\times\LieG \to \R$ denote the Killing form and $\theta: \LieG \to \LieG$ be a Cartan involution, i.e., the symmetric bilinear form $B_\theta:\LieG\times\LieG\to \R$ defined by $B_\theta(x,y)=-B(x,\theta(y))$ for all $x,y\in \LieG$ is positive definite. We write $\langle *_1, *_2\rangle = B_\theta(*_1, *_2)$ for the inner product and $\|\cdot\|$ for the induced norm on $\LieG$. Let $\LieG = \LieK \oplus \LieP$ be the associated eigenspace decomposition corresponding to the eigenvalues $+1$ and $-1$ of $\theta$ respectively. Then $K = \exp(\LieK) < G$ is a maximal compact subgroup. Let $\LieA \subset \LieP$ be a maximal abelian subalgebra and let $A := \exp\LieA$ the corresponding maximal real split torus. Identifying $\LieA \cong \LieA^*$ via the Killing form, let $\Phi \subset \LieA^*$ be the restricted root system of $(\LieG, \LieA)$ and $\Phi^\pm \subset \Phi$ be a choice of sets of positive and negative roots. Let $\LieA^+ \subset \LieA$ be the corresponding closed positive Weyl chamber. We have the associated restricted root space decomposition
\begin{align*}
\LieG = \LieA \oplus \LieM \oplus \bigoplus_{\alpha \in \Phi} \LieG_\alpha,
\end{align*}
where $\LieM$ is the centralizer of $\LieA$ in $\LieK$.

Let $\Pi \subset \Phi^+$ denote the set of all simple roots. We denote by $\langle\Theta\rangle$ the root subsystem generated by any subset $\Theta \subset \Pi$ (it is empty if so is $\Theta$). Fix a nonempty subset $\Theta \subset \Pi$. We will introduce several objects associated to $\Theta$; we omit the subscript/superscript $\Pi$ when $\Theta=\Pi$. Define Lie subalgebras $\LieA_\Theta$ and $\LieN_\Theta^\pm$ of $\LieG$ by 
\begin{align*}
& \LieA_\Theta := \bigcap_{\alpha \in \Pi -\Theta}\ker\alpha, & \LieN^\pm_\Theta := \bigoplus_{\alpha \in \Phi^+ - \langle \Pi - \Theta\rangle} \LieG_{\mp\alpha}
\end{align*}
where for the former, we take the empty intersection to be $\LieA$. Denote $\LieA^+_\Theta :=\LieA_\Theta \cap \LieA^+$ and by $\interior\LieA^+_\Theta$ the interior of $\LieA^+_\Theta$ in the topology of $\LieA_\Theta$. Define
\begin{align*}
A_\Theta &= \exp(\LieA_\Theta), & A^+_\Theta &= \exp(\LieA^+_\Theta) \subset A_\Theta.
\end{align*}
Then, $G$ has the Cartan decomposition $G = KA^+K$. Let $\mathcal{W} = N_K(A)/Z_K(A)$ be the Weyl group which is a finite group. The adjoint action $\Ad$ induces an action $\mathcal{W} \curvearrowright \LieA$ which we also denote by $\Ad$. Let $\pi_{\LieA_\Theta}: \LieG \to \LieA_\Theta$ denote the orthogonal projection map. Then the restriction $\pi_{\LieA_\Theta}|_{\LieA}: \LieA \to \LieA_\Theta$ is in fact the unique projection map invariant under precomposition by all elements of the Weyl group $\mathcal{W}$ which act trivially on $\LieA_\Theta$. Note that in general $\ker\pi_{\LieA_\Theta}|_{\LieA} \neq \LieA_{\Pi - \Theta}$. Let $w_0 \in \mathcal{W}$ be the longest element with respect to the generating set consisting of root reflections. Then, $\Ad_{w_0}(\LieA^+) = -\LieA^+$ and the map $\involution=-\Ad_{w_0}:\LieA \to \LieA$ is called the \emph{opposition involution} of $\LieA$. The opposition involution preserves $\LieA^+$ and acts on $\Pi$ by precomposition. 

For a fixed $v \in \interior \LieA_\Theta^+$, the expanding/contracting \emph{$\Theta$-horospherical subgroups} are 
\begin{align*}
N^\pm_\Theta &:= \exp(\LieN^\pm_\Theta) = \left\{h \in G: \lim_{t \to \pm\infty} \exp(tv) h \exp(-tv) = e\right\}.
\end{align*}
The definition is independent of the choice of $v \in \interior \LieA_\Theta^+$ and $N_\Theta^\pm = w_0 N_{\involution\Theta}^\mp w_0^{-1}$. Let $P_\Theta$ denote the standard parabolic subgroup associated to $\Theta$, i.e., $P_\Theta$ is the normalizer of $N^-_\Theta$ in $G$. Its Levi subgroup $L_\Theta < P_\Theta$ is the centralizer of $A_\Theta$ in $G$. It satisfies
\begin{align*}
P_\Theta &= L_\Theta N_\Theta^-, & w_0P_{\involution\Theta}w_0^{-1} &= L_\Theta N_\Theta^+, & L_\Theta &= P_\Theta \cap w_0P_{\involution\Theta}w_0^{-1}.
\end{align*}
Also $L_\Theta = A_\Theta S_\Theta \cong A_\Theta \times S_\Theta$ where $S_\Theta$ is an almost direct product of a connected semisimple real algebraic group and a compact center. Moreover, $S_\Theta$ has a Cartan decomposition $S_\Theta = M_\Theta A_{\Pi - \Theta}^+ M_\Theta$. Note that the Lie algebra of $P_\Theta$ is given by $\LieP_\Theta = \LieA \oplus \LieM \oplus \bigoplus_{\alpha \in \Phi^+ \cup \langle \Pi - \Theta \rangle} \LieG_\alpha$.  We denote the \emph{$\Theta$-Furstenberg boundary} of $G$ by
\begin{align*}
\Fboundary = G/P_\Theta \cong K/M_\Theta.
\end{align*}
Note that we also have an action $\mathcal{W} \curvearrowright \Fboundary$ induced by the left translation action. For all $g \in G$, we denote 
\begin{align*}
g^+ &:= g[P_\Theta] \in \Fboundary, & g^- &:= gw_0[P_{\involution\Theta}] \in \iFboundary.
\end{align*}
There is a unique open $G$-orbit of $\Fboundary \times \iFboundary$ given by
\begin{align*}
\Fboundary^{(2)} := G\cdot(e^+,e^-) = G\cdot([P_\Theta], w_0[P_{\involution\Theta}]).
\end{align*}
Similarly, there is a unique $G$-orbit of \emph{singular pairs} (see \cite[Definition 4.8]{GW12}) given by
\begin{align*}
\mathcal{S}_\Theta := G\cdot([P_\Theta],[P_{\involution\Theta}]).
\end{align*}

\begin{definition}[Iwasawa cocycle]
The \emph{Iwasawa cocycle} is the map $\sigma: G \times \mathcal{F} \rightarrow \LieA$ which assigns to a pair $(g, \xi) \in G \times \mathcal{F}$ the unique element $\sigma(g,\xi) \in \LieA$ such that $gk \in K\exp (\sigma(g,\xi))N^-$, and satisfies the cocycle relation $\sigma(gh,\xi) = \sigma(g,h\xi) + \sigma(h,\xi)$, for all $g, h \in G$ and $\xi = kP \in \mathcal{F}$.
\end{definition}

\begin{definition}[$\Theta$-Busemann function]
The \emph{$\Theta$-Busemann function} is the map $\beta^\Theta: \Fboundary \times G \times G \to \LieA_\Theta$ defined by
\begin{align*}
\beta^\Theta_\xi(g, h) := \pi_{\LieA_\Theta}(\sigma(g^{-1}, \tilde{\xi}) - \sigma(h^{-1}, \tilde{\xi}))
\end{align*}
for all $\xi \in \Fboundary$ and $g, h \in G$, where $\tilde{\xi} \in \mathcal{F}$ is any lift of $\xi$ under the canonical projection $\mathcal{F} \to \Fboundary$. This is well-defined, i.e., independent of the choice of $\tilde{\xi}$ \cite[Lemma 6.1]{Qui02a}, left $G$-equivariant, right $K$-invariant, and satisfies the identity
\begin{align*}
\beta^\Theta_\xi(g_1, g_2) = \beta^\Theta_\xi(g_1, h) + \beta^\Theta_\xi(h, g_2)
\end{align*}
for all $\xi \in \Fboundary$ and $g_1, g_2, h \in G$. 
\end{definition}

Define a left $G$-action on $\Fboundary^{(2)} \times \LieA_\Theta$ by
\begin{align*}
g(\xi,\eta,v) = \bigl(g\xi,g\eta,v + \beta^\Theta_{g\xi}(e, g)\bigr) = (g\xi,g\eta,v + \pi_{\LieA_\Theta}(\sigma(g,\xi)))
\end{align*}
for all $g \in G$ and $(\xi, \eta, v) \in \Fboundary^{(2)} \times \LieA_\Theta$. Then the stabilizer of $(e^+,e^-,0)$ is $S_\Theta$ and we have the following diffeomorphism.

\begin{definition}[$\Theta$-Hopf parameterization]
\label{def:HopfParametrization}
The \emph{$\Theta$-Hopf parameterization} is the left $G$-equivariant diffeomorphism $G/S_\Theta \to \Fboundary^{(2)} \times \LieA_\Theta$ defined by
\begin{align*}
gS_\Theta \mapsto \bigl(g^+, g^-, \beta^\Theta_{g^+}(e, g)\bigr).
\end{align*}
We note that the right $A_\Theta$-action on $G/S_\Theta$ corresponds to the $\LieA_\Theta$-translation on the $\LieA_\Theta$-coordinate of $\Fboundary^{(2)} \times \LieA_\Theta$. 
\end{definition}

We introduce the $\Theta$-Gromov product as defined in \cite[\S\,2]{KOW23a} (see also \cite[\S\,4]{Sam15}).

\begin{definition}[$\Theta$-Gromov product]
For $(\xi, \eta) = (g^+, g^-) \in \Fboundary^{(2)}$, the $\Theta$-Gromov product is defined as
\begin{align*}
\mathcal{G}^\Theta[\xi, \eta] := \beta^\Theta_{g^+}(e, g) + \involution\beta^{\involution\Theta}_{g^-}(e, g)
\end{align*} 
which is independent of choice of $g$ \cite[Lemma 9.13]{KOW23b}.
\end{definition}

\subsection{Zariski dense \texorpdfstring{$\Theta$}{Θ}-discrete subgroups}
\label{subsec:PreliminariesOnDiscreteSubgroups}

Let $\Gamma < G$ be a Zariski dense discrete subgroup.

\begin{definition}[$\Theta$-limit set]
	The \emph{$\Theta$-limit set} $\limitset \subset \Fboundary$ of $\Gamma$ is defined by
	\begin{align*}
		\limitset = \{\xi \in \Fboundary : \exists \{\gamma_n\}_{n \in \N} \subset \Gamma, (\gamma_n)_*\nu \to \delta_\xi\}
	\end{align*}
	where $\nu$ denotes the unique $K$-invariant probability measure on $\Fboundary$ and $\delta_\xi$ denotes the Dirac measure at $\xi$. It is the unique minimal nonempty closed $\Gamma$-invariant subset of $\Fboundary$ \cite{Ben97}.
\end{definition}

See also \cite[\S\,7]{Qui02a} and \cite[\S\,5]{KOW23b} for equivalent definitions and further details.

The Jordan projection $\lambda: G \to \LieA^+$ is a map that assigns to each element $g \in G$ the unique element $\lambda(g) \in \LieA^+$ such that the hyperbolic component of the Jordan decomposition of $g$ is conjugate to $\exp(\lambda(g))$. We denote $\lambda_\Theta := \pi_{\LieA_\Theta} \circ \lambda$ and call it the \emph{$\Theta$-Jordan projection}.

\begin{definition}[$\Theta$-limit cone]
\label{def:ThetaLimitCone}
The \emph{$\Theta$-limit cone} $\limitcone \subset \LieA_\Theta^+$ of $\Gamma$ is the unique minimal closed cone in $\LieA_\Theta^+$ containing $\lambda_\Theta(\Gamma)$. It is a convex cone with nonempty interior \cite{Ben97}.
\end{definition}

The Cartan projection $\mu: G \to \LieA^+$ is a map that associates to each element $g \in G$ the unique element $\mu(g) \in \LieA^+$ such that $g \in Ka_{\mu(g)}K$. We denote $\mu_\Theta := \pi_{\LieA_\Theta} \circ \mu$. We say $\Gamma$ is \emph{$\Theta$-discrete} if $\mu_\Theta|_{\Gamma}$ is a proper map \cite[Definition 3.1]{KOW23b}. We now further assume that $\Gamma$ is $\Theta$-discrete and introduce the $\Theta$-growth indicator of $\Gamma$. Quint introduced and studied the growth indicator when $\Theta = \Pi$ and his work was generalized by \cite[\S\,3]{KOW23b}.

\begin{definition}[$\Theta$-growth indicator]
\label{def:ThetaGrowthIndicator}
The \emph{growth indicator} $\growthindicator : \LieA^+_\Theta \to \R \cup \{-\infty\}$ of $\Gamma$ is defined by
\begin{align*}
\growthindicator(v) = \|v\|\inf_{\substack{\text{open cones }\mathcal{C}\subset \LieA_\Theta^+ \\ v \in \mathcal{C}}}\tau_{\mathcal{C}}^\Theta \qquad \text{ for all }v \in \LieA_\Theta
\end{align*}
with the convention that $0 \cdot (-\infty) = 0$, where $\|\cdot\|$ is any Euclidean norm on $\LieA_\Theta$ and $\tau_{\mathcal{C}}^\Theta$ is the abscissa of convergence of the Poincar\'{e} series 
\begin{align*}
\sum_{\gamma \in \Gamma, \, \mu_\Theta(\gamma) \in \mathcal{C}} e^{-t\|\mu_\Theta(\gamma)\|}.
\end{align*} 
This definition does not depend on the choice of norm $\|\cdot\|$ (though we have fixed the one induced by the Killing form). The growth indicator $\growthindicator$ is homogeneous of degree $1$, concave, upper semicontinuous, and satisfies $\growthindicator|_{\exterior(\limitcone)} = -\infty$, $\growthindicator|_{\limitcone} \ge 0$, and $\growthindicator|_{\interior(\limitcone)} > 0$. Denote by $\mathsf{u}_\Theta \in \interior\limitcone$ any maximal growth direction of $\growthindicator$.
\end{definition}

The following generalization of Patterson--Sullivan measures \cite{Pat76,Sul79} is also due to Quint \cite{Qui02a} (see also \cite{Alb99}).

\begin{definition}[$(\Gamma,\psi)$-conformal measure]
Let $\psi \in \LieA_\Theta^*$. A Borel probability measure $\nu$ on $\Fboundary$ is called a \emph{$(\Gamma,\psi)$-conformal measure} if
\begin{align*}
\frac{d\gamma_*\nu}{d\nu}(\xi) = e^{\psi(\beta^\Theta_\xi(e, \gamma))} \qquad \text{ for all } \gamma \in \Gamma, \, \xi \in \Fboundary.
\end{align*}
\end{definition}

We say that $\psi \in \LieA_\Theta^*$ is \emph{tangent to $\growthindicator$ at $v \in \limitcone$} if $\psi \geq \growthindicator$ and $\psi(v) = \growthindicator(v)$ \cite[Definition 3.10]{KOW23b}. We also denote the dual limit cone
\begin{align}
\label{eqn:DualLimitCone}
\limitcone^* := \{\psi \in \LieA_\Theta^* : \psi|_{\limitcone} \ge 0\}.
\end{align}

The following is a generalization of \cite[Theorem 8.4]{Qui02b}, which is for $\Theta = \Pi$, to arbitrary $\Theta$.

\begin{theorem}[{\citealp[Proposition 5.9]{KOW23b}}]
If $\psi \in \LieA_\Theta^*$ is tangent to $\growthindicator$ at $v \in \interior \LieA_\Theta^+ \cap \limitcone$, then there exists a $(\Gamma,\psi)$-conformal measure.
\end{theorem}

\subsection{Anosov subgroups}
\label{subsec:AnosovSubgroups}
The notation introduced in the remainder of this section will be fixed throughout the paper. The following definition of Anosov subgroups in the Zariski dense case is appropriate for our purposes and is due to Guichard--Wienhard \cite[Theorem 4.11]{GW12} (see also \cite{Lab06,GGKW17,KLP17,BPS19} for other equivalent characterizations). For a (finitely generated) Gromov hyperbolic group $\Gamma$, we denote by $\Gboundary$ its Gromov boundary. We also denote
\begin{align*}
\Gboundary^{(2)} := \{(x, y) \in \Gboundary^2: x \neq y\}.
\end{align*}

\begin{definition}[$\Theta$-Anosov subgroup]
\label{def:Theta-AnosovSubgroup}
We call a Zariski dense discrete subgroup $\Gamma < G$ a \emph{($\Theta$-)Anosov subgroup} if it is a Gromov hyperbolic group and admits continuous $\Gamma$-equivariant \emph{boundary maps}
\begin{align*}
\zeta_\Theta&: \Gboundary \to \Fboundary, & \zeta_{\involution\Theta}&: \Gboundary \to \iFboundary
\end{align*}
such that
\begin{enumerate}
\item $(\zeta_\Theta(x), \zeta_{\involution\Theta}(x)) \in \mathcal{S}_\Theta$ for all $x \in \Gboundary$;
\item $(\zeta_\Theta(x), \zeta_{\involution\Theta}(y)) \in \Fboundary^{(2)}$ for all $(x, y) \in \Gboundary^{(2)}$.
\end{enumerate}
\end{definition}

By \cite[Lemma 3.18]{GW12}, if $\Gamma$ is $\Theta$-Anosov, then it is also $\Theta'$-Anosov for any other nonempty subset $\Theta' \subset \Theta$ (note that the notation $\Theta$ in loc. cit. is dual to ours).

Let $\Gamma$ be a Zariski dense $\Theta$-Anosov subgroup with boundary maps $\zeta_\Theta$ and $\zeta_{\involution\Theta}$. We note that $\Gamma$ is in particular $\Theta$-discrete \cite{GW12,KLP17}. We record some basic properties in the following theorem. The first 3 properties are \cite[Theorem 6.1]{BCLS15}, \cite[Proposition 4.6]{PS16}, \cite[Lemma 3.1]{GW12}, respectively, and the last three properties are \cite[Theorem 12.2, Lemmas 12.3, 4.3 and 4.5]{KOW23b} (cf. \cite[Theorem A]{Sam24} and \cite[Lemma 4.8]{Sam14a}).

\begin{theorem}
\label{thm:BasicProperties}
We have the following properties:
\begin{enumerate}
\item\label{itm:BasicProperties1} the boundary maps $\zeta_\Theta$ and $\zeta_{\involution\Theta}$ are H\"{o}lder homeomorphisms onto the limit sets $\limitset$ and $\ilimitset$, respectively;
\item\label{itm:BasicProperties2} the limit cone $\limitcone$ is contained in $\interior \LieA_\Theta^+ \cup \{0\}$;
\item\label{itm:BasicProperties3} every nontorsion element $\gamma \in \Gamma$ is loxodromic;
\item\label{itm:BasicProperties4} the growth indicator $\growthindicator$ is analytic and strictly concave except along radial directions on $\interior \limitcone$ and vertically tangent on $\partial\limitcone$;
\item\label{itm:BasicProperties5} if $\psi \in \LieA_\Theta^*$ is tangent to $\growthindicator$, then $\psi|_{\limitcone-\{0\}} > 0$;
\item\label{itm:BasicProperties6} if $\psi \in \interior\limitcone^*$, then there exists a finite positive constant
\begin{align}
\label{eqn:psi_TopologicalEntropy}
\delta_\psi := \limsup_{t \to +\infty} \frac{1}{t} \log\#\{[\gamma]: \psi(\lambda_\Theta(\gamma)) \leq t\} \in (0, +\infty),
\end{align}
where $[\gamma]$ denotes the conjugacy class of $\gamma \in \Gamma$, so that $\delta_\psi\psi$ is tangent to $\growthindicator$.
\end{enumerate}
\end{theorem}

\begin{remark}
\label{rem:limitset homeomorphism}
For any nonempty subset $\Theta' \subset \Theta$, we have the canonical projection $\Fboundary \to \mathcal{F}_{\Theta'}$. It follows from definitions that it induces a surjective map $\limitset \to \Lambda_{\Theta'}$. As a consequence of \cref{thm:BasicProperties}\labelcref{itm:BasicProperties1}, this map is in fact a homeomorphism.
\end{remark}

For $\psi \in \LieA_\Theta^*$, let $\pi_\psi: \Fboundary^{(2)} \times \LieA_\Theta \to \Fboundary^{(2)} \times \R$ be the projection map defined by
\begin{align}
\label{eqn:ProjectionMap} 
\pi_\psi(x, y, v) = (x, y, \psi(v)) \qquad \text{ for all } (x, y, v) \in \Fboundary^{(2)} \times \LieA_\Theta.
\end{align}
The $\Gamma$-action on $G/S_\Theta \cong \Fboundary^{(2)} \times \LieA_\Theta$ descends to an action on $\Fboundary^{(2)} \times \R$ via the projection $\pi_\psi$ given explicitly by
\begin{align}
\label{eqn:GammaActionOnFboundary^(2)XR}
\gamma (x, y, s) &= \bigl(\gamma x, \gamma y, s + \psi(\beta^\Theta_{\gamma x}(e, \gamma))\bigr) \qquad \text{ for all } \gamma \in \Gamma, \, (x, y, s) \in \Fboundary^{(2)} \times \R.
\end{align}
Let $\limitset^{(2)} := (\zeta_\Theta \times \zeta_{\involution\Theta})(\Gboundary^{(2)}) \subset \Fboundary^{(2)}$. The $\Gamma$-action then restricts to $\limitset^{(2)} \times \R$ and we may define
\begin{align*}
\mathcal{X}_\psi := \Gamma \backslash \bigl(\limitset^{(2)} \times \R\bigr).
\end{align*}
A priori, this topological space may be a terrible, but we will see below that it is nice for $\psi \in \interior\limitcone^*$. The left $\Gamma$-action on $G/S_\Theta \cong \Fboundary^{(2)} \times \LieA_\Theta$ is not necessarily properly discontinuous. For example, when $G=\SL_3(\R)$ and $\#\Theta = 1$, the $\Gamma$-action is not properly discontinuous (see \cite{Ben96,Kob96,DGK17} for a general criterion) but on the other hand, when $G=\SL_n(\R)$ for some $n \ge 4$, there are examples where the $\Gamma$-action is properly discontinuous (see \cite[Corollaries 1.10 and 1.11]{GGKW17}). The $\Gamma$-action on $\Fboundary^{(2)} \times \R$ is also not necessarily properly discontinuous.

\begin{remark}
\label{rem: properly discontinuous}
In fact, there can be points in $\Fboundary^{(2)} \times \R$ with infinite cyclic stabilizer. We give a brief justification for certain choices of $\psi \in \interior\limitcone^*$. One can calculate for any loxodromic element $\gamma \in \Gamma$ with $\gamma \in g(\interior A^+)g^{-1}$ for some $g \in G$, and any element in the Weyl group $w \in \mathcal{W} = N_K(A)/Z_K(A)$, that
\begin{align*}
	\beta_{(gw^{-1})^+}(g, \gamma^j g) = j\Ad_w(\lambda(\gamma)) \qquad \text{for all $j \in \Z_{\geq 0}$}.
\end{align*}
Thus, the claim follows for the choices of $w \in \mathcal{W}$ and $\psi \in \interior\limitcone^*$ such that $\Ad_w(\lambda(\gamma)) \in \ker\psi$, and the point $(gw^{-1}e^+, gw^{-1}e^-, 0) \in \Fboundary^{(2)} \times \R$.
\end{remark}

However, as alluded to above, we have the following theorem for the $\Gamma$-action on $\limitset^{(2)}\times \R$.

\begin{theorem}
\label{thm:GammaAction}
If $\psi \in \interior\limitcone^*$, then the $\Gamma$-action on $\limitset^{(2)} \times \R$ induced by $\pi_\psi$ is properly discontinuous and cocompact; and in particular, $\mathcal{X}_\psi$ is a compact Hausdorff topological space.
\end{theorem}

We remark that \cref{thm:GammaAction} is a special case of \cite[Theorem 9.2]{KOW23b} which describes the linear forms for which the action of a so-called $\Theta$-transverse group on $\limitset^{(2)} \times \R$ is properly discontinuous. Moreover, it was shown that the action is cocompact if and only if the $\Theta$-transverse group is in fact $\Theta$-Anosov. 

Let $\psi \in \interior\limitcone^*$. We equip the space $\mathcal{X}_\psi$ with a Bowen--Margulis--Sullivan (BMS) measure $\BMSX^\psi$ as follows. Let $\delta_\psi > 0$ be as in \cref{thm:BasicProperties}\ref{itm:BasicProperties6}. By \cite[Theorem 1.11]{KOW23b}, there exists a unique $(\Gamma, \delta_\psi\psi)$-conformal measure $\nu_\psi$ (resp. $(\Gamma, \delta_\psi\psi \circ \involution)$-conformal measure $\nu_{\psi\circ\involution}$) on $\Fboundary$ (resp. $\iFboundary$) and moreover, $\nu_{\psi}$ (resp. $\nu_{\psi\circ\involution}$) is supported on $\limitset$ (resp. $\ilimitset$). We define a locally finite Borel measure $\tilde{m}^\psi_{\mathcal{X}}$ on $\limitset^{(2)} \times \R$ by
\begin{equation}
d\tilde{m}^\psi_{\mathcal{X}}(x, y, s) = e^{\delta_\psi\psi(\mathcal{G}^\Theta[x, y])} \, d\nu_{\psi}(x) \, d\nu_{\psi\circ\involution}(y) \, ds \qquad \text{ for all } (x, y, s) \in \limitset^{(2)} \times \R
\end{equation}
where $ds$ denotes the Lebesgue measure on $\R$. Observe that $\tilde{m}^\psi_{\mathcal{X}}$ is left $\Gamma$-invariant, so it descends to a probability measure $\BMSX^\psi$ on $\mathcal{X}_\psi$ (after normalization). By \cite[Proposition 3.3.2]{Sam24}, $\BMSX^\psi$ is the measure of maximal entropy for the translation flows on $\mathcal{X}_\psi$.

A consequence of \cref{thm:GammaAction} is that the restriction of the $\Gamma$-action on $\Fboundary^{(2)} \times \LieA_\Theta$ to $\limitset^{(2)} \times \LieA_\Theta$ is properly discontinuous. Moreover, the locally compact Hausdorff topological space $\Omega :=\Gamma \backslash (\limitset^{(2)} \times \LieA_\Theta)$ is a trivial $\ker\psi$-vector bundle over $\mathcal{X}_\psi$ by a standard argument.

\subsection{Translation flows}
\label{subsec:TranslationFlows}
Given a vector $\sfv \in \LieA_\Theta$, we have a flow $\{a^\psi_{t\sfv}\}_{t \in \R}$ on $\Omega$ which is given by translation by $t\sfv$ in the $\LieA_\Theta$-coordinate. The flow $\{a^\psi_{t\sfv}\}$ descends via $\pi_\psi$ to what we call a \emph{translation flow} $\{a^{\psi, \sfv}_t\}_{t \in \R}$ on $\mathcal{X}_\psi$ given explicitly by
\begin{align*}
a^{\psi, \sfv}_t \cdot \Gamma(x, y, s) := \Gamma(x, y, s + \psi({\sfv})t) \qquad \text{ for all } t \in \R, \, (x, y, s) \in \limitset^{(2)}\times\R.
\end{align*}
By \cite{Sam24}, the topological entropy of the flow $\bigl\{a^{\psi,\sfv}_t\bigr\}_{t \in \R}$ is $\delta_\psi\psi(\sfv)$ (see \cref{eqn:psi_TopologicalEntropy}). Note that only the value of $\psi(\sfv)$ is required to determine the translation flow on $\mathcal{X}_\psi$ and furthermore, rescaling $\psi$ simply rescales the $\R$-coordinate of $\mathcal{X}_\psi$. We introduce the following notation to fix a family of translation flows without the redundacies from scaling $\psi$ or $\sfv$. For $\sfv \in \interior\limitcone$, set 
\begin{align}
\label{eqn:GeneralLinearForm}
\psi_{\sfv} := \frac{\langle\nabla\growthindicator(\sfv), \cdot \rangle}{\growthindicator(\sfv)};&& \mathcal{X}_{\sfv} := \mathcal{X}_{\psi_{\sfv}}; && a^{\sfv}_t := a^{\psi_{\sfv}, \sfv}_t; && \BMSX^{\sfv} := \BMSX^{\psi_{\sfv}}.
\end{align}
Note that $\growthindicator(\sfv)\psi_{\sfv} = \langle\nabla\growthindicator(\sfv), \cdot \rangle$ is the unique linear form tangent to $\growthindicator$ at $\sfv$. Then, the aforementioned redundancies are removed if $\sfv$ is chosen so that $\psi_{\sfv}(\sfv) = 1$ and the topological entropy is then $\delta_\sfv := \delta_{\psi_{\sfv}} = \growthindicator(\sfv)$ by \cref{thm:BasicProperties}\labelcref{itm:BasicProperties6}.

We recall the relation between the translation flow and the (unique up to H\"older reparametrization (see \cref{def:Reparametrization})) Gromov geodesic flow $\bigl(\mathcal{G}, \bigl\{a^{\mathcal{G}}_t\bigr\}_{t \in \R}\bigr)$ associated to $\Gamma$.

\begin{theorem}
\label{thm:GromovGeodesicFlow}
For all $\sfv \in \interior \limitcone$, there exists a H\"{o}lder homemorphism $\mathcal{X}_{\sfv} \to \mathcal{G}$ conjugating $\{a^{\sfv}_t\}_{t \in \R}$ to a H\"{o}lder reparametrization of $\{a^{\mathcal{G}}_t\}_{t \in \R}$.
\end{theorem}

Henceforth, fix $\sfv_0 \in \interior\limitcone$, say $\sfv_0 := \mathsf{u}_\Theta$. The above theorem in particular shows that for any $\sfv \in \interior\limitcone$ the dynamical system $(\mathcal{X}_{\sfv},\{a^{\sfv}_t\}_{t\in\R})$ is conjugate to a reparametrization of $(\mathcal{X}_{\sfv_0},\{a^{\sfv_0}_t\}_{t\in\R})$. In fact, we will see that the conjugating homeomorphism and reparametrizations are bi-Lipschitz (see \cref{thm:TranslationFlowConjugateToRhoFlow}). In light of this fact, we take a unified viewpoint where we have a single compact metric space 
\begin{align*}
\mathcal{X} := \mathcal{X}_{\sfv_0}    
\end{align*}
equipped with a family of pairs of BMS measures and translation flows 
\begin{align*}
\{(\BMSX^{\sfv}, \{a^{\sfv}_t\}_{t \in \R})\}_{\sfv \in \interior\limitcone}.
\end{align*}
To distinguish the family for $\sfv = \sf{v_0}$, we set
\begin{align*}
\psi := \psi_{\sfv_0}, && \BMSX:=\BMSX^{\sfv_0} &&\{a_t := a^{\sfv_0}_t\}_{t \in\R}. 
\end{align*}
The translation flows are then homothety equivariant, i.e., we have
\begin{align}
\label{eqn:HomothetyEquivariance}
a^{c \sfv}_t = a^{\sfv}_{ct} \qquad \text{for all $t \in \R$ and $c > 0$}.
\end{align}

When $\Theta = \Pi$, \cref{thm:GammaAction,thm:GromovGeodesicFlow} follow from the authors' prequel work \cite[Theorem 4.15]{CS23}. Such theorems first appeared in the work of Sambarino \cite[Theorem 3.2]{Sam14b} for fundamental groups of compact negatively curved manifolds and in \cite[Proposition 4.2]{BCLS15} for projective Anosov subgroups whose generalization was outline in \cite[Theorem A.2]{Car23}. The proofs can be generalized for arbitrary $\Theta$ (cf. \cite{Sam24} and see \cite[Theorem 9.1]{KOW23b} for a complete alternative proof).

\section{Metric Anosov property of the translation flow}
\label{sec:AnosovProperty}

The purpose of this section is to show that the translation flow $\{a_t\}_{t\in\R}$ on $\mathcal{X}$ is metric Anosov (see \cref{def:MetricAnosovFlow}).

Before stating the precise theorem, we first need to make an appropriate choice of metric to realize $\mathcal{X}$ as a metric space. Recall the bi-invariant Riemannian metric on the maximal compact subgroup $K < G$ induced by the Killing form $B$ (see \cref{subsec:LieTheoreticPreliminaries}). Since $\Fboundary$ and $\iFboundary$ are right quotients of $K$, the metric on $K$ induces left $K$-invariant metrics $d_{\Fboundary}$ on $\Fboundary$ and $d_{\iFboundary}$ on $\iFboundary$. Together with the standard Euclidean metric $d_\R$ on $\R$ we obtain a product metric $d_{\Fboundary} \times d_{\iFboundary} \times d_\R$ on $\Fboundary\times\iFboundary\times\R$ which restricts to $\Fboundary^{(2)} \times \R$. These metrics also restrict to $\limitset$, $\Lambda_{\involution\Theta}$, and $\limitset^{(2)} \times \R$. We call any metric on $\limitset^{(2)} \times \R$ which is locally bi-Lipschitz equivalent to the product metric on $\limitset^{(2)} \times \R$ a \emph{locally product-like metric}. We equip $\limitset^{(2)} \times \R$ with any \emph{$\Gamma$-invariant} locally product-like metric $\tilde{d}$; it descends to a metric $d$ on $\mathcal{X}$ and gives the metric space $(\mathcal{X}, d)$.

A $\Gamma$-invariant locally product-like metric on $\limitset^{(2)} \times \R$ can be constructed as in \cite[Lemma 4.11]{Cha94} (cf. \cite[Lemma 5.2]{BCLS15}). Let us recount it here since the metric on $\mathcal{X}$ is very important. First, by compactness, we can take a finite open cover of $\mathcal{X}$ and lift it to obtain a locally finite open cover $\mathcal{U} = \{\gamma  \mathcal{U}_j\}_{\gamma \in \Gamma, 1 \le j \le j_0}$ of $\limitset^{(2)} \times \R$, for some $j_0 \in \N$, such that
\begin{align*}
\gamma  \mathcal{U}_j \cap \gamma' \mathcal{U}_j = \varnothing \qquad \text{for all distinct $\gamma, \gamma' \in \Gamma$ and $1 \leq j\leq j_0$}.
\end{align*}
We then put the restricted product metric and their pushforwards
\begin{align*}
\tilde{d}_{e, j} &:= (d_{\Fboundary} \times d_{\iFboundary} \times d_\R)|_{\mathcal{U}_j}, & \tilde{d}_{\gamma, j} &:= \gamma_*\tilde{d}_{e, j} \qquad \text{for all $\gamma \in \Gamma$ and $1 \le j \le j_0$},
\end{align*}
to obtain a $\Gamma$-invariant family of metrics $\{\tilde{d}_{\gamma, j}\}_{\gamma \in \Gamma, 1\le j\le j_0}$ on the open cover $\mathcal{U} = \{\gamma  \mathcal{U}_j\}_{\gamma \in \Gamma, 1 \le j \le j_0}$. We then define the metric $\tilde{d}$ on $\limitset^{(2)} \times \R$ by
\begin{align*}
	\tilde{d}(z_{\mathrm{i}}, z_{\mathrm{f}}) = \inf\left\{\sum_{k = 0}^{N - 1} \tilde{d}_{\gamma_k, j_k}(z_k, z_{k + 1}):
	\begin{array}{l}
		\text{$N \in \N$, $z_0 = z_{\mathrm{i}}$, $z_N = z_{\mathrm{f}}$,} \\
		\text{$\forall\ 0 \leq k \leq N - 1$, $\exists\ \gamma_k \in \Gamma$, $\exists\ 1 \le j_k \le j_0$,} \\
		\text{such that $z_k, z_{k + 1} \in \gamma_k \mathcal{U}_{j_k}$}
	\end{array}
	\!\!\!\right\}
\end{align*}
for all $z_{\mathrm{i}}, z_{\mathrm{f}} \in \limitset^{(2)} \times \R$. Then, one can check that $\tilde{d}$ is $\Gamma$-invariant and locally bi-Lipschitz equivalent to $\tilde{d}_{\gamma, j}$ for all $\gamma \in \Gamma$ and $1 \le j \le j_0$, as desired.

\begin{remark}
We obtain a metric space $(\Omega, d)$ by a similar construction using the Euclidean metric $d_\LieA$ on $\LieA$ induced by the Killing form $B$ (see \cref{subsec:LieTheoreticPreliminaries}).
\end{remark}

\begin{remark}
When $\Theta = \Pi$, we have the diffeomorphism $G/M \cong \mathcal{F}^{(2)} \times \LieA$ where $M := S_\Pi < K$. Then using a $\Gamma$-equivariant continuous section $\Lambda^{(2)} \times \R \to \Lambda^{(2)} \times \LieA$, any left $G$-invariant and right $K$-invariant Riemannian metric on $G$ descends to a metric on $G/M$ pulls back to a locally product-like metric on $\Lambda^{(2)} \times \R$.
\end{remark}

\begin{remark}
\label{rem:MetricsMain}
The construction of the metric space $(\mathcal{X}, d)$ is highly sensitive to the metric space used for the boundary coordinates, and hence to the choice of $\Theta \subset \Pi$. Suppose we view $\Gamma < G$ as a $\Theta'$-Anosov subgroup for some nonempty proper subset $\Theta' \subset \Theta$ and obtain corresponding metric spaces $(\mathcal{X}', d')$ and $(\Omega', d')$. Then, the canonical projection $\Fboundary \to \mathcal{F}_{\Theta'}$ induces a Lipschitz homeomorphism $\limitset \to \Lambda_{\Theta'}$ (see \cref{rem:limitset homeomorphism}) and hence also Lipschitz homeomorphisms
\begin{align*}
\mathcal{X} &\to \mathcal{X}', & \Omega &\to \Omega',
\end{align*}
but their inverses need not be Lipschitz or even H\"older.
\end{remark}

This section is devoted to proving the following theorem.

\begin{theorem}
\label{thm:TranslationFlowIsMetricAnosov}
The translation flow $\{a_t\}_{t \in \R}$ on $(\mathcal{X}, d)$ is a metric Anosov flow with respect to the pair of laminations $(W^{\mathrm{su}}, W^{\mathrm{ss}})$ induced by the $\Theta$-horospherical foliations (see \cref{subsec:StableandUnstable}).
\end{theorem}

We obtain the following as a consequence of \cref{thm:TranslationFlowIsMetricAnosov,thm:MarkovSectionForMetricAnosovFlow}.

\begin{corollary}
\label{cor:TranslationFlowHasMarkovSection}
The translation flow $\{a_t\}_{t\in\R}$ on $(\mathcal{X}, d)$ has a Markov section with respect to $(W^{\mathrm{su}}, W^{\mathrm{ss}})$ (see \cref{def:MarkovSection}).
\end{corollary}

\begin{remark}
This gives another similar approach to obtain \cite[Corollary 4.17]{CS23}, where the metric Anosov property was not directly required. In contrast, the metric Anosov property is crucial in this work, not only to obtain a symbolic coding.
\end{remark}

\subsection{Metric Anosov flows, Markov sections, and reparametrizations}
\label{subsec:MetricAnosovFlows}

In this subsection, we briefly recall the definition of metric Anosov flows and the fact that they admit Markov sections due to Pollicott \cite{Pol87}. The reader can also refer to \cite[\S\,3.2]{BCLS15} and \cite[\S\,4.1]{CS23} for details omitted in this subsection. We also discuss reparametrizations and give a Lipschitz criterion for when a reparametrization of a metric Anosov flow is also metric Anosov.

For this subsection, let $\{\phi_t\}_{t \in \R}$ be an arbitrary continuous flow on an arbitrary metric space $(\mathcal{Y}, d)$. Given a lamination $W$ of $\mathcal{Y}$ and a point $x \in \mathcal{Y}$, let $W(x)$ be the leaf through $x$ and
\begin{align*}
W_\epsilon(x) = \{y \in W(x): d(x, y) < \epsilon\} \qquad \text{for all $\epsilon > 0$}.
\end{align*}

\begin{definition}[Local product structure]
\label{def:LocalProductStructure}
A pair of laminations $(W, W')$ of $\mathcal{Y}$ is said to have \emph{local product structure} if there exists $\epsilon > 0$ such that for all $x \in \mathcal{Y}$, there exists a homeomorphism $[\cdot, \cdot]: W_\epsilon(x) \times W'_\epsilon(x) \to O_x$ where $O_x \subset \mathcal{Y}$ is a neighborhood of $x$ such that $[\cdot, \cdot]^{-1}$ is a chart for both $W$ and $W'$.
\end{definition}

A lamination $W$ is \emph{equivariant} under $\phi$ if $\phi_t(W(x)) = W(\phi_t(x))$ for all $x \in \mathcal{Y}$ and $t \in \mathbb R$. We say that $W$ is \emph{transverse} to $\phi$ if it is equivariant under $\phi$ and there exists $\epsilon > 0$ such that for all $x \in \mathcal{Y}$, there exists a chart $\varphi_x: O_x \to V_x^1 \times V_x^2 = V_x^1 \times (-\epsilon, \epsilon) \times \tilde{V}_x^2$ for $W$, for some topological spaces $V_x^1$ and $\tilde{V}_x^2$, such that
\begin{align*}
\phi_t((\varphi_x)^{-1}(y, s, v)) = (\varphi_x)^{-1}(y, s + t, v)
\end{align*}
for all $(y, s, v) \in V_x^1 \times (-\epsilon, \epsilon) \times \tilde{V}_x^2$ and $t \in \mathbb R$ with $s + t \in (-\epsilon, \epsilon)$.

Given a lamination $W$ transverse to $\{\phi_t\}_{t \in \R}$, we define the corresponding \emph{central lamination} $W^{\mathrm{c}}$ such that for all $x, y \in \mathcal{Y}$, we have $y \in W^{\mathrm{c}}(x)$ if and only if $\phi_t(y) \in W(x)$ for some $t \in \R$. For convenience, we denote by $W^\phi$ the central lamination induced by the trivial lamination by points, i.e., $W^\phi(x) = \phi_\R(x)$ for all $x \in \mathcal{Y}$.

Let $(W^{\mathrm{su}}, W^{\mathrm{ss}})$ be a pair of laminations transverse to $\{\phi_t\}_{t\in\R}$. Denote $W^{\mathrm{wu}} = (W^{\mathrm{su}})^{\mathrm{c}}$ and $W^{\mathrm{ws}} = (W^{\mathrm{ss}})^{\mathrm{c}}$ and suppose $(W^{\mathrm{wu}}, W^{\mathrm{ss}})$ and $(W^{\mathrm{ws}}, W^{\mathrm{su}})$ have local product structures for some sufficiently small $\epsilon_0 \in (0,1)$. Denote by $[\cdot,\cdot]$ only the map provided by the local product structure of $(W^{\mathrm{wu}}, W^{\mathrm{ss}})$.
Subsets $U \subset W_{\epsilon_0}^{\mathrm{su}}(x)$ and $S \subset W_{\epsilon_0}^{\mathrm{ss}}(x)$ for some $x \in \mathcal{Y}$ are called \emph{proper} if $U = \overline{\interior(U)}$ and $S = \overline{\interior(S)}$, where the interiors and closures are taken in the respective plaque topologies.
For any proper subsets $U \subset W_{\epsilon_0}^{\mathrm{su}}(x)$ and $S \subset W_{\epsilon_0}^{\mathrm{ss}}(x)$, we call
\begin{align*}
R = [U, S] = \{[u, s] \in \mathcal{Y}: u \in U, s \in S\} \subset \mathcal{Y}
\end{align*}
a \emph{rectangle of size $\hat{\delta}$} if $\diam(R) \leq \hat{\delta}$ for some $\hat{\delta} \in (0, 2\epsilon_0)$, and the chosen $x$ the \emph{center} of $R$. For any rectangle $R = [U, S]$, we can extend the map $[\cdot, \cdot]$ to $[\cdot, \cdot]: R \to R$ defined by $[v_1, v_2] = [u_1, s_2]$ for all $v_1 = [u_1, s_1] \in [U, S]$ and $v_2 = [u_2, s_2] \in [U, S]$.

\begin{definition}[Complete set of rectangles]
\label{def:CompleteSetOfRectangles}
A set $\mathcal{R} = \{R_1, R_2, \dotsc, R_N\} = \{[U_1, S_1], [U_2, S_2], \dotsc, [U_N, S_N]\}$ for some $N \in \N$ consisting of rectangles of size $\hat{\delta}$ with respect to $(W^{\mathrm{su}}, W^{\mathrm{ss}})$ in $\mathcal{Y}$ is called a \emph{complete set of rectangles of size $\hat{\delta}$} if:
\begin{enumerate}
\item \label{itm:MarkovProperty1} $R_j \cap R_k = \varnothing$ for all $1 \leq j, k \leq N$ with $j \neq k$;
\item \label{itm:MarkovProperty2} $\mathcal{Y} = \bigcup_{j = 1}^N \bigcup_{t \in [0, \hat{\delta}]} \phi_t(R_j)$.
\end{enumerate}
\end{definition}

Let $\mathcal{R} = \{R_1, R_2, \dotsc, R_N\} = \{[U_1, S_1], [U_2, S_2], \dotsc, [U_N, S_N]\}$ be a complete set of rectangles of size $\hat{\delta} \in (0, 2\epsilon_0)$ in $\mathcal{Y}$. We introduce some notation related to $\mathcal{R}$. Let
\begin{align*}
U &= \bigsqcup_{j = 1}^N U_j, & R &= \bigsqcup_{j = 1}^N R_j.
\end{align*}
Define the first return time map $\tau: R \to \R$ by
\begin{align*}
\tau(u) = \inf\{t \in \R_{>0}: \phi_t(u) \in R\} \qquad \text{for all $u \in R$}.
\end{align*}
Define
\begin{align}
\label{eqn:sup_inf_of_tau}
\overline{\tau} &= \sup_{u \in R} \tau(u), & \underline{\tau} = \inf_{u \in R} \tau(u).
\end{align}
Define the Poincar\'{e} first return map $\mathcal{P}: R \to R$ by
\begin{align*}
\mathcal{P}(u) = \phi_{\tau(u)}(u) \qquad \text{for all $u \in R$}.
\end{align*}
Let $\sigma = (\proj_U \circ \mathcal{P})|_U: U \to U$ be its projection where $\proj_U: R \to U$ is the projection defined by $\proj_U([u, s]) = u$ for all $[u, s] \in R$. We define the \emph{cores}
\begin{align*}
\hat{U} &= \{u \in U: \sigma^k(u) \in \interior(U) \text{ for all } k \in \Z_{\geq 0}\}, \\
\hat{R} &= \{u \in R: \mathcal{P}^k(u) \in \interior(R) \text{ for all } k \in \Z\}.
\end{align*}

\begin{definition}[Markov section]
\label{def:MarkovSection}
We say that a complete set of rectangles $\mathcal{R}$ is a \emph{Markov section (with respect to $(W^{\mathrm{su}}, W^{\mathrm{ss}})$)} if is satisfies the \emph{Markov property}:
\begin{align*}
[\interior(U_k), \mathcal{P}(u)] &\subset \mathcal{P}([\interior(U_j), u]), & \mathcal{P}([u, \interior(S_j)]) &\subset [\mathcal{P}(u), \interior(S_k)],
\end{align*}
for all $u \in R$ such that $u \in \interior(R_j) \cap \mathcal{P}^{-1}(\interior(R_k)) \neq \varnothing$, for all $1 \leq j, k \leq N$.
\end{definition}

Observe that if $\mathcal{R}$ is a Markov section, then $\tau$ is constant on $[u, S_j]$ for all $u \in U_j$ and $1 \leq j \leq N$.

\begin{definition}[Metric Anosov flow]
\label{def:MetricAnosovFlow}
The flow $\{\phi_t\}_{t\in\R}$ is said to be \emph{metric Anosov (with respect to $(W^{\mathrm{su}}, W^{\mathrm{ss}})$)} if there exist $\epsilon > 0$, $\eta > 0$, and $C > 0$ such that
\begin{enumerate}
\item $(W^{\mathrm{wu}}, W^{\mathrm{ss}})$ and $(W^{\mathrm{ws}}, W^{\mathrm{su}})$ have local product structures with constant $\epsilon$;
\item it satisfies the \emph{metric Anosov property}: for all $x \in \mathcal{Y}$, $y \in W_\epsilon^{\mathrm{su}}(x)$, and $z \in W_\epsilon^{\mathrm{ss}}(x)$, we have
\begin{align*}
d(\phi_{-t}(x), \phi_{-t}(y)) &\leq Ce^{-\eta t}d(x, y), & d(\phi_t(x), \phi_t(z)) &\leq Ce^{-\eta t}d(x, z),
\end{align*}
for all $t \geq 0$.
\end{enumerate}
\end{definition}

The existence of Markov sections for metric Anosov flows is due to Pollicott \cite{Pol87} and generalizes results of Bowen and Ratner \cite{Bow70,Rat73}.

\begin{theorem}
\label{thm:MarkovSectionForMetricAnosovFlow}
If $\{\phi_t\}_{t\in\R}$ is a metric Anosov flow with respect to $(W^{\mathrm{su}}, W^{\mathrm{ss}})$, then there exists a Markov section with respect to $(W^{\mathrm{su}}, W^{\mathrm{ss}})$ (see \cref{def:MarkovSection}).
\end{theorem}

We now recall the definition of a reparametrization of a flow and give a Lipschitz criterion to verify if a given reparametrization of a metric Anosov flow is also metric Anosov with respect to a given pair of laminations.

\begin{definition}[Reparametrization]
\label{def:Reparametrization}
A flow $\{\hat{\phi}_t\}_{t\in\R}$ on $\mathcal{Y}$ is called a \emph{reparametrization of $\{\phi_t\}_{t\in\R}$} if it is of the form $\hat{\phi}_t(x) = \phi_{\kappa(x, t)}(x)$ for all $x \in \mathcal{Y}$ and $t \in \R$, where $\kappa: \mathcal{Y} \times \R \to \R$ is a continuous map satisfying
\begin{enumerate}
\item positivity: $\kappa(x, t) > 0$ for all $x \in \mathcal{Y}$ and $t > 0$;
\item cocycle condition: $\kappa(x, s + t) = \kappa(\hat{\phi}_s(x), t) + \kappa(x, s)$ for all $x \in \mathcal{Y}$ and $s, t \in \R$.
\end{enumerate}
In that case, $\{\phi_t\}_{t \in \R}$ is itself a reparametrization of $\{\hat{\phi}_t\}_{t \in \R}$ for some continuous map $\kappa^*: \mathcal{Y} \times \R \to \R$ which we call the \emph{inverse of $\kappa$} and satisfies
\begin{enumerate}
\item $(\kappa^*)^* = \kappa$;
\item $\kappa(x, \kappa^*(x, t)) = \kappa^*(x, \kappa(x, t)) = t$ for all $(x, t) \in \mathcal{Y} \times \R$.
\end{enumerate}
We say that a reparametrization is Lipschitz (resp. H\"{o}lder) if $\kappa(\cdot, t)$ is Lipschitz (resp. H\"{o}lder) continuous for all $t \in \R$.
\end{definition}

\begin{remark}
We observe the following:
\begin{enumerate}
\item For any reparametrization $\{\hat{\phi}_t\}_{t\in\R}$ of $\{\phi_t\}_{t\in\R}$, we have $W^{\hat{\phi}} = W^{\phi}$.
\item If a reparametrization is Lipschitz/H\"{o}lder, then so is its inverse; see \cite[Remark 4.11]{CS23}.
\end{enumerate}
\end{remark}

Recall that according to our definitions, laminations which are transverse to a flow are assumed to be equivariant under the same flow. This equivariance property will be used in the proof of the following lemma.

\begin{proposition}
\label{pro:ReparametrizationIsMetricAnosov}
Let $\{\phi_t\}_{t\in\R}$ be a metric Anosov flow with respect to $(W^{\mathrm{su}}, W^{\mathrm{ss}})$ and with constants $(\epsilon, \eta, C)$. Suppose that $\bigl\{\hat{\phi}_t := \phi_{\kappa(\cdot, t)}\bigr\}_{t \in \R}$ is a reparametrization of $\{\phi_t\}_{t \in \R}$ and $\bigl(\widehat{W}^{\mathrm{su}}, \widehat{W}^{\mathrm{ss}}\bigr)$ is a pair of laminations transverse to $\{\hat{\phi}_t\}_{t \in \R}$ such that $\widehat{W}^{\mathrm{su}} \subset W^{\mathrm{wu}}$ and $\widehat{W}^{\mathrm{ss}} \subset W^{\mathrm{ws}}$ leafwise, and $\bigl(\widehat{W}^{\mathrm{wu}}, \widehat{W}^{\mathrm{ss}}\bigr)$ and $\bigl(\widehat{W}^{\mathrm{ws}}, \widehat{W}^{\mathrm{su}}\bigr)$ have local product structures with constant $\epsilon$. Suppose also that there exist $c \geq 1$ and $c' > 0$ such that 
\begin{enumerate}
\item \label{itm:FoliationProperty} 
for all $x \in \mathcal{Y}$, $y \in \widehat{W}_\epsilon^{\mathrm{su}}(x)$, $z \in \widehat{W}_\epsilon^{\mathrm{ss}}(x)$, and the unique intersections $y' \in W_{c^{-1}\epsilon}^{\mathrm{su}}(x) \cap W_{2\epsilon}^{\phi}(y)$ and $z' \in W_{c^{-1}\epsilon}^{\mathrm{ss}}(x) \cap W_{2\epsilon}^\phi(z)$, we have
\begin{align*}
c^{-1}d(x, y') \le d(x, y) \le cd(x, y'), & & c^{-1}d(x, z') \le d(x,z) \le cd(x, z');
\end{align*}
\item \label{itm:KappaProperty} 
$\kappa(x, t) \ge c't$ for all $t > 0$ and $x \in \mathcal{Y}$.
\end{enumerate}
Then, $\{\hat{\phi}_t\}_{t \in \R}$ is a metric Anosov flow with respect to $\bigl(\widehat{W}^{\mathrm{su}}, \widehat{W}^{\mathrm{ss}}\bigr)$ and with constants $(\epsilon, c'\eta, c^2C)$.
\end{proposition}

\begin{proof}
Let $\{\phi_t\}_{t\in\R}$, $\{\hat{\phi}_t\}_{t\in\R}$, $(\epsilon, \eta, C)$, $c$, and $c'$ be as in the proposition. Suppose the hypotheses of the proposition hold. It suffices to check the metric Anosov property for $\{\hat{\phi}_t\}_{t \in \R}$. We check it for the stable lamination and the unstable lamination case is similar. By the metric Anosov property of $\{\phi_t\}_{t \in \R}$, for all $x \in \mathcal{Y}$, $z' \in W_\epsilon^{\mathrm{ss}}(x)$, and $t>0$ we have
\begin{align*}
d(\phi_t(x), \phi_t(z')) \leq Ce^{-\eta t}d(x, z').
\end{align*}
Let $x \in \mathcal{Y}$, $z \in \widehat{W}_{c^{-2}C^{-1}\epsilon}^{\mathrm{ss}}(x)$, and $t > 0$. By equivariance of $\widehat{W}^{\mathrm{ss}}$ under $\{\hat{\phi}_t\}_{t \in \R}$, the points $\hat{\phi}_t(x)$ and $\hat{\phi}_t(z)$ are on the same leaf of the lamination $\widehat{W}^{\mathrm{ss}}$. Take the unique intersection $z' \in W_{c^{-1}\epsilon}^{\mathrm{ss}}(x) \cap W_{2\epsilon}^\phi(z)$ so that the points $x$ and $z'$ are on the same leaf of the lamination $W^{\mathrm{ss}}$. Again, by equivariance of $W^{\mathrm{ss}}$ under $\{\phi_t\}_{t \in \R}$, the points $\phi_{\kappa(x, t)}(x)$ and $\phi_{\kappa(x, t)}(z')$ are on the same leaf of the lamination $W^{\mathrm{ss}}$. We now calculate that
\begin{align*}
d(\hat{\phi}_t(x), \hat{\phi}_t(z)) & = d(\phi_{\kappa(x, t)}(x), \phi_{\kappa(z, t)}(z))
\\
& \le cd(\phi_{\kappa(x, t)}(x), \phi_{\kappa(x, t)}(z')) \qquad \text{by \cref{itm:FoliationProperty}}
\\
& \le cCe^{-\eta\kappa(x, t)}d(x, z') \qquad \text{by the metric Anosov property}
\\
& \le c^2Ce^{-c'\eta t}d(x, z) \qquad \text{by \cref{itm:FoliationProperty,itm:KappaProperty}}.
\end{align*}
\end{proof}

\subsection{Stable and unstable laminations for the translation flow}
\label{subsec:StableandUnstable}
In this subsection, we introduce a pair of natural laminations on $\mathcal{X}$ transverse to the translation flow. The right $N_\Theta^\pm$-orbits in $G$ give the \emph{$\Theta$-horospherical} foliations. Recall that $P_\Theta$ is the normalizer of $N_\Theta^-$ in $G$. In particular, $S_\Theta \subset L_\Theta = P_\Theta \cap w_0P_{\involution\Theta}w_0^{-1}$ normalizes both $N_\Theta^-$ and $N_\Theta^+ = w_0N_{\involution\Theta}^-w_0^{-1}$. Then the $\Theta$-horospherical foliations descend to foliations on $G/S_\Theta \cong \Fboundary^{(2)} \times \LieA_\Theta$ which we denote by $H^{\Fboundary, \pm}$, and restrict to laminations on $\limitset^{(2)} \times \LieA_\Theta$ which we denote by $H^\pm$. More explicitly,
\begin{align*}
H^{\Fboundary, \pm}(gS_\Theta) := \{ghS_\Theta: h \in N_\Theta^\pm\} \qquad \text{for all $gS_\Theta \in G/S_\Theta$}.
\end{align*}
Using the Hopf parametrization, they are expressed as
\begin{align*}
H^{\Fboundary, +}(gS_\Theta) = &\, \bigl\{\bigl((gh)^+, (gh)^-, \beta^\Theta_{(gh)^+}(e, gh)\bigr) \, : \, h \in N_\Theta^+\bigr\}
\\
= &\, \bigl\{\bigl((gh)^+, g^-, \beta^\Theta_{g^+}(e, g) + \mathcal{G}^\Theta[(gh)^+, g^-] - \mathcal{G}^\Theta[g^+, g^-]\bigr) : h \in N_\Theta^+\bigr\},
\\
H^{\Fboundary, -}(gS_\Theta) = &\, \bigl\{\bigl((gn)^+, (gn)^-, \beta^\Theta_{(gn)^+}(e, gn)\bigr) \, : \, n \in N_\Theta^-\bigr\}
\\
= &\, \bigl\{\bigl(g^+, (gn)^-, \beta^\Theta_{g^+}(e, g)\bigr) \, : \, n \in N_\Theta^-\bigr\},
\end{align*} 
(see \cite[Lemma 7.4]{KOW23a} for the equalities) for all $gS_\Theta \in G/S_\Theta$. We denote the image foliations/laminations induced by $H^{\Fboundary, \pm}$ and $H^\pm$ under the projection $\pi_\psi$ (see \cref{eqn:ProjectionMap}) by $W^{\Fboundary, \mathrm{su}/\mathrm{ss}}$ and $W^{\mathrm{su}/\mathrm{ss}}$ and call them the \emph{strong unstable} and \emph{strong stable} foliations/laminations, respectively. Finally, we use the same notation $W^{\mathrm{su}/\mathrm{ss}}$ for the image laminations under the quotient map $\limitset^{(2)} \times \R \to \mathcal{X}$. The following lemma, which is analogous to \cite[Lemma 4.16]{CS23}, ensures that the above procedure produces well-defined foliations/laminations on $\Fboundary^{(2)} \times \R$, $\limitset^{(2)} \times \R$, and $\mathcal{X}$.

\begin{lemma}
\label{lem:FoliationsWellDefined}
Let $g_1, g_2 \in G$ such that
\begin{align*}
\bigl(g_1^+, g_1^-, \psi\bigl(\beta^\Theta_{g_1^+}(e, g_1)\bigr)\bigr) = \bigl(g_2^+, g_2^-, \psi\bigl(\beta^\Theta_{g_2^+}(e, g_2)\bigr)\bigr).
\end{align*}
Then,
\begin{enumerate}
\item for all $h_1, h_2 \in N_\Theta^+$ with $(g_1h_1)^+ = (g_2h_2)^+$, we have 
\begin{align*}
	\psi\bigl(\beta^\Theta_{(g_1h_1)^+}(e, g_1h_1)\bigr) = \psi\bigl(\beta^\Theta_{(g_2h_2)^+}(e, g_2h_2)\bigr);
\end{align*}
\item $\psi\bigl(\beta^\Theta_{(g_1n)^+}(e, g_1n)\bigr) = t$ for all $n \in N_\Theta^-$.
\end{enumerate}
\end{lemma}

\begin{proof}
Since the calculation for property~(2) is similar (and simpler), we only prove property~(1). Let $g_1, g_2 \in G$ as in the lemma and suppose $h_1, h_2 \in N_\Theta^+$ with $(g_1h_1)^+ = (g_2h_2)^+$. Since $(g_1^+, g_1^-) = (g_2^+, g_2^-)$, using $P_\Theta \cap w_0P_{\involution\Theta}w_0^{-1} = L_\Theta$ there exists $as \in A_\Theta S_\Theta = L_\Theta$ such that $g_1 = g_2as$. Then, by an elementary computation, we have
\begin{align*}
\beta^\Theta_{g_1^+}(e, g_1) = \beta^\Theta_{g_2^+}(e, g_2) + \log(a).
\end{align*}
Since $\psi\bigl(\beta^\Theta_{g_1^+}(e, g_1)\bigr) = \psi\bigl(\beta^\Theta_{g_2^+}(e, g_2)\bigr)$ by hypothesis, we obtain $\psi(\log(a)) = 0$.

Now, since
\begin{align*}
((g_1h_1)^+, (g_1h_1)^-) = ((g_1h_1)^+, g_1^-) = ((g_2h_2)^+, g_2^-) = ((g_2h_2)^+, (g_2h_2)^-),
\end{align*}
we also have $g_1h_1 = g_2h_2a's'$ for some $a's' \in A_\Theta S_\Theta$. Combining with $g_1 = g_2as$, we obtain 
$ash_1 = h_2a's' = a's'((a's')^{-1}h_2(a's'))$. Since $L_\Theta$ normalizes $N_\Theta^+$ and $A_\Theta S_\Theta N_\Theta^+$ is a direct product, it follows that 
\begin{align*}
a' = a.
\end{align*} 
Similar to before, we have 
\begin{align*}
\beta^\Theta_{(g_1h_1)^+}(e, g_1h_1) = \beta^\Theta_{(g_2h_2)^+}(e, g_2h_2) + \log(a') = \beta^\Theta_{(g_2h_2)^+}(e, g_2h_2) + \log(a).
\end{align*} 
Since $\psi(\log(a)) = 0$, applying $\psi$ to the above equation gives $\psi\bigl(\beta^\Theta_{(g_1h_1)^+}(e, g_1h_1)\bigr) = \psi\bigl(\beta^\Theta_{(g_2h_2)^+}(e, g_2h_2)\bigr)$ as desired.
\end{proof}

Hence, we obtain \emph{(weak/strong) (stable/unstable)} laminations on $\mathcal{X}$ as follows: for all $z =\Gamma(x, y, s) \in \mathcal{X}$, we have
\begin{align}
\label{eqn:StrongUnstableAndStrongStableFoliations}
\begin{aligned}
W^{\mathrm{su}}(z) ={}&\bigl\{\Gamma\bigl((gh)^+, g^-, s + \psi\bigl(\mathcal{G}^\Theta[(gh)^+, g^-] - \mathcal{G}^\Theta[g^+, g^-]\bigr)\bigr): \\
{}&h \in N_\Theta^+, (gh)^+ \in \limitset\bigr\}, 
\\
W^{\mathrm{ss}}(z) ={}&\bigl\{\Gamma\bigl(g^+, (gn)^-, s\bigr): n \in N_\Theta^-, (gn)^- \in \ilimitset\bigr\},
\\
W^{\mathrm{wu}}(z) ={}&\bigcup_{t \in \R} a_t W^{\mathrm{ss}},
\\
W^{\mathrm{ws}}(z) ={}&\bigcup_{t \in \R} a_t W^{\mathrm{ss}}
\end{aligned}
\end{align}
for any choice of $g \in G$ such that $z = \bigl(g^+, g^-, \psi\bigl(\beta^\Theta_{g^+}(e, g)\bigr)\bigr)$. It is easy to see that $W^{\mathrm{ss}}$ and $W^{\mathrm{su}}$ are transverse to the translation flow, and $(W^{\mathrm{ss}}, W^{\mathrm{wu}})$ and $(W^{\mathrm{su}}, W^{\mathrm{ws}})$ have local product structures.

\subsection{Projective Anosov representations}
In this subsection, we recall from \cite{BCLS15} facts about projective Anosov representations that we will need for the proof of \cref{thm:TranslationFlowIsMetricAnosov}. Let $\mathfrak{p}_\Theta$ denote the Lie algebra of $P_\Theta$. Then the adjoint representation $\Ad: G \to \Aut(\LieG)$ induces the representation $\rho_0:G \to \PSL\bigl(\bigwedge^{\dim\mathfrak{p}_\Theta}\LieG\bigr)$. Define the vector space $V := \Span\bigl(\rho_0(G)\bigl(\bigwedge^{\dim\mathfrak{p}_\Theta}\mathfrak{p}_\Theta\bigr)\bigr)$. Then, we have the induced irreducible representation $\rho: G \to \PSL(V)$, called the \emph{Pl\"ucker representation}, and it induces smooth embeddings $\Fboundary \to \bP(V)$ and $\iFboundary \to \bP(V^*)$ \cite[Theorem 7.25]{Lee13}. By \cite[Proposition 4.3]{GW12}), the restriction $\rho|_\Gamma: \Gamma \to \PSL(V)$ is a \emph{projective Anosov representation} in the sense of \cite[Definition 2.2]{BCLS15} and in particular, the previous smooth maps restrict to $\Gamma$-equivariant bi-Lipschitz maps $\zeta_\rho : \limitset \to \bP(V)$ and $\zeta_\rho^* : \ilimitset \to \bP(V^*)$. 

The $\R$-bundle
\begin{align*}
\tilde{\mathcal{G}}_\rho := \bigl\{(x, y, (v, \Psi)) : (x, y) \in \limitset^{(2)}, (v, \Psi) \in \zeta_\rho(x) \times \zeta_\rho^*(y), \Psi(v) = 1\bigr\}/{\sim}
\end{align*}
over $\limitset^{(2)}$, where $(v, \Psi) \sim (-v, -\Psi)$ is equipped with a $\Gamma$-action and a flow $\{\tilde{a}_{\rho, t}\}_{t \in \R}$ that commute with each other which are given as follows: for all $(x, y, (v, \Psi)) \in \tilde{\mathcal{G}}_\rho$, $\gamma \in \Gamma$, and $t \in \R$, let
\begin{align*}
\gamma (x, y, (v, \Psi)) & := (\gamma x, \gamma y, (\rho(\gamma)v, \Psi \circ \rho(\gamma)^{-1}));
\\
\tilde{a}_{\rho, t} (x, y, (v, \Psi)) & := (x, y, (e^tv, e^{-t}\Psi)).
\end{align*}
Let $\mathcal{G}_\rho := \Gamma \backslash \tilde{\mathcal{G}}_\rho$. Then $\{\tilde{a}_{\rho, t}\}_{t \in \R}$ descends to a flow $\{a_{\rho, t}\}_{t \in \R}$ on $\mathcal{G}_\rho$ called the \emph{$\rho$-geodesic flow}.

Any Euclidean metric on $V$ induces a metric on $\bP(V) \times \bP(V^*) \times (V \times V^*)$ and this further induces a metric on $\tilde{\mathcal{G}}_\rho$ in a natural way. Any metric on $\tilde{\mathcal{G}}_\rho$ obtained in this fashion is called a \emph{linear} metric. 

\begin{theorem}[{\citealp[Proposition 5.7]{BCLS15}}]
\label{thm:RhoGeodesicFlowIsMetricAnosov}
There exists a $\Gamma$-invariant metric $\tilde{d}_\rho$ on $\tilde{\mathcal{G}}_\rho$, which is bi-Lipschitz equivalent to any linear metric, such that it descends to a metric $d_\rho$ on $\mathcal{G}_\rho$ for which the flow $\{a_{\rho, t}\}_{t \in \R}$ is metric Anosov with respect to the pair of laminations $\bigl(W^{\mathrm{su}}_\rho, W^{\mathrm{ss}}_\rho\bigr)$ which are defined as follows: for all $z = \Gamma(x, y, (v, \Psi)) \in \mathcal{G}_\rho$, 
\begin{align*}
W^{\mathrm{su}}_\rho(z) &:= \bigl\{\Gamma(x', y, (v', \Psi)) : (x', y) \in \limitset^{(2)}, (v', \Psi) \in \zeta_\rho(x') \times \zeta_\rho^*(y), \Psi(v') = 1\bigr\};
\\
W^{\mathrm{ss}}_\rho(z) &:= \bigl\{\Gamma(x, y', (v, \Psi')) : (x, y') \in \limitset^{(2)}, (v, \Psi') \in \zeta_\rho(x) \times \zeta_\rho^*(y'), \Psi'(v) = 1\bigr\}.
\end{align*}
\end{theorem}

\subsection{Proof of \texorpdfstring{\cref{thm:TranslationFlowIsMetricAnosov}}{\autoref{thm:TranslationFlowIsMetricAnosov}}}
\label{subsec:TranslationFlowIsMetricAnosovProof}
We will prove \cref{thm:TranslationFlowIsMetricAnosov} using techniques much of which are established and appear, for instance, in \cite{BCLS15}, though it will be convenient to refer to \cite{CS23} due to similarity in the setting and notation.

\begin{theorem}
\label{thm:TranslationFlowConjugateToRhoFlow}
There exists a bi-Lipschitz homeomorphism $\mathcal{H}:(\mathcal{G}_\rho, d_\rho) \to (\mathcal{X}, d)$ which conjugates the translation flow $\{a_t\}_{t \in \R}$ to a Lipschitz reparametrization\linebreak $\{a_{\rho, \kappa(\cdot, t)}\}_{t \in \R}$ of the $\rho$-geodesic flow $\{a_{\rho, t}\}_{t \in \R}$.
\end{theorem}

\begin{proof}
We proceed in the following two steps.

\medskip
\noindent
\textit{Step 1: Construction of $\mathcal{H}$.}
We only give an overview (omitting details) of the construction of $\mathcal{H}$ as it is similar to the full construction in \cite[Theorem 4.15]{CS23} which deals with the case $\Theta = \Pi$. Consider the trivial $\R$-bundle 
\begin{align*}
\tilde{\mathcal{L}} := \tilde{\mathcal{G}}_\rho \times \R_{>0}
\end{align*}
equipped with a left $\Gamma$-action defined by
\begin{align*}
\gamma  (z, r) := \bigl(\gamma z, re^{\psi(\beta^\Theta_{x}(\gamma^{-1}, e))}\bigr) \qquad \text{ for all } \gamma \in \Gamma \text{ and } (z,r) \in \tilde{\mathcal{L}}
\end{align*} 
and a flow $\{\tilde{\phi}^{\mathcal{L}}_t\}_{t \in \R}$ defined by
\begin{align*}
\tilde{\phi}^{\mathcal{L}}_t(z, r) := (\tilde{a}_{\rho, t}z, r) \qquad \text{ for all } (z,r) \in \tilde{\mathcal{L}}, \text{ and } t \in \R
\end{align*}
which commutes with the $\Gamma$-action.

Let $\mathcal{U} = \{\gamma \mathcal{U}_j\}_{\gamma \in \Gamma, \, 1 \le j \le j_0}$ for some $j_0 \in \N$ be a locally finite open cover of $\tilde{\mathcal{G}}_\rho$ such that $\gamma \mathcal{U}_j \cap \gamma' \mathcal{U}_j = \varnothing$ for all distinct $\gamma, \gamma' \in \Gamma$ and $1 \le j \le j_0$. Using a partition of unity $\{\varphi_i\}_{i \in I}$, for some index set $I$, subordinate to $\mathcal{U}$ such that $\varphi_i: \limitset^{(2)} \times \R \to [0, +\infty)$ is smooth along the $\rho$-geodesic flow for all $i \in I$, we can construct a $\Gamma$-equivariant section $u_0:\tilde{\mathcal{G}}_\rho \to \tilde{\mathcal{L}}$ of the form $u_0(z) := (z, \hat{u}_0(z))$. Explicitly, we can take
\begin{align*}
\log \hat{u}_0(z) = \sum_{i \in I, \, \supp\varphi_i \subset \gamma\mathcal{U}_j}\psi\bigl(\beta^\Theta_x(e,\gamma)\bigr) \cdot \varphi_i(z) \qquad \text{ for all } z = (x,y, (v, \Psi)) \in \tilde{\mathcal{G}}_\rho.
\end{align*}

Equip $\tilde{\mathcal{L}}$ with the unique $\Gamma$-invariant bundle norm $\|\cdot\|^{\tilde{\mathcal{L}}}_0$ satisfying $\|u_0(z)\|_0^{\tilde{\mathcal{L}}}=1$ for all $z \in \tilde{\mathcal{G}}_\rho$. Then, $\mathcal{L} := \Gamma \backslash \tilde{\mathcal{L}}$ is an $\R_{>0}$-bundle over $\mathcal{G}_\rho$, and $\{\tilde{\phi}^{\mathcal{L}}_t\}_{t \in \R}$ descends to a flow $\{\phi^{\mathcal{L}}_t\}_{t \in \R}$ on $\mathcal{L}$, and $\|\cdot\|_0^{\tilde{\mathcal{L}}}$ descends to a norm $\|\cdot\|_0^{\mathcal{L}}$ on $\mathcal{L}$. Using the Morse property \cite[Theorem 4.13]{DKO24} of Kapovich--Leeb--Porti \cite[Proposition 5.16]{KLP17} and an analogue of Sullivan's Shadow Lemma \cite[Lemma 3.1]{KOW23a}, one can show that $\{\phi^{\mathcal{L}}_t\}_{t \in \R}$ is \emph{discretely contracting} with respect to $\|\cdot\|_0^{\mathcal{L}}$, i.e., there exists $T > 0$ and $\eta > 0$ such that $\|\phi_T^{\mathcal{L}}(\ell)\|_0^{\mathcal{L}} \le e^{-\eta}\|\ell\|_0^{\mathcal{L}}$ for all $\ell \in \mathcal{L}$. By \cite[Lemma 4.3]{BCLS15}, the norm $\|\cdot\|^{\tilde{\mathcal{L}}} = \int_0^T e^{ s/T}\|\phi_s^{\mathcal{L}}(\cdot)\|_0^{\tilde{\mathcal{L}}} \, ds$ on $\tilde{\mathcal{L}}$ descends to a norm $\|\cdot\|^{\mathcal{L}}$ on $\mathcal{L}$ such that $\{\phi_t^{\mathcal{L}}\}_{t\in\R}$ is \emph{continuously contracting} with respect to $\|\cdot\|^{\mathcal{L}}$. Note that $\|(z, \hat{u}(z))\|^{\tilde{\mathcal{L}}} = 1$ where
\begin{align*}
\frac{1}{\hat{u}(z)} = \int_0^T \frac{e^{ s/T}}{\hat{u}_0(\tilde{a}_{\rho,s}(z))} \, ds \qquad \text{ for all } z \in \tilde{\mathcal{G}}_\rho.
\end{align*}

Let $\tilde{\mathcal{H}} : \tilde{\mathcal{G}}_\rho \to \limitset^{(2)} \times \R$ be the $\Gamma$-equivariant map defined by
\begin{align*}
\tilde{\mathcal{H}}(z) = (x, y, \log\hat{u}(z)) \qquad \text{ for all } z = (x,y, (v, \Psi)) \in \tilde{\mathcal{G}}_\rho.
\end{align*}
Then, it can be shown as in \cite[Theorem 4.15]{CS23} that $\hat{u}$ is locally Lipschitz and $\tilde{\mathcal{H}}$ descends to a Lipschitz homeomorphism $\mathcal{H} : \mathcal{G}_\rho \to \mathcal{X}$ and the reparametrization is Lipschitz. Here, $\mathcal{H}$ is Lipschitz instead of H\"{o}lder as in \cite{CS23} since the induced maps $\Fboundary \to \bP(V)$ and $\iFboundary \to \bP(V^*)$ are Lipschitz.

\medskip
\noindent
\textit{Step 2: $\mathcal{H}$ is bi-Lipschitz.}
By the above, it suffices to show that $\tilde{\mathcal{H}}^{-1}$ is locally Lipschitz with respect to $\tilde{d}$ and $\tilde{d}_\rho$.
Recall that $\tilde{d}$ is a locally bi-Lipschitz equivalent to the product metric on $\limitset^{(2)} \times \R$ and $\tilde{d}_\rho$ is locally bi-Lipschitz equivalent to any linear metric on $\tilde{\mathcal{G}}_\rho$. Thus, it suffices to show that $\tilde{\mathcal{H}}^{-1}$ is locally Lipschitz with respect to the product metric on $\limitset^{(2)} \times \R$ and any linear metric on $\tilde{\mathcal{G}}_\rho$. Fix a Euclidean metric $d_{\mathrm{E}}$ on $V$. We also use $d_{\mathrm{E}}$ to denote the induced metrics on $V^*$ and $V \times V^*$. Fix a compact subset $\mathcal{K} \subset \tilde{\mathcal{G}}_\rho$ that is ``radially symmetric'' in the sense that if $z_i = (x_i,y_i, (v_i, \Psi_i)) \in \mathcal{K}$ for $i\in\{1,2\}$, then $p(z_1,z_2) := \Bigl(x_1,y_1, \Bigl(\tfrac{\|v_2\|}{\|v_1\|}v_1, \tfrac{\|v_1\|}{\|v_2\|}\Psi_1\Bigr)\Bigr) \in \mathcal{K}$ as well. We need to show that there exists $c_1 > 0$ such that for any $z_i = (x_i,y_i, (v_i, \Psi_i)) \in \mathcal{K}$ for $i \in \{1, 2\}$, we have
\begin{align*}
d_{\mathrm{E}}((v_1, \Psi_1), (v_2, \Psi_2)) \le c_1((d_{\Fboundary} \times d_{\iFboundary})((x_1,y_1), (x_2,y_2)) + |\log\hat{u}(z_1) - \log\hat{u}(z_2)|).
\end{align*} 

For convenience, let $t(z_1,z_2) := \log (\|v_2\|/\|v_1\|)$ so that 
\begin{align*}
p(z_1,z_2) = \bigl(x_1,y_1, \bigl(e^{t(z_1,z_2)}v_1,e^{-t(z_1,z_2)}\Psi_1\bigr)\bigr) = \tilde{a}_{\rho, t(z_1,z_2)}(z_1).
\end{align*}
Observe that since $\mathcal{K}$ is compact, there exists constants $c_3 > c_2 > 0$ such that
\begin{align*}
d_{\mathrm{E}}\bigl((v_1,\Psi_1), \bigl(e^{t(z_1,z_2)}v_1,e^{-t(z_1,z_2)}\Psi_1\bigr)\bigr) \le c_2t(z_1,z_2)
\end{align*}
and 
\begin{align*}
d_{\mathrm{E}}\Bigl(\Bigl(\tfrac{\|v_2\|}{\|v_1\|}v_1, \tfrac{\|v_1\|}{\|v_2\|}\Psi_1\Bigr), (v_2,\Psi_2)\Bigr) & \le c_2d_{\Fboundary}(x_1,x_2) + \tfrac{1}{\|v_2\|}d_{\mathrm{E}}\bigr(\|v_1\|\Psi_1,\|v_2\|\Psi_2\bigl)
\\
& \le c_3(d_{\Fboundary}(x_1,x_2) + d_{\iFboundary}(y_1,y_2)),
\end{align*}
where the last inequality uses the observation that for $i \in \{1, 2\}$, the linear form $\|v_i\|\Psi_i \in \zeta_\rho^*(y_i)$ which satisfies $\|v_i\|\Psi_i\bigl(\tfrac{v_i}{\|v_i\|}\bigr) = 1$ is smoothly determined by $y_i$ and the unit vector $\tfrac{v_i}{\|v_i\|}$ (or equivalently, $x_i$). The triangle inequality gives
\begin{align*}
d_{\mathrm{E}}((v_1,\Psi_1), (v_2,\Psi_2)) \le c_3((d_{\Fboundary} \times d_{\iFboundary})((x_1,y_1), (x_2,y_2)) + |t(z_1,z_2)|). 
\end{align*} 
Hence, it suffices to show that there exists a constant $c_4 > 0$ such that for all $z_1, z_2 \in \mathcal{K}$, we have
\begin{align*}
|t(z_1,z_2)| \le c_4((d_{\Fboundary} \times d_{\iFboundary})((x_1,y_1), (x_2,y_2)) + |\log\hat{u}(z_1) - \log\hat{u}(z_2)|). 
\end{align*}
Since $\{\phi_t^{\mathcal{L}}\}_{t\in\R}$ is uniformly contracting with respect to $\|\cdot\|^{\mathcal{L}}$, there exists $\eta> 0$ such that
\begin{align*}
\frac{1}{\hat{u}(\tilde{a}_{\rho,t}(z))} \le \frac{e^{-\eta t}}{\hat{u}(z)} \qquad \text{ for all } t > 0 \text{ and } z \in \tilde{\mathcal{G}}_\rho,
\end{align*}
i.e., $\eta t \le \log\hat{u}(\tilde{a}_{\rho,t}(z_1)) - \log\hat{u}(z_1)$. 
Then, for all $z_1, z_2 \in \mathcal{K}$, we have
\begin{align*}
& |\log\hat{u}(z_1) - \log\hat{u}(z_2)| 
\\
\ge & |\log\hat{u}(z_1) - \log\hat{u}(\tilde{a}_{\rho,t(z_1,z_2)}(z_1))| -  |\log\hat{u}(p(z_1,z_2)) - \log\hat{u}(z_2)|
\\
\ge & \eta |t(z_1,z_2)| - |\log\hat{u}(p(z_1,z_2)) - \log\hat{u}(z_2)|.
\end{align*}
Lastly, since $\hat{u}$ is locally Lipschitz with respect to any linear metric, there exists $c_5 > 0$ such that $|\log\hat{u}(p(z_1,z_2)) - \log\hat{u}(z_2)| \le c_5 (d_{\Fboundary} \times d_{\iFboundary})((x_1,y_1), (x_2,y_2))$ for all $z_1, z_2 \in \mathcal{K}$ and this completes the proof.
\end{proof}

The following is an immediate consequence of \cref{thm:RhoGeodesicFlowIsMetricAnosov,thm:TranslationFlowConjugateToRhoFlow}.

\begin{proposition}		\label{pro:ReparametrizationOfTranslationFlowIsMetricAnosov}
The reparametrization $\{a_{\kappa^*(\mathcal{H}^{-1}(\cdot), t)}\}_{t \in \R}$ of $\{a_t\}_{t \in \R}$ is a metric Anosov flow with respect to $\bigl(\mathcal{H}(W^{\mathrm{su}}_\rho), \mathcal{H}(W^{\mathrm{ss}}_\rho)\bigr)$. 		
\end{proposition}

\begin{proof}[Proof of \cref{thm:TranslationFlowIsMetricAnosov}]
In view of \cref{pro:ReparametrizationOfTranslationFlowIsMetricAnosov}, by considering $\{a_t\}_{t \in \R}$ as a reparametrization of $\{a_{\kappa^*(\mathcal{H}^{-1}(\cdot), t)}\}_{t \in \R}$, it suffices to check \cref{itm:FoliationProperty,itm:KappaProperty} in \cref{pro:ReparametrizationIsMetricAnosov}. 

\medskip
\noindent
\textit{Proof of \cref{itm:FoliationProperty}.}
Fix a compact fundamental domain $D \subset \limitset^{(2)} \times \R$ for the $\Gamma$-action. Fix $r>0$ sufficiently small so that the closed $r$-balls in $\limitset^{(2)} \times \R$ injectively project into $\mathcal{X}$ and that the projection $D' \subset \limitset^{(2)}$ of the closed $r$-neighborhood of $D$ is compact. Since $D$ is compact and $\tilde{d}$ is locally bi-Lipschitz equivalent to the product metric, there exists $c_1 > 1$ such that for all $z \in D$ and $z' \in \limitset^{(2)} \times \R$ with $\tilde{d}(z, z') < r$, we have 
\begin{equation}
\label{eqn:BiLipschitzEquivalence}
c_1^{-1}(d_{\Fboundary} \times d_{\iFboundary} \times d_\R)(z, z') \le d(\Gamma z, \Gamma z') = \tilde{d}(z, z') \le c_1(d_{\Fboundary} \times d_{\iFboundary} \times d_\R)(z, z').
\end{equation}
Let $\epsilon \in (0,r)$ which will be specified later. Fix $z = (x, y, t) \in D$, $z_1 = (x', y, t_1) \in \tilde{W}_\epsilon^{\mathrm{su}}(z)$, and $z_2 = (x', y, t_2) \in \limitset^{(2)} \times \R$ such that $\Gamma z_2 \in a_\R(\Gamma z_1) \cap \mathcal{H}(W^{\mathrm{su}}_\rho(\mathcal{H}^{-1}(\Gamma z))$. By compactness of $D$, we can uniformly choose $\epsilon$ sufficiently small so that $\tilde{d}(z, z_2) < r$. We want to show that there exists a uniform constant $c > 1$ such that
\begin{equation}
\label{eqn:FoliationsAreCompatible}
c^{-1}d(\Gamma z, \Gamma z_2) \le d(\Gamma z, \Gamma z_1) \le cd(\Gamma z, \Gamma z_2).
\end{equation}
By the inequality of \eqref{eqn:BiLipschitzEquivalence}, for each $j \in \{1,2\}$, we have
\begin{align*}
c_1^{-1}(d_{\Fboundary}(x, x') + |t - t_i|) \le \tilde{d}(z, z_i) \le c_1(d_{\Fboundary}(x, x') + |t - t_i|).
\end{align*}
By \cref{eqn:StrongUnstableAndStrongStableFoliations}, we have
\begin{align*}
|t - t_1| = \bigl|\psi\bigl(\mathcal{G}^\Theta[x, y] - \mathcal{G}^\Theta[x', y]\bigr)\bigr| \le c_2d_{\Fboundary}(x, x')
\end{align*}
for a uniform constant $c_2 > 0$ by smoothness of the $\Theta$-Gromov product restricted to $D'$. Then
\begin{align*}
d(\Gamma z, \Gamma z_1) \le c_1(d_{\Fboundary}(x, x') + |t - t_1|) \le c_1(1+c_2)d_{\Fboundary}(x, x') \le c_1^2(1+c_2)d(\Gamma z, \Gamma z_2).
\end{align*}
This establishes the right hand side inequality of \eqref{eqn:FoliationsAreCompatible}. 

For the left hand side inequality of \eqref{eqn:FoliationsAreCompatible}, we observe that compactness of $D$ and the Lipschitz properties of the maps $\hat{u}$ and $\mathcal{H}$ in the proof of \cref{thm:TranslationFlowConjugateToRhoFlow} implies that there exists a uniform constant $c_3 > 0$ such that
\begin{align*}
|t - t_2| = |\log \hat{u}(\mathcal{H}^{-1})(z) - \log \hat{u}(\mathcal{H}^{-1})(z_2)| \le c_3\tilde{d}(z, z_2).
\end{align*}
This establishes \cref{itm:FoliationProperty} for the strong unstable lamination. The argument for the strong stable lamination is similar. 

\medskip
\noindent
\textit{Proof of \cref{itm:KappaProperty}.}
\Cref{itm:KappaProperty} is immediate from the fact that $\kappa(\cdot,t) = \int_0^t f(a_{\rho, s}(\cdot)) \, ds$ for all $t \in \R$ for some positive continuous function $f:\mathcal{G}_\rho \to \R$ (see \cite[Remark 4.11]{CS23}).
\end{proof}

\begin{remark}
\label{rem:ConjugatingTranslationFlows}
Although we have only shown that the distinguished translation flow $\{a_t\}_{t\in\R}$ on $\mathcal{X}$ is metric Anosov, it is not difficult to see that the above work can be adapted to show that any translation flow $\{a^{\sfv}_t\}_{t \in \R}$ with $\sfv \in \interior\limitcone$ on $\mathcal{X}$ is a Lipschitz reparametrization of $\{a_t\}_{t\in\R}$ and also metric Anosov. Indeed, the proof of \cref{thm:TranslationFlowConjugateToRhoFlow,pro:ReparametrizationOfTranslationFlowIsMetricAnosov} is easily modified to show that for all $\sfv \in \interior\limitcone$ there exists a bi-Lipschitz homeomorphism $\mathcal{H}^{\sfv_0,\sfv}: \mathcal{X} \to \mathcal{X}_{\sfv}$ conjugating $\{a^{\sfv}_t\}_{t\in\R}$ to a Lipschitz reparametrization of $\{a_t\}_{t\in\R}$ and that $\{a^{\sfv}_t\}_{t\in\R}$ is metric Anosov. In fact, such a homeomorphism $\mathcal{H}^{\sfv_0, \sfv}$ can be explicitly described as follows. Let $\pi_{\ker\psi}: \LieA_\Theta \to \ker\psi$ be the projection map determined by the decomposition $\LieA_\Theta = \R\sfv_0\oplus\ker\psi$, $\mathcal{U} = \{\gamma \mathcal{U}_j\}_{\gamma \in \Gamma, \, 1 \le j \le j_0}$ be a locally finite open cover of $\limitset^{(2)} \times \R$ such that $\gamma \mathcal{U}_j \cap \gamma' \mathcal{U}_j = \varnothing$ for all distinct $\gamma, \gamma' \in \Gamma$ and $1 \le j \le j_0$ and $\{\varphi_i\}_{i \in I}$ for some index set $I$ be a partition of unity subordinate to $\mathcal{U}$ such that $\varphi_i: \limitset^{(2)} \times \R \to [0, +\infty)$ is smooth along the $\rho$-geodesic flow for all $i \in I$ and $\{\varphi_i\}_{i \in I}$ be a partition of unity subordinate to the cover $\{\gamma U_j\}_{\gamma \in \Gamma, 1 \le i \le j_0}$. Define the map $\tilde{\mathcal{H}}^{\sfv_0,\sfv} : \limitset^{(2)} \times \R \to \limitset^{(2)} \times \R$ by $\tilde{\mathcal{H}}^{\sfv_0, \sfv}(x,y,r) = (x,y,r\psi_{\sfv}({\sfv}_0)+\log\hat{u}(x,y,r))$ where $\hat{u}$ is given by
\begin{align*}
\frac{1}{\hat{u}(z)} = \int_0^T \frac{e^{s/T}}{\exp\left(\sum_{i \in I, \, \supp \varphi_i \subset \gamma U_j}(\psi_{\sfv}\circ\pi_{\ker\psi})(\beta^\Theta_x(e,\gamma)) \cdot \varphi_i(x,y,r+t)\right)} \, ds
\end{align*}
for all $z = (x,y,r) \in \limitset^{(2)} \times \R$, for some $T >0$. Then, it descends to the desired map $\mathcal{H}^{\sfv_0, \sfv}:\mathcal{X} \to \mathcal{X}_{\sf{v}}$.
\end{remark}

\section{Transfer operators}
\label{sec:TransferOperators}
In this section, we cover the necessary background for transfer operators in order to state \cref{thm:SpectralBoundOnTransferOperator} which is the main technical theorem regarding spectral bounds. We also outline how to derive the main theorem stated in \cref{thm:ExponentialMixingOnXWithPR-Resonances} from \cref{thm:SpectralBoundOnTransferOperator}. A key point in this section is that all the theorems include the precise dependence on a parameter associated to $\sfv \in \interior\limitcone$.

\subsection{Reduction of \texorpdfstring{\cref{thm:ExponentialMixingOnXWithPR-Resonances}}{\autoref{thm:ExponentialMixingOnXWithPR-Resonances}} by rescaling}
Due to the homothety equivariance property for the family of translation flows (see \cref{eqn:HomothetyEquivariance}), it is convenient to fix a scaling for each direction in $\interior\limitcone$ for the purpose of proving \cref{thm:ExponentialMixingOnXWithPR-Resonances}. It turns out that there is a particular scaling so that we have additional properties. Using the isomorphism $\LieA_\Theta \cong \LieA_\Theta^*$ induced by the inner product $\langle \cdot, \cdot\rangle$ on $\LieA_\Theta$, we abuse notation and identify the dual limit cone $\limitcone^* \subset \LieA_\Theta^*$ with a dual limit cone $\limitcone^* \subset \LieA_\Theta$ (see \cref{eqn:DualLimitCone}). For convenience, we write $\ker\Theta := \bigcup_{\alpha \in \Theta} \ker\alpha$ which contains $\partial\LieA_\Theta^+$. Note that $\limitcone \subset \interior\LieA_\Theta^+ \cup \{0\}$ (see \cref{thm:BasicProperties}\labelcref{itm:BasicProperties2}) and non-obtuseness of $\LieA_\Theta^+$ (see \cite[Chapter II, \S\,5, Proposition 2.48(e)]{Kna96}) implies that $\partial\LieA_\Theta^+ \subset \interior\limitcone^* \cup \{0\}$. For all $\sfw \in \interior\limitcone^*$ with $\|\sfw\| = 1$, we define $\sfv(\sfw) \in \interior\limitcone$ to be the unique vector such that $\nabla\psi_{\sfv(\sfw)} = \sfw$. Then, \cref{eqn:GeneralLinearForm} gives
\begin{align}
\label{eqn:GradientofGrowthindicatorEqualsGrowthindicator}
\frac{\|\nabla\growthindicator(\sfv(\sfw))\|}{\growthindicator(\sfv(\sfw))} = 1.
\end{align}
Note also that for an arbitrary vector  $\sfv \in \interior\limitcone$, if it is scaled to a vector $\hat{\sfv} \in \interior\limitcone$ of the above form so that $\sfv = \kappa \hat{\sfv}$, then $\kappa = \frac{\growthindicator(\sfv)}{\|\nabla\growthindicator(\sfv)\|}$. Therefore, \cref{thm:ExponentialMixingOnXWithPR-Resonances} follows from the following exponential mixing theorem. \Cref{thm:UniformExponentialMixingOnX} also follows either directly from the following or via \cref{thm:ExponentialMixingOnXWithPR-Resonances}. Here, $L(\mathcal{X})$ denotes the space of real-valued Lipschitz continuous functions on $\mathcal{X}$ equipped with the Lipschitz norm $\|\cdot\|_{\Lip}$; we use similar notations for other domains and indicate the codomain for complex-valued functions.

\begin{theorem}
\label{thm:RescaledExponentialMixingOnXWithPR-Resonances}
For all open neighborhoods $\mathcal{N} \supset \ker\Theta$, the following holds. Let $\sfv \in \interior\limitcone$ with $\|\nabla\psi_{\sfv}\| = 1$ and $\nabla\psi_{\sfv} \notin \mathcal{N}$. There exist $\eta_\sfv > 0$ which is continuous in $\sfv$ and decays exponentially in $\delta_\sfv$ to $0$, $C_\sfv > 0$ which is continuous in $\sfv$, $k_\sfv \in \N$, a finite set $\{\mu_{k, \sfv}\}_{k = 1}^{k_\sfv} \subset (-\eta_\sfv, 0) + i[-1, 1]$ which come in conjugate pairs, and a finite set of finite-rank positive semi-definite bilinear forms $\{\mathcal{B}_{k, \sfv}\}_{k = 1}^{k_\sfv}$, such that for all $\phi_1,\phi_2 \in L(\mathcal{X})$ and $t > 0$, we have
\begin{multline*}
\left|\int_{\mathcal{X}} \phi_1(a^{\sfv}_tx)\phi_2(x) \, d\BMSX^{\sfv}(x) - \left(\BMSX^{\sfv}(\phi_1)  \BMSX^{\sfv}(\phi_2) + \sum_{k = 1}^{k_\sfv} e^{\mu_{k, \sfv} t}\mathcal{B}_{k, \sfv}(\phi_1, \phi_2)\right)\right| \\
\leq C_\sfv e^{-\eta_\sfv t} \|\phi_1\|_{\Lip} \|\phi_2\|_{\Lip}.
\end{multline*}
\end{theorem}

The rest of the paper is devoted to the proof of the above theorem. To this end, we fix an open neighborhood $\mathcal{N} \supset \ker\Theta$ henceforth and define the subset of unit vectors
\begin{align*}
\ww = \{\sfw \in \interior\limitcone^*: \|\sfw\| = 1,  \sfw \notin \mathcal{N}\}.
\end{align*}
Note that we have a corresponding compact subset of unit vectors
\begin{align*}
\overline{\ww} = \{\sfw \in \limitcone^*: \|\sfw\| = 1,  \sfw \notin \mathcal{N}\}.
\end{align*}

\subsection{Compatible family of Markov sections}
\label{subsec:CompatibleFamilyOfMarkovSections}
\Cref{cor:TranslationFlowHasMarkovSection} gives a Markov section for any fixed translation flow. Recalling that we have a family of translation flows $\bigl\{\bigl\{a^{\sfv(\sfw)}_t\bigr\}_{t \in \R}\bigr\}_{\sfw \in \ww}$, we fix a corresponding family of Markov sections $\{\mathcal{R}^{\sfw}\}_{\sfw \in \ww}$ on $\mathcal{X}$. They have several other corresponding objects (see \cref{subsec:MetricAnosovFlows}) all of which we denote using a superscript `$\sfw$' on the same notations. By the same observation as in \cite[\S\,4.2]{CS23} regarding Markov sections of reparametrized flows, we deduce the remarkable property that the family of Markov sections can be chosen such that for different unit vectors $\sfw, \sfw' \in \ww$, the Markov sections $\mathcal{R}^{\sfw}$ and $\mathcal{R}^{\sfw'}$ consist of corresponding rectangles which project to each other along the translation flow lamination while fixing their centers. Consequently, for some fixed $N \in \N$, we denote
\begin{align*}
\mathcal{R}^{\sfw} = \{R^{\sfw}_1, R^{\sfw}_2, \dotsc, R^{\sfw}_N\} = \{[U^{\sfw}_1, S^{\sfw}_1], [U^{\sfw}_2, S^{\sfw}_2], \dotsc, [U^{\sfw}_N, S^{\sfw}_N]\} \qquad \text{for all $\sfw \in \ww$}
\end{align*}
and the family $\{\mathcal{R}^{\sfw}\}_{\sfw \in \ww}$ comes equipped with a family
\begin{align*}
\{\Psi^{\sfw, \sfw'}: R^{\sfw} \to R^{\sfw'}\}_{\sfw, \sfw' \in \ww}
\end{align*}
of compatibility maps which are Lipschitz homeomorphisms and vary smoothly in $\sfw, \sfw' \in \ww$. More precisely, there exists a family $\{t^{\sfw, \sfw'}: R^\sfw \to \R\}_{\sfw, \sfw' \in \ww}$ of Lipschitz continuous functions which vary smoothly in $\sfw, \sfw' \in \ww$ such that $t^{\sfw, \sfw'}(w_j) = 0$ at the centers $w_j \in R^{\sfw}_j$ for all $1 \leq j \leq N$ and
\begin{align*}
\Psi^{\sfw, \sfw'}(u) = a^{\sfv(\sfw)}_{t^{\sfw, \sfw'}(u)} u \in R^{\sfw'} \qquad \text{for all $u \in R^{\sfw}$ and $\sfw, \sfw' \in \ww$}.
\end{align*}
All the properties are clear save injectivity which we now justify. For all $\sfw, \sfw' \in \ww$, the map $\Psi^{\sfw, \sfw'}$ moves the points in $R^\sfw$ along the translation flow lamination to which the strong stable and unstable laminations corresponding to both $\sfw, \sfw' \in \ww$ are transverse and hence preserves the time ordering along the translation flow lamination. Now, if $\Psi^{\sfw, \sfw'}(u_1) = \Psi^{\sfw, \sfw'}(u_2)$ for some distinct $u_1, u_2 \in R^\sfw$, then the time ordering would be violated since $u_1$ and $u_2$ must lie on the same translation flow orbit; implying injectivity.

We now use the compatibility maps $\Psi^{\sfw_0, \sfw}$ to identify the Markov sections $\mathcal{R}^\sfw$ for all $\sfw \in \ww$ with a fixed Markov section $\mathcal{R} := \mathcal{R}^{\sfw_0}$ of some size $\hat{\delta} \in (0, \epsilon_0)$, where $\sfw_0$ corresponds to $\sfv_0$, i.e., $\sfv(\sfw_0) \in \R_{>0}\sfv_0$. We need to address a subtle point. For any $\sfw \in \ww$ outside of some open neighborhood of $\sfw_0 \in \ww$, the corresponding Markov sections need not be of size $\hat{\delta}$. However, apart from that, they are legitimate Markov sections and come from the local product structures on the image $\Psi^{\sfw_0, \sfw}\bigl(\bigsqcup_{j \in \mathcal{A}}[W^{\mathrm{su}}_{\epsilon_0}(w_j), W^{\mathrm{ss}}_{\epsilon_0}(w_j)]\bigr)$, where $w_j \in R_j^{\sfw_0}$ are the centers for all $j \in \mathcal{A}$. Thus, we can drop all superscripts henceforth, except for $\{\tau^{\sfw}\}_{\sfw \in \ww}$ which we view as a family of first return time maps on the same set $U$; since they still vary in $\sfw$, they need to be distinguished. A useful property is that they remain positive, again due to preservation of time ordering. We record this and other useful properties in the following lemma (recall \cref{eqn:sup_inf_of_tau}).

\begin{lemma}
\label{lem:tau_uniformLipschitzNormBound}
The family $\{\tau^{\sfw}\}_{\sfw \in \ww}$ is smooth in $\sfw \in \ww$ and there exists $C_\tau > 0$ such that
\begin{enumerate}
\item $0 < \tau^{\sfw}(u) \leq \frac{C_\tau}{2\overline{\tau^{\sfw_0}}}\tau^{\sfw_0}(u) \leq \frac{C_\tau}{2}$ for all $u \in U$ and $\sfw \in \ww$,
\item $\max_{(j, k)}\Lip\bigl(\tau^{\sfw}|_{\mathtt{C}[j, k]}\bigr) \leq \frac{C_\tau}{2}$ for all $\sfw \in \ww$.
\end{enumerate}
Consequently, $\bigl\{\max_{(j, k)}\|\tau^{\sfw}|_{\mathtt{C}[j, k]}\|_{\Lip}\bigr\}_{\sfw \in \ww}$ is bounded above by $C_\tau$.
\end{lemma}

\begin{proof}
Positivity of $\tau^\sfw$ for all $\sfw \in \ww$, as mentioned above, is clear and so we prove the other properties. We use notations from \cref{subsec:SmoothExtensions} and terminology from \cref{subsec:SymbolicDynamics} to shorten the proof. Namely, we use the first return vector map $\mathsf{K}: U \to \LieA_\Theta$ descended from the one in \cref{subsec:SmoothExtensions} which is Lipschitz continuous on cylinders of length $1$. Then, we have the identity $\tau^\sfw = \psi_{\sfv(\sfw)} \circ \mathsf{K}$ for all $\sfw \in \ww$. Thus, the lemma follows from this and the property that $\|\nabla \psi_{\sfv(\sfw)}\| = 1$.
\end{proof}

Due to the above lemma, we conclude that the smooth family $\{\tau^{\sfw}\}_{\sfw \in \ww}$ can be extended to a smooth family of uniformly essentially Lipschitz (i.e., with a uniform Lipschitz constant on cylinders of length $1$) nonnegative first return time maps $\{\tau^{\sfw}\}_{\sfw \in \overline{\ww}}$.

Using a compatible family of Markov sections as constructed in this subsection is necessary to execute Dolgopyat's method in a uniform fashion for the family of translation flows. Although we cannot make sense of a translation flow on $\mathcal{X}$ or even transfer operators associated to any $\sfv \in \partial\limitcone$, the above uniform bounds and the extended family of first return time maps is used to obtain other Lipschitz bounds and a uniform version of LNIC to derive the precise form of the exponential decay rate in \cref{thm:SpectralBoundOnTransferOperator} for the family of transfer operators in $\sfw \in \ww$.

\subsection{Symbolic dynamics}
\label{subsec:SymbolicDynamics}
Let $\mathcal A = \{1, 2, \dotsc, N\}$ be the \emph{alphabet} of $\mathcal{R}$. Define the $N \times N$ \emph{transition matrix} $T$ by
\begin{align*}
T_{j, k} =
\begin{cases}
1, & \interior(R_j) \cap \mathcal{P}^{-1}(\interior(R_k)) \neq \varnothing \\
0, & \text{otherwise}
\end{cases}
\qquad
\text{for all $1 \leq j, k \leq N$}.
\end{align*}
The transition matrix $T$ is \emph{topologically mixing} as a consequence of \cite[Theorem 8.1]{CS23}, i.e., there exists $N_T \in \N$ such that $T^{N_T}$ consists only of positive entries.

\begin{remark}
In \cite[\S\,4.3]{CS23}, there is a minor inaccuracy---the transition matrix $T$ is not topologically mixing \emph{a priori}. However, it is topologically transitive by \cite[Proposition 6.3]{CS23}. We may assume that $T$ is aperiodic by inserting extra rectangles in the Markov section if necessary so that coprime primitive periods exist for $\mathcal{P}$. Then, near the end of the proof of \cite[Theorem 6.10]{CS23}, ``for any $n \geq N_T$'' should be replaced with ``for some coprime $n_1, n_2 \geq N_T$'' and still conclude $\theta \in \{0, 1\}$. With this correction, we can use the topological mixing result \cite[Theorem 8.1]{CS23}. As a consequence, $T$ is indeed topologically mixing.
\end{remark}

Define the spaces of bi-infinite and infinite \emph{admissible sequences}
\begin{align*}
\Sigma &= \{(\dotsc, x_{-1}, x_0, x_1, \dotsc) \in \mathcal A^{\Z}: T_{x_j, x_{j + 1}} = 1 \text{ for all } j \in \Z\}; \\
\Sigma^+ &= \{(x_0, x_1, \dotsc) \in \mathcal A^{\Z_{\geq 0}}: T_{x_j, x_{j + 1}} = 1 \text{ for all } j \in \Z_{\geq 0}\}
\end{align*}
respectively. We will use the term \emph{admissible sequences} for finite sequences as well in the natural way. For any fixed $\beta_0 \in (0, 1)$, we endow $\Sigma$ with the metric $d$ defined by $d(x, y) = \beta_0^{\inf\{|j| \in \Z_{\geq 0}: x_j \neq y_j\}}$ for all $x, y \in \Sigma$. We similarly endow $\Sigma^+$ with a metric which we also denote by $d$.

\begin{definition}[Cylinder]
For all admissible sequences $x = (x_0, x_1, \dotsc, x_k)$ of \emph{length} $\len(x) = k \in \Z_{\geq 0}$, we define its \emph{cylinder} of \emph{length} $\len(\mathtt{C}[x]) = k$ to be
\begin{align*}
\mathtt{C}[x] &:= \{u \in U: \sigma^j(u) \in \interior(U_{x_j}) \text{ for all } 0 \leq j \leq k\}.
\end{align*}
We denote cylinders simply by $\mathtt{C}$ (or other typewriter style letters) when we do not need to specify the admissible sequence. We also similarly define cylinders in $S$ and $R$.
\end{definition}

By a slight abuse of notation, let $\sigma$ also denote the shift map on $\Sigma$ or $\Sigma^+$. There exist natural continuous surjections $\eta: \Sigma \to R$ and $\eta^+: \Sigma^+ \to U$ defined by $\eta(x) = \bigcap_{j = -\infty}^\infty \overline{\mathcal{P}^{-j}(\interior(R_{x_j}))}$ for all $x \in \Sigma$ and $\eta^+(x) = \bigcap_{j = 0}^\infty \overline{\sigma^{-j}(\interior(U_{x_j}))}$ for all $x \in \Sigma^+$. Define $\hat{\Sigma} = \eta^{-1}(\hat{R})$ and $\hat{\Sigma}^+ = (\eta^+)^{-1}(\hat{U})$. Then the restrictions $\eta|_{\hat{\Sigma}}: \hat{\Sigma} \to \hat{R}$ and $\eta^+|_{\hat{\Sigma}^+}: \hat{\Sigma}^+ \to \hat{U}$ are bijective and satisfy $\eta|_{\hat{\Sigma}} \circ \sigma|_{\hat{\Sigma}} = \mathcal{P}|_{\hat{R}} \circ \eta|_{\hat{\Sigma}}$ and $\eta^+|_{\hat{\Sigma}^+} \circ \sigma|_{\hat{\Sigma}^+} = \sigma|_{\hat{U}} \circ \eta^+|_{\hat{\Sigma}^+}$. We can take $\beta_0 \in (0, 1)$ sufficiently close to $1$ so that $\eta$ and $\eta^+$ are Lipschitz continuous \cite[Lemma 2.2]{Bow73}.

Let $\sfw \in \overline{\ww}$. Since the horospherical foliations on $G$ are smooth, we conclude that $\tau^\sfw$ is Lipschitz continuous on cylinders of length $1$. Then the maps $(\tau^\sfw \circ \eta)|_{\hat{\Sigma}}$ and $(\tau^\sfw \circ \eta^+)|_{\hat{\Sigma}^+}$ are Lipschitz continuous and hence there exist unique Lipschitz extensions which, by abuse of notation, we also denote by $\tau^\sfw: \Sigma \to \R$ and $\tau^\sfw: \Sigma^+ \to \R$, respectively.

\subsection{Thermodynamics}
\label{subsec:Thermodynamics}
\begin{definition}[Pressure]
\label{def:Pressure}
For all $f \in L(\Sigma)$, called the \emph{potential}, the \emph{pressure} is defined by
\begin{align*}
\Pr_\sigma(f) = \sup_{\nu \in \mathcal{M}^1_\sigma(\Sigma)}\left\{\int_\Sigma f \, d\nu + h_\nu(\sigma)\right\}
\end{align*}
where $\mathcal{M}^1_\sigma(\Sigma)$ is the set of $\sigma$-invariant Borel probability measures on $\Sigma$ and $h_\nu(\sigma)$ is the measure theoretic entropy of $\sigma$ with respect to $\nu$.
\end{definition}

For all $f \in L(\Sigma)$, there exists a unique $\sigma$-invariant Borel probability measure $\nu_f$ on $\Sigma$ which attains the supremum in \cref{def:Pressure}, called the \emph{$f$-equilibrium state} \cite[Theorems 2.17 and 2.20]{Bow08}. It satisfies $\nu_f(\hat{\Sigma}) = 1$ \cite[Corollary 3.2]{Che02}.

Let $\sfw \in \ww$. Associated to $\psi_{\sfv(\sfw)}$, we relabel $\delta_\sfw := \delta_{\sfv(\sfw)} = \delta_{\psi_{\sfv(\sfw)}}$ (see \cref{eqn:psi_TopologicalEntropy}) and recall from \cref{subsec:TranslationFlows} that $\delta_\sfw = \growthindicator(\sfv(\sfw)) = \|\nabla\growthindicator(\sfv(\sfw))\|$ where we have used \cref{eqn:GradientofGrowthindicatorEqualsGrowthindicator}. Thanks to \cref{cor:TranslationFlowHasMarkovSection}, we can use \cite[Theorem A.2]{Car23} which states that $\BMSX^{\sfv(\sfw)}$ is a measure of maximal entropy for the translation flow which attains the maximal entropy of $\delta_\sfw$. We will consider in particular the probability measure $\nu_{-\delta_\sfw\tau^\sfw}$ on $\Sigma$ with corresponding pressure $\Pr_\sigma(-\delta_\sfw\tau^\sfw) = 0$ \cite[Proposition 3.1]{BR75} (cf. \cite[Theorem 4.4]{Che02}), which we will denote simply by $\nu_\Sigma^{\sfw}$. Define $\nu_R^{\sfw} = \eta_*(\nu_\Sigma^{\sfw})$ and $\nu_U^{\sfw} = (\proj_U)_*(\nu_R^{\sfw})$. Note that $\nu_U^{\sfw}(\tau^\sfw) = \nu_R^{\sfw}(\tau^\sfw)$. See \cref{subsec:ProofOutlineSpectralBoundsToExponentialMixing} for the relation between $\BMSX^{\sfv(\sfw)}$ and $\nu_R^{\sfw}$ (cf. \cite[\S\,4.4]{CS23}).

\subsection{Transfer operators}
\label{subsec:TransferOperators}
Throughout the rest of the paper, we will use the notation
\begin{align*}
\xi = a + ib \in \C
\end{align*}
for the complex parameter for the transfer operators.

\begin{definition}[Transfer operators]
For all $\xi \in \C$ and $\sfw \in \overline{\ww}$, the \emph{transfer operator} $\mathcal{L}_{\xi\tau^\sfw}: C\bigl(U, \C\bigr) \to C\bigl(U, \C\bigr)$ is defined by
\begin{align*}
\mathcal{L}_{\xi\tau^\sfw}(H)(u) = \sum_{u' \in \sigma^{-1}(u)} e^{\xi\tau^\sfw(u')}H(u')
\end{align*}
for all $u \in U$ and $H \in C\bigl(U, \C\bigr)$.
\end{definition}

We recall the Ruelle--Perron--Frobenius (RPF) theorem along with the theory of Gibbs measures in this setting \cite{Bow08,PP90}. For all $a \in \R$ and $\sfw \in \ww$, there exist a unique positive function $h_{a, \sfw} \in L(U)$ and a unique Borel probability measure $\nu_{a, \sfw}$ on $U$ such that $\nu_{a, \sfw}(h_{a, \sfw}) = 1$ and
\begin{align*}
\mathcal{L}_{-(\delta_\sfw + a)\tau^\sfw}(h_{a, \sfw}) &= \lambda_{a, \sfw} h_{a, \sfw}, & \mathcal{L}_{-(\delta_\sfw + a)\tau^\sfw}^*(\nu_{a, \sfw}) &= \lambda_{a, \sfw} \nu_{a, \sfw}
\end{align*}
where $\lambda_{a, \sfw} := e^{\Pr_\sigma(-(\delta_\sfw + a)\tau^\sfw)}$ is the maximal simple eigenvalue of $\mathcal{L}_{-(\delta_\sfw + a)\tau^\sfw}$ and the rest of the spectrum of $\mathcal{L}_{-(\delta_\sfw + a)\tau^\sfw}|_{L(U, \C)}$ is contained in a disk of radius strictly less than $\lambda_{a, \sfw}$. Moreover, $d\nu_U^{\sfw} = h_{0, \sfw} \, d\nu_{0, \sfw}$ and $\lambda_{0, \sfw} = 1$ (see \cref{subsec:Thermodynamics}).

Let $\sfw \in \ww$. As usual, it is convenient to normalize the transfer operators. For all $a \in \R$ define the map
\begin{align*}
\tau^{\sfw, (a)} = (\delta_\sfw + a)\tau^\sfw + \log \circ h_{0, \sfw} \circ \sigma - \log \circ h_{0, \sfw} \in L(U).
\end{align*}
For all $k \in \N$, we define the Lipschitz continuous maps $\tau_k^{\sfw}: U \to \R$ and $\tau_k^{\sfw, (a)}: U \to \R$ by
\begin{align*}
\tau_k^{\sfw}(u) &= \sum_{j = 0}^{k - 1} \tau^\sfw(\sigma^j(u)), & \tau_k^{\sfw, (a)}(u) &= \sum_{j = 0}^{k - 1} \tau^{\sfw, (a)}(\sigma^j(u))
\end{align*}
and $\tau_0^{\sfw}(u) = \tau_0^{(a), \sfw}(u) = 0$ for all $u \in U$. For all $\xi \in \C$, we define $\mathcal{L}_{\xi, \sfw}: C(U, \C) \to C(U, \C)$ and its $k$\textsuperscript{th} iterates for all $k \in \Z_{\geq 0}$ by
\begin{align*}
\mathcal{L}_{\xi, \sfw}^k(H)(u) = \sum_{u' \in \sigma^{-k}(u)} e^{-(\tau_k^{\sfw, (a)} + ib\tau_k^{\sfw})(u')} H(u')
\end{align*}
for all $u \in U$ and $H \in C(U, \C)$. With this normalization, for all $a \in \R$, the maximal simple eigenvalue of $\mathcal{L}_{a, \sfw}$ is $\lambda_{a, \sfw}$ with eigenvector $\frac{h_{a, \sfw}}{h_{0, \sfw}}$. Moreover, we have $\mathcal{L}_{0, \sfw}^*(\nu_U^{\sfw}) = \nu_U^{\sfw}$.

We fix some related constants. Let $\sfw \in \ww$.
By perturbation theory of operators as in \cite[Chapter 7]{Kat95} and \cite[Proposition 4.6]{PP90}, we can fix $a_0' > 0$ such that the map $[-a_0', a_0'] \to \R$ defined by $a \mapsto \lambda_{a, \sfw}$ and the map $[-a_0', a_0'] \to C(U)$ defined by $a \mapsto h_{a, \sfw}$ are Lipschitz uniformly in $\sfw \in \ww$. To see uniformity in $\sfw \in \ww$, a standard calculation using the eigenvalue equation and \cref{lem:tau_uniformLipschitzNormBound} gives $\bigl|\left.\frac{d}{da'}\right|_{a' = a}\lambda_{a', \sfw}\bigr| = |\nu_{a, \sfw}(\tau^\sfw h_{a, \sfw})| \leq C_\tau$. Similar estimates in the construction of the family of eigenvectors $\{h_{a, \sfw}\}_{a \in [-a_0', a_0']}$ (see \cite[Theorem 2.2]{PP90}) gives the latter. Fix $A_\tau > 0$ such that it is greater that the aforementioned Lipschitz constants and so that $|\tau^{\sfw, (a)}(u) - \tau^{\sfw, (0)}(u)| \leq A_\tau|a|$ for all $u \in U$ and $|a| \leq a_0'$.
Again by similar estimates, we also have $\|h_{a, \sfw}\|_{\Lip} \ll \delta_\sfw = \|\nabla\growthindicator(\sfv(\sfw))\|$ which tends to $+\infty$ as $\sfw$ tends to the boundary of $\interior\limitcone^*$. Fix
\begin{align*}
T_\sfw >{}&\max\bigg(\max_{\sfw \in \overline{\ww}} \max_{(j, k)} \|\tau^\sfw|_{\mathtt{C}[j, k]}\|_{\Lip}, \sup_{|a| \leq a_0'} \max_{(j, k)} \bigl\|\tau^{\sfw, (a)}|_{\mathtt{C}[j, k]}\bigr\|_{\Lip}\bigg)
\end{align*}
which can be chosen as some elementary function evaluated at $\delta_\sfw = \|\nabla\growthindicator(\sfv(\sfw))\|$, for all $\sfw \in \ww$ (see also \cite[Lemma 4.1]{PS16} for taking the supremum over $[-a_0', a_0']$).

\section{Dolgopyat's Method and the proofs of \texorpdfstring{\cref{thm:RescaledExponentialMixingOnXWithPR-Resonances}}{\autoref{thm:RescaledExponentialMixingOnXWithPR-Resonances}} and applications}
\label{sec:DolgopyatsMethodAndProofs}
In this section, we state \cref{thm:SpectralBoundOnTransferOperator} on spectral bounds of transfer operators which is deduced from the more elaborate \cref{thm:DolgopyatsMethod}. We then outline the proofs that \cref{thm:SpectralBoundOnTransferOperator} implies our main theorems.

\subsection{Dolgopyat's Method}
\label{subsec:DolgopyatsMethod}
Recall that $L(U,\C)$ denotes the space of complex Lipschitz continuous functions on $U$. It is a Banach space with the Lipschitz seminorm and norm
\begin{align*}
\Lip(H) &= \sup_{u,u' \in U, u \ne u'} \frac{|H(u)-H(u')|}{d(u,u')}, & \|H\|_{\Lip} &= \|H\|_\infty +\Lip(H)
\end{align*}
for all $H \in L(U,\C)$, where $\|\cdot\|_\infty$ denotes the usual $L^\infty$-norm. For all $b \in \R$, we also generalize the above norm to
\begin{align*}
\|H\|_{1,b} = \|H\|_\infty + \frac{1}{\max(1,|b|)}\Lip(H)
\end{align*}
for all $H \in L(U,\C)$, which will be required later.

Following Stoyanov \cite[\S\,5]{Sto11}, we define the convenient new metric $D$ on $U$ by
\begin{align*}
D(u, u') =
\begin{cases}
0, & u = u' \\
\displaystyle\min_{\substack{u, u' \in \mathtt{C}\\ \text{cylinder } \mathtt{C}}} \diam(\overline{\mathtt{C}}), & u \neq u' \text{ and } u, u' \in U_j \text{ for some } j \in \mathcal{A} \\
1, & \text{otherwise}.
\end{cases}
\end{align*}
We denote by $L_D(U)$ the space of real-valued $D$-Lipschitz continuous functions and $\Lip_D(h)$ to be the $D$-Lipschitz constant for all $h \in L_D(U)$. We define the cones
\begin{align*}
\mathcal{C}_B(U) &= \{h \in L_D(U): h > 0, |h(u) - h(u')| \leq Bh(u)D(u, u') \text{ for all } u, u' \in U\} \\
\tilde{\mathcal{C}}_B(U) &= \Big\{h \in L_D(U): h > 0, e^{-B D(u, u')} \leq \frac{h(u)}{h(u')} \leq e^{B D(u, u')} \text{ for all } u, u' \in U\Big\}.
\end{align*}

\begin{remark}
If $h \in \mathcal{C}_B(U)$, then using the convexity of $-\log$, we can derive that $h$ is $\log$-$D$-Lipschitz, i.e., $|(\log \circ h)(u) - (\log \circ h)(u')| \leq B D(u, u')$ for all $u, u' \in U$. It follows that $\mathcal{C}_B(U) \subset \tilde{\mathcal{C}}_B(U)$, however the reverse containment is not true.
\end{remark}

We now state \cref{thm:SpectralBoundOnTransferOperator} regarding spectral bounds of transfer operators. Such bounds are the key for obtaining exponential mixing results.

\begin{theorem}
\label{thm:SpectralBoundOnTransferOperator}
Let $\sfw \in \ww$. There exist $\eta_\sfw > 0$ and $a_\sfw > 0$ both of which are continuous in $\sfw$ and decay exponentially in $\delta_\sfw$ to $0$, and $b_0 > 0$ such that for all $\xi \in \C$ with $|a| < a_\sfw$ and $|b| > b_0$, $k \in \N$, and $H \in L(U,\C)$, we have
\begin{align*}
\big\|\mathcal{L}_{\xi, \sfw}^k(H)\big\|_2 \ll_\sfw e^{-\eta(\delta_\sfw) k} \|H\|_{1, b}
\end{align*}
where the implicit constant is continuous in $\sfw$.
\end{theorem}

By a standard inductive argument (see the proof after \cite[Theorem 5.4]{SW21}) which goes back to Dolgopyat \cite{Dol98}, \cref{thm:SpectralBoundOnTransferOperator} is a consequence of the following theorem which records the mechanism of Dolgopyat's method.

\begin{theorem}
\label{thm:DolgopyatsMethod}
There exist $m \in \N$, a continuous function $\eta: \R_{>0} \to (0, 1)$ which is increasing to $1$ at an exponential rate, a continuous function $a_0: \R_{>0} \to \R_{>0}$ which decays exponentially to $0$, $b_0 > 0$, $E > \max\left(1, \frac{1}{b_0}\right)$, and a set of operators
\begin{align*}
\{\mathcal{N}_{a, J}^{\sfw}: L_D(U) \to L_D(U): \sfw \in \ww, |a| < a_0(\delta_\sfw), J \in \mathcal{J}(b) \text{ for some } |b| > b_0\},
\end{align*}
where $\mathcal{J}(b)$ is some finite set for all $|b| > b_0$, such that
\begin{enumerate}
\item\label{itm:DolgopyatProperty1}	$\mathcal{N}_{a, J}^{\sfw}(\mathcal{C}_{E|b|}(U)) \subset \mathcal{C}_{E|b|}(U)$ for all $\sfw \in \ww$, $|a| < a_0(\delta_\sfw)$, $J \in \mathcal{J}(b)$, and $|b| > b_0$;
\item\label{itm:DolgopyatProperty2}	$\left\|\mathcal{N}_{a, J}^{\sfw}(h)\right\|_2 \leq \eta(\delta_\sfw) \|h\|_2$ for all $h \in \mathcal{C}_{E|b|}(U)$, $\sfw \in \ww$, $|a| < a_0(\delta_\sfw)$, $J \in \mathcal{J}(b)$, and $|b| > b_0$;
\item\label{itm:DolgopyatProperty3}	for all $|a| < a_0(\delta_\sfw)$, $|b| > b_0$, and $\sfw \in \ww$, if $H \in L(U,\C)$ and $h \in \mathcal{C}_{E|b|}(U)$ satisfy
\begin{enumerate}[label=(1\alph*), ref=\theenumi(1\alph*)]
\item\label{itm:DominatedByh}	$|H(u)| \leq h(u)$ for all $u \in U$;
\item\label{itm:LogLipschitzh}	$|H(u) - H(u')| \leq E|b|h(u)D(u, u')$ for all $u, u' \in U$;
\end{enumerate}
then there exists $J \in \mathcal{J}(b)$ such that
\begin{enumerate}[label=(2\alph*), ref=\theenumi(2\alph*)]
\item\label{itm:DominatedByDolgopyat}	$\bigl|\mathcal{L}_{\xi, \sfw}^m(H)(u)\bigr| \leq \mathcal{N}_{a, J}^{\sfw}(h)(u)$ for all $u \in U$;
\item\label{itm:LogLipschitzDolgopyat}	$\bigl|\mathcal{L}_{\xi, \sfw}^m(H)(u) - \mathcal{L}_{\xi, \sfw}^m(H)(u')\bigr| \leq E|b|\mathcal{N}_{a, J}^{\sfw}(h)(u)D(u, u')$ for all $u, u' \in U$.
\end{enumerate}
\end{enumerate}
\end{theorem}

Proofs for the fact that \cref{thm:SpectralBoundOnTransferOperator} implies \cref{thm:RescaledExponentialMixingOnXWithPR-Resonances} (and hence \cref{thm:ExponentialMixingOnXWithPR-Resonances}), \cref{thm:EssentialSpectralGap}, and \cref{thm:ExponentialPrimeOrbitTheorem}, essentially exist throughout the literature. In the rest of this section, we collect the references and give proof outlines. Consequently, in subsequent sections of the paper, we focus only on the proof of \cref{thm:DolgopyatsMethod}.

\subsection{Proof outline of \cref{thm:SpectralBoundOnTransferOperator} implies \cref{thm:RescaledExponentialMixingOnXWithPR-Resonances}}
\label{subsec:ProofOutlineSpectralBoundsToExponentialMixing}
\Cref{thm:RescaledExponentialMixingOnXWithPR-Resonances} for Lipschitz continuous functions is derived from the spectral bounds of transfer operators, \cref{thm:SpectralBoundOnTransferOperator}, using arguments which go back to Pollicott \cite{Pol85} together with a Paley--Wiener type theorem. We refer the reader to \cite[\S\,7 and 8]{AGY06} for a clean framework with all the details; see also \cite[\S\,10]{SW21} and \cite[Chapter 6]{Sar22b} (cf. \cite[\S\,9]{LPS23} for a Lipschitz version). Inspired by \cite{Pol85}, one extra feature in the current paper in contrast with the previous papers is that we include the Pollicott--Ruelle resonances since we cannot control the simple eigenvalues of the transfer operators $\mathcal{L}_{\xi, \sfw}$ for the complex parameter $\xi$ in a fixed neighborhood of $0 \in \C$ \emph{uniformly} in $\sfw \in \ww$. We now outline the derivation and refer the reader to the above references for the omitted proofs. We set $\sfw \in \ww$ and decorations by $\sfv(\sfw)$ will simplified to decorations by $\sfw$.

\subsubsection{Approximating the test functions and their correlation function}
First we define the suspension spaces
\begin{align*}
U^{\tau^\sfw} &= (U \times \R_{\geq 0})/\sim, & R^{\tau^\sfw} &= (R \times \R_{\geq 0})/\sim,
\end{align*}
where $\sim$ is an equivalence relation defined by $(u, r + \tau^\sfw(u)) \sim (\sigma(u), r)$ for all $(u, r) \in U \times \R_{\geq 0}$ and similarly $(u, r + \tau^\sfw(u)) \sim (\mathcal{P}(u), r)$ for all $(u, r) \in R \times \R_{\geq 0}$. Let $\nu^{\tau^\sfw}$ be the product measure $\nu_R^\sfw \times \mathrm{Leb}_\R$ on $R^{\tau^\sfw}$. Then, $\BMSX^\sfw = \frac{\nu^{\tau^\sfw}}{\nu_U^\sfw(\tau^\sfw)}$; see \cite[\S\,4.4]{CS23}.

By Rokhlin's disintegration theorem with respect to the projection $\proj_U: R \to U$, the probability measure $\nu_R^\sfw$ disintegrates into a set of probability measures $\{\nu_u^\sfw\}_{u \in U}$. For all $j \in \mathcal{A}$ and $u \in U_j$, we may view $\nu_u^\sfw$ as a measure on $S_j$ by pushing forward by the map $[u, S_j] \to S_j$ given by $[u, s] \mapsto s$.

Let $\varsigma > 0$. We introduce the integration linear map $\mathcal{I}_\varsigma: L(\mathcal{X}) \to L^\infty\bigl(R^{\tau^\sfw}\bigr)$ defined as follows: for all test functions $\phi \in L(\mathcal{X})$, $j \in \mathcal{A}$, $[u, s] \in \interior(R_j)$, and $r \in [0, \tau^\sfw(u)]$, we put
\begin{align*}
\mathcal{I}_\varsigma(\phi)([u, s], r + \varsigma) = \int_{S_j} \phi\bigl(a^\sfw_{r + \varsigma}[u, s']\bigr) \, d\nu_u^\sfw(s')
\end{align*}
and extend it arbitrarily to $R^{\tau^\sfw}$. Then, in the open flow boxes $\{a^\sfw_{r + \varsigma}(u): u \in \interior(R_j), r \in (0, \tau^\sfw(u))\}$ for all $j \in \mathcal{A}$, the function $\mathcal{I}_\varsigma(\phi)$ is (by construction) constant along the $S$-coordinate.
Note also that $\mathcal{I}_\varsigma(\phi)$ approximates $\phi$; more precisely, there exists $\eta > 0$ such that
\begin{align}
\label{eqn: Iphi approximates phi}
\|\mathcal{I}_\varsigma(\phi) - \phi\|_\infty \ll e^{-(\eta/\overline{\tau^\sfw})\varsigma} \|\phi\|_{\Lip}.
\end{align}

It follows from definitions that for any $t \geq \varsigma$,
we may define a map $\mathcal{T}_{\varsigma, t}: L(\mathcal{X}) \to L^\infty\bigl(U^{\tau^\sfw}\bigr)$ by
\begin{align*}
\mathcal{T}_{\varsigma, t}(\phi) = \mathcal{I}_\varsigma(\phi) \circ a^\sfw_t
\end{align*}
where we abuse notation and view the right hand side as a function in $L^\infty\bigl(U^{\tau^\sfw}\bigr)$ since it is constant along $\interior(S_j)$ for all $j \in \mathcal{A}$. Moreover, for all $\varsigma \leq t < \varsigma + \underline{\tau^\sfw}$, the right hand side is in fact Lipschitz continuous along the $U$-coordinate in $\interior(U)$.

For all $\phi_1, \phi_2 \in L^\infty\bigl(U^{\tau^\sfw}\bigr)$, define the correlation function $\Upsilon_{\phi_1, \phi_2} \in L^\infty(\R_{\geq 0}, \R)$ by
\begin{align}
\label{eqn:CorrelationFunctinoOnSuspensionSpace}
\Upsilon_{\phi_1, \phi_2}(t) = \int_U \int_0^{\tau^\sfw(u)} \phi_1(u, r + t) \phi_2(u, r) \, dr \, d\nu_U^\sfw(u).
\end{align}
We have the following lemma by starting with the correlation function for the translation flow on $R^{\tau^\sfw}$ for test functions $\mathcal{I}_\varsigma(\phi_1)$ and $\mathcal{I}_\varsigma(\phi_2)$, and then using invariance of $\nu^{\tau^\sfw}$ under the translation flow.

\begin{lemma}
\label{lem: correlation functions equal}
For all $\varsigma > 0$ and $t'' \geq t' \geq \varsigma$ and $\phi_1, \phi_2 \in L(\mathcal{X})$, we have
\begin{align*}
\Upsilon_{\mathcal{T}_{\varsigma, t'}(\phi_1), \mathcal{T}_{\varsigma, t'}(\phi_2)}(t) = \Upsilon_{\mathcal{T}_{\varsigma, t''}(\phi_1), \mathcal{T}_{\varsigma, t''}(\phi_2)}(t) \qquad \text{for all $t \geq 0$}.
\end{align*}
\end{lemma}

Using the metric Anosov property from \cref{thm:TranslationFlowIsMetricAnosov}, we can derive the following lemma. 

\begin{lemma}
\label{lem:ApproximatingTheCorrelationFunction}
There exists $\eta > 0$ such that for all $\varsigma > 0$, $t' \geq \varsigma$, and $\phi_1, \phi_2 \in L(\mathcal{X})$, we have
\begin{multline*}
\left|\int_{\mathcal{X}} \phi_1(a^\sfw_t x) \phi_2(x) \, d\BMSX(x) - \frac{1}{\nu_U^\sfw(\tau^\sfw)}\Upsilon_{\mathcal{T}_{\varsigma, t'}(\phi_1), \mathcal{T}_{\varsigma, t'}(\phi_2)}(t)\right| \\
\ll e^{-(\eta/\overline{\tau^\sfw})\varsigma} \|\phi_1\|_{\Lip} \|\phi_2\|_{\Lip} \qquad \text{for all $t \geq 0$}.
\end{multline*}
\end{lemma}

\subsubsection{Laplace transform of the correlation function}
For all $\phi_1, \phi_2 \in L^\infty\bigl(U^{\tau^\sfw}\bigr)$, it is convenient to work with the the modified correlation function $\Upsilon^0_{\phi_1, \phi_2}$ defined as in \cref{eqn:CorrelationFunctinoOnSuspensionSpace} with the integration $\int_0^{\tau^\sfw(u)}$ replaced by $\int_{\max\{0, \tau^\sfw(u) - t\}}^{\tau^\sfw(u)}$ (which coincide as soon as $t \geq \overline{\tau^\sfw}$). Its Laplace transform $\Upsilon^0_{\phi_1, \phi_2}: (0, +\infty) + i\R \to \C$ is a holomorphic function defined by
\begin{align*}
\hat{\Upsilon}_{\phi_1, \phi_2}^0(\xi) = \int_0^{+\infty} e^{-\xi t} \Upsilon^0_{\phi_1, \phi_2}(t) \, dt \qquad \text{for all $\xi \in \C$ with $a > 0$}.
\end{align*}
For all $\phi \in L^\infty\bigl(U^{\tau^\sfw}\bigr)$ and $\xi \in \C$, we also define $\hat{\phi}_\xi \in L^\infty(U, \C)$ by
\begin{align*}
\hat{\phi}_\xi(u) = \int_0^{\tau^\sfw(u)} e^{-\xi t} \phi(u, t) \, dt \qquad \text{for all $u \in U$}.
\end{align*}
Suppose $\phi$ is Lipschitz continuous along the $U$-coordinate in $\interior(U)$. Even then, $\hat{\phi}_\xi$ would not be continuous on $\interior(U)$ because of using $\tau^\sfw$ in the integration; however, $\mathcal{L}_{\xi, \sfw}\hat{\phi}_\xi$ would be. Consequently, in this case, we may view $\mathcal{L}_{\xi, \sfw}\hat{\phi}_\xi \in L(U, \C)$ insofar as being a measurable function is concerned, which allows us later on to use the spectral bounds from \cref{thm:SpectralBoundOnTransferOperator}.

Let $\mathcal{B}(L(U, \C))$ be the Banach algebra of bounded operators on $L(U, \C)$. Define the holomorphic map $\mathsf{L}_{\bullet, \sfw}: (0, +\infty) + i\R \to \mathcal{B}(L(U, \C))$ by the Neumann series
\begin{align*}
\mathsf{L}_{\xi, \sfw} = \sum_{k = 1}^\infty \mathcal{L}_{\xi, \sfw}^k \qquad \text{for all $\xi \in \C$ with $a > 0$}.
\end{align*}

\begin{lemma}
For all $\phi_1, \phi_2 \in L^\infty\bigl(U^{\tau^\sfw}\bigr)$, we have
\begin{align*}
\hat{\Upsilon}^0_{\phi_1, \phi_2}(\xi) = \left\langle [\hat{\phi}_1]_\xi, \mathsf{L}_{\xi, \sfw} [\hat{\phi}_2]_{-\overline{\xi}}\right\rangle \qquad \text{for all $\xi \in \C$ with $a > 0$}.
\end{align*}
\end{lemma}

\subsubsection{Meromorphic extension and Paley--Wiener}
We have the following a priori bounds. This part of the argument requires integration by parts. Recall that we may do so for pairs of Lipschitz functions, and more generally for pairs of functions of bounded variation where one of them is continuous (cf. \cite{AGY06}).\footnote{Alternatively, we may initially work with functions which are smooth along the translation flow and then use a convolution argument as mentioned in the beginning of \cref{sec:TransferOperators}.}

\begin{lemma}
\label{lem: a priori bounds}
For all $\phi_1, \phi_2 \in L\bigl(U^{\tau^\sfw}\bigr)$ and $\xi \in \C$ with $|a| < a_0'$, we have
\begin{align*}
\|[\hat{\phi}_1]_\xi\|_2 &\leq \frac{\|\phi_1\|_{\Lip}}{\max(1, |b|)}, & \bigl\|\mathcal{L}_{\xi, \sfw} [\hat{\phi}_2]_{-\overline{\xi}}\bigr\|_{1, b} &\leq \frac{\|\phi_2\|_{\Lip}}{\max(1, |b|)}.
\end{align*}
\end{lemma}

We have the following result on meromorphic extension of the Neumann series of the transfer operator. It is proved using the characterization of the spectrum of transfer operators and perturbation theory of operators \cite{Pol86,PP90,Kat95}.

\begin{lemma}
\label{lem: meromorphic extension of transfer operators}
There exists $\eta_\sfw \in (0, 1)$ which is continuous in $\sfw$ such that $\mathsf{L}_{\bullet, \sfw}$ has a meromorphic extension
\begin{align*}
\mathsf{L}_{\bullet, \sfw}: (-\eta_\sfw, +\infty) + i\R \to \mathcal{B}(L(U, \C))
\end{align*}
and for all $\xi_0 \in (-\eta_\sfw, +\infty) + i\R$, there exists an open neighborhood $\mathcal{O}_{\xi_0} \subset \C$ of $\xi _0$ such that $\mathsf{L}_{\bullet, \sfw}$ has the following form on $\mathcal{O}_{\xi_0}$: there exist holomorphic maps
\begin{itemize}
\item $\mathcal{O}_{\xi_0} \to \C$ denoted by $\xi \mapsto \lambda_{\xi, \sfw}$, which gives the maximal eigenvalue of $\mathcal{L}_{\xi, \sfw}$,
\item $\mathcal{O}_{\xi_0} \to \mathcal{B}(L(U, \C))$ denoted by $\xi \mapsto \mathsf{P}_{\xi, \sfw}$, which gives the projection operator onto the finite dimensional $\lambda_{\xi, \sfw}$-eigenspace in $L(U, \C)$,
\item $\mathcal{O}_{\xi_0} \to \mathcal{B}(L(U, \C))$ denoted by $\xi \mapsto \mathsf{Q}_{\xi, \sfw}$, which gives an operator commuting with $\mathsf{P}_{\xi, \sfw}$,
\end{itemize}
such that
\begin{align*}
\mathsf{L}_{\xi, \sfw} = \frac{1}{1 - \lambda_{\xi, \sfw}} \mathsf{P}_{\xi, \sfw} + \mathsf{Q}_{\xi, \sfw} \qquad \text{for all $\xi \in \mathcal{O}_{\xi_0}$}.
\end{align*}
\end{lemma}

Since $\lambda_{\bullet, \sfw}$ is not constant on the compact set $[-1, 0] + i[-1, 1]$, there are finitely many $\xi \in (-\eta_\sfw, 0] + i[-1, 1]$ such that $\lambda_{\xi, \sfw} = 1$, and thereby contributes a pole for $\mathsf{L}_{\bullet, \sfw}$. Using the identity $\overline{\lambda_{\bullet, \sfw}} = \lambda_{\overline{\bullet}, \sfw}$, we conclude that they come in conjugate pairs. Let us denote the finite set of such complex numbers by $\{\mu_{k, \sfw}\}_{k = 0}^{k_\sfw} \subset (-\eta_\sfw, 0] + i[-1, 1]$ with $\mu_{0, \sfw} = 0$.

Let $\phi_1, \phi_2 \in L\bigl(U^{\tau^\sfw}\bigr)$. We use \cref{lem: a priori bounds,lem: meromorphic extension of transfer operators} and more importantly \cref{thm:SpectralBoundOnTransferOperator} to decompose $\hat{\Upsilon}^0_{\phi_1, \phi_2}$ into a part containing all the finitely many poles and a part which is holomorphic on a strip $(-2\eta_\sfw, 2\eta_\sfw) + i\R$ for some $\eta_\sfw \in (0, 1)$ which is continuous in $\sfw$ and decays exponentially in $\delta_\sfw$ to $0$:
\begin{align*}
\hat{\Upsilon}^0_{\phi_1, \phi_2} = \sum_{k = 0}^{k_\sfw} \frac{\Bigl\langle [\hat{\phi}_1]_{\mu_{k, \sfw}}, \mathsf{P}_{\overline{\mu_{k, \sfw}}, \sfw} \mathcal{L}_{\overline{\mu_{k, \sfw}}, \sfw} [\hat{\phi}_2]_{-\overline{\mu_{k, \sfw}}}\Bigr\rangle}{1 - \lambda_{\mu_{k, \sfw}, \sfw}} + \hat{\Upsilon}^{0, \mathrm{Hol}}_{\phi_1, \phi_2}.
\end{align*}
Note that $\lambda_{0, \sfw} = 1$ is a maximal simple eigenvalue of $\mathcal{L}_{0, \sfw}$ and $\mathsf{P}_{0, \sfw}$ is the projection operator onto the $1$-eigenspace which is spanned by constant functions. Taking the inverse Laplace transform, a Paley--Wiener type theorem gives the following.

\begin{proposition}
\label{prop:ExponentialMixingOnSuspensionSpace}
There exist $\eta_\sfw \in (0, 1)$ which is continuous in $\sfw$ and decays exponentially in $\delta_\sfw$ to $0$ and
\begin{itemize}
\item a finite set $\{\mu_{k, \sfw}\}_{k = 0}^{k_\sfw} \subset (-\eta_\sfw, 0) + i[-1, 1]$ which come in conjugate pairs, where $\mu_{0, \sfw} = 0$,
\item a finite set of finite-rank positive semi-definite bilinear forms $\{\mathcal{B}_{k, \sfw}\}_{k = 0}^{k_\sfw}$, where $\mathcal{B}_{0, \sfw}$ is of rank $1$ and given by the product of integrals with respect to $\nu^{\tau^\sfw}$,
\end{itemize}
such that for all $\phi_1, \phi_2 \in L\bigl(U^{\tau^\sfw}\bigr)$, we have
\begin{align*}
\Upsilon_{\phi_1, \phi_2}(t) = \sum_{k = 0}^{k_\sfw} e^{\mu_{k, \sfw} t} \mathcal{B}_{k, \sfw}(\phi_1, \phi_2) + O_\sfw\bigl(e^{-\eta_\sfw t} \|\phi_1\|_{\Lip} \|\phi_2\|_{\Lip}\bigr) \qquad \text{for all $t > 0$}
\end{align*}
where the implicit constant is continuous in $\sfw$.
\end{proposition}

\subsubsection{Completing the proof outline}

\begin{lemma}
Let $\phi_1, \phi_2 \in L(\mathcal{X})$ and $\xi \in \C$. Fix the sequence $\{t_j = (\underline{\tau^\sfw}/2)j\}_{j \in \N}$. We have that
\begin{align*}
\Bigl\{\Bigl\langle [\widehat{\mathcal{T}_{t_j, t_j}(\phi_1)}]_\xi, \mathsf{P}_{\overline{\xi}, \sfw} \mathcal{L}_{\overline{\xi}, \sfw} [\widehat{\mathcal{T}_{t_j, t_j}(\phi_2)}]_{-\overline{\xi}}\Bigr\rangle\Bigr\}_{j \in \N}
\end{align*}
forms a Cauchy sequence.
\end{lemma}

\begin{proof}
Let $\phi_1$, $\phi_2$, $\xi$, and $\{t_j\}_{j \in \N}$ be as in the lemma. Let $j' \geq j$. Firstly, by \cref{lem: correlation functions equal}, we have
\begin{align*}
\Upsilon_{\mathcal{T}_{t_j, t_j}(\phi_1), \mathcal{T}_{t_j, t_j}(\phi_2)} = \Upsilon_{\mathcal{T}_{t_j, t_{j'}}(\phi_1), \mathcal{T}_{t_j, t_{j'}}(\phi_2)}.
\end{align*}
We conclude that the poles of the two meromorphic functions $\hat{\Upsilon}^0_{\mathcal{T}_{t_j, t_j}(\phi_1), \mathcal{T}_{t_j, t_j}(\phi_2)}$ and $\hat{\Upsilon}^0_{\mathcal{T}_{t_j, t_{j'}}(\phi_1), \mathcal{T}_{t_j, t_{j'}}(\phi_2)}$ coincide, i.e., their orders and residues coincide. Therefore, the main coefficient of the poles of $\hat{\Upsilon}^0_{\mathcal{T}_{t_j, t_j}(\phi_1), \mathcal{T}_{t_j, t_j}(\phi_2)}$ can be expressed as
\begin{align*}
\Bigl\langle [\widehat{\mathcal{T}_{t_j, t_j}(\phi_1)}]_\xi, \mathsf{P}_{\overline{\xi}, \sfw} \mathcal{L}_{\overline{\xi}, \sfw} [\widehat{\mathcal{T}_{t_j, t_j}(\phi_2)}]_{-\overline{\xi}}\Bigr\rangle = \Bigl\langle [\widehat{\mathcal{T}_{t_j, t_{j'}}(\phi_1)}]_\xi, \mathsf{P}_{\overline{\xi}, \sfw} \mathcal{L}_{\overline{\xi}, \sfw} [\widehat{\mathcal{T}_{t_j, t_{j'}}(\phi_2)}]_{-\overline{\xi}}\Bigr\rangle.
\end{align*}
Recalling \cref{eqn: Iphi approximates phi} and invariance of $\nu^{\tau^\sfw}$ under the translation flow, there exists $\eta > 0$ such that
\begin{align*}
\bigl\|\mathcal{T}_{t_j, t_{j'}}(\phi) - \mathcal{T}_{t_{j'}, t_{j'}}(\phi)\bigr\|_2 \ll e^{-(\eta/\underline{\tau^\sfw})t_j} \|\phi\|_{\Lip} \qquad \text{for all $\phi \in L(\mathcal{X})$}.
\end{align*}
Combining the above identity and inequality, we obtain
\begin{multline*}
\Bigl|\Bigl\langle [\widehat{\mathcal{T}_{t_j, t_j}(\phi_1)}]_\xi, \mathsf{P}_{\overline{\xi}, \sfw} \mathcal{L}_{\overline{\xi}, \sfw}  [\widehat{\mathcal{T}_{t_j, t_j}(\phi_2)}]_{-\overline{\xi}}\Bigr\rangle - \Bigl\langle [\widehat{\mathcal{T}_{t_{j'}, t_{j'}}(\phi_1)}]_\xi, \mathsf{P}_{\overline{\xi}, \sfw} \mathcal{L}_{\overline{\xi}, \sfw} [\widehat{\mathcal{T}_{t_{j'}, t_{j'}}(\phi_2)}]_{-\overline{\xi}}\Bigr\rangle\Bigr| \\
\ll \overline{\tau^\sfw}e^{-(\eta/2)j + \max(0, -a) \overline{\tau^\sfw}}\|\phi_1\|_{\Lip}\|\phi_2\|_{\Lip}.
\end{multline*}
The lemma follows.
\end{proof}

As a result of the above lemma, we may define the set of finite-rank positive semi-definite bilinear forms $\{\mathcal{B}_{k, \sfw}\}_{k = 0}^{k_\sfw}$ as follows: for all $1 \leq k \leq k_\sfw$ and $\phi_1, \phi_2 \in L(\mathcal{X})$, we define
\begin{align*}
\mathcal{B}_{k, \sfw}(\phi_1, \phi_2) := \lim_{j \to \infty} \Bigl\langle [\widehat{\mathcal{T}_{t_j, t_j}(\phi_1)}]_{\mu_{k, \sfw}}, \mathsf{P}_{\overline{\mu_{k, \sfw}}, \sfw} \mathcal{L}_{\overline{\mu_{k, \sfw}}, \sfw} [\widehat{\mathcal{T}_{t_j, t_j}(\phi_2)}]_{-\overline{\mu_{k, \sfw}}}\Bigr\rangle,
\end{align*}
where $\mathcal{B}_{0, \sfw}$ is of rank $1$ and given by the product of integrals with respect to $m_\sfw^{\mathrm{BMS}}$. The proof outline of \cref{thm:RescaledExponentialMixingOnXWithPR-Resonances} is now completed by combining \cref{lem:ApproximatingTheCorrelationFunction,prop:ExponentialMixingOnSuspensionSpace} and the above definition of the set of bilinear forms.

\subsection{Reduction of $\nabla\growthindicator(\sfv) \in \ker\Theta$ to $\nabla\psi_{\Theta'}(\sfv') \notin \ker\Theta'$}
\label{subsec:ReductionOfGradGrowthIndicatorInKernelToNotInKernel}
Recall that \cref{thm:EssentialSpectralGap,thm:ExponentialPrimeOrbitTheorem} hold for any $\sfv \in \interior\limitcone$ including when $\nabla\growthindicator(\sfv) \in \ker\Theta$. To deal with the case that $\nabla\growthindicator(\sfv) \in \ker\Theta$, we exploit the fact that both the Selberg zeta function and the primitive orbit counting function do not directly depend on the particular dynamical system $\bigl(\mathcal{X}, \BMSX^{\sfv}, \{a^{\sfv}_t\}_{t \in \R}\bigr)$ but only on the \emph{existence} of such a dynamical system which gives the correct periods and the corresponding transfer operators.

Indeed, let us show that if $\nabla\growthindicator(\sfv) \in \ker\Theta$, we can find a nonempty proper subset $\Theta' \subset \Theta$ such that a corresponding dynamical system $\bigl(\mathcal{X}', m_{\mathcal{X}'}^{\sfv'}, \{a^{\sfv'}_t\}_{t \in \R}\bigr)$ gives the correct periods and $\nabla\growthindicator(\sfv) \notin \ker\Theta'$. Let
\begin{align*}
\Sigma \subset \Sigma_0 &:= \{\alpha \in \Theta: \nabla\growthindicator(\sfv) \in \ker\alpha\} \neq \varnothing, & \Theta' &:= \Theta - \Sigma.
\end{align*}
Recalling $\psi_\sfv = \langle \growthindicator(\sfv)^{-1}\nabla\growthindicator(\sfv), \cdot\rangle$ from \cref{eqn:GeneralLinearForm}, observe that
\begin{align*}
\nabla\growthindicator(\sfv) \in \LieA_{\Theta'} = \LieA_\Theta \cap \bigcap_{\alpha \in \Sigma} \ker\alpha \implies \LieA_{\Theta'}^\perp = \bigoplus_{\alpha \in \Sigma} \pi_{\LieA_\Theta}((\ker\alpha)^\perp) \subset \ker\psi_\sfv.
\end{align*}
Therefore, $\psi_\sfv$ factors through $\pi_{\LieA_{\Theta'}}|_{\LieA_\Theta}$: there exists $f \in \LieA_{\Theta'}^*$ such that
\begin{align*}
\psi_\sfv = f \circ \pi_{\LieA_{\Theta'}}|_{\LieA_\Theta}.
\end{align*}
Since $\mathcal{L}_{\Theta'} = \pi_{\LieA_{\Theta'}}(\limitcone)$, we conclude that $f \in \interior\mathcal{L}_{\Theta'}^*$. We may then write $f = \psi_{\sfv'} := \langle \psi_{\Theta'}(\sfv')^{-1}\nabla\psi_{\Theta'}(\sfv'), \cdot\rangle$ for some $\sfv' \in \interior\mathcal{L}_{\Theta'}$. By rescaling $\sfv'$ if necessary, we may also assume that $\psi_{\sfv'}(\sfv') = \psi_\sfv(\sfv)$. The canonical projection maps $\Fboundary \to \mathcal{F}_{\Theta'}$ and $\iFboundary \to \mathcal{F}_{\involution\Theta'}$ together with $\pi_{\LieA_{\Theta'}}|_{\LieA_\Theta}$ induce a $\Gamma$-equivariant projection map $\limitset^{(2)} \times \LieA_\Theta \to \Lambda_{\Theta'}^{(2)} \times \LieA_{\Theta'}$. Then, $\psi_\sfv$ and $\psi_{\sfv'}$ induce a $\Gamma$-equivariant projection map $\limitset^{(2)} \times \R \to \Lambda_{\Theta'}^{(2)} \times \R$. Thus, it descends to a map $\mathcal{X} \to \mathcal{X}'$ where $\mathcal{X} = \Gamma \backslash \bigl(\limitset^{(2)} \times \R\bigr)$ and $\mathcal{X}' = \Gamma \backslash \bigl(\Lambda_{\Theta'}^{(2)} \times \R\bigr)$. Clearly, the translation flows $\{a^\sfv_t\}_{t \in \R}$ and $\bigl\{a^{\sfv'}_t\bigr\}_{t \in \R}$ are intertwined by the map $\mathcal{X} \to \mathcal{X}'$ and has the same periods by construction of $\psi_{\sfv'}$. Finally, taking $\Sigma = \Sigma_0$ and the corresponding minimal $\Theta'$ ensures that $\nabla\psi_{\Theta'}(\sfv') \notin \ker\Theta'$.

As a result, in the subsequent subsections, we may assume that $\nabla\growthindicator(\sfv) \notin \ker\Theta$ and correspondingly that $\sfw \in \ww$ so that we can invoke \cref{thm:SpectralBoundOnTransferOperator}.

\subsection{Proof outline of \cref{thm:SpectralBoundOnTransferOperator} implies \cref{thm:EssentialSpectralGap}}
Recall the Ruelle zeta function $\zeta_\sfw$ from \cref{subsec:EssentialSpectralGapAndPrimeOrbitTheorem}. We can follow the short argument of Pollicott--Sharp \cite[\S\,2, pp. 1025--1030]{PS98} to derive an essential spectral gap for the Ruelle zeta function, \cref{thm:EssentialSpectralGapRuelle}, from the spectral bounds of the transfer operators, \cref{thm:SpectralBoundOnTransferOperator}, as we now outline. We can first put \cref{thm:SpectralBoundOnTransferOperator} into the following form (see the derivation just after \cite[Proposition 5.3]{Nau05}).

\begin{theorem}
\label{thm:SpectralBoundOnTransferOperator_1bNorm}
Let $\sfw \in \ww$. There exist $\{\eta_{\sfw, \epsilon}\}_{\epsilon > 0} \subset \R_{>0}$ and $a_\sfw > 0$ all of which are continuous in $\sfw$ and decay exponentially in $\delta_\sfw$ to $0$, and $b_0 > 0$, such that for all $\xi \in \C$ with $|a| < a_\sfw$ and $|b| > b_0$, $k \in \N$, and $H \in L(U,\C)$, we have
\begin{align*}
\big\|\mathcal{L}_{\xi, \sfw}^k(H)\big\|_{1, b} \ll_{\sfw, \epsilon} |b|^\epsilon e^{-\eta_{\sfw, \epsilon} k} \|H\|_{1, b}
\end{align*}
where the implicit constant is continuous in $\sfw$ and $\epsilon$.
\end{theorem}

The above theorem can then be used in place of \cite[Proposition 4]{PS98} of Dolgopyat.

The following lemma gives fundamental formulas for the Ruelle zeta function (see \cite[Lemma 1]{PS98} and \cite[Chapter 6, pp. 100]{PP90}).

\begin{lemma}
\label{lem:RuelleZetaFunctionIdentity}
We have
\begin{align*}
\begin{aligned}
\zeta_\sfw(\xi) &= \exp\left(\sum_{k = 1}^\infty \sum_{\gamma \in \mathcal{P}} \frac{1}{k}e^{-k\delta_\sfw \ell_\sfw(\gamma)}\right) \\
&= \exp\left(\sum_{k = 1}^\infty \sum_{u \in \Fix(\sigma^k)} \frac{1}{k}e^{-\xi\tau^\sfw_k(u)}\right)
\end{aligned}
\qquad \text{for all $\xi \in \C$ with $a > \delta_\sfw$}.
\end{align*}
\end{lemma}

The following lemma is obtained using \cref{lem:RuelleZetaFunctionIdentity,thm:SpectralBoundOnTransferOperator_1bNorm}.

\begin{lemma}
There exists $\eta_\sfw > 0$ which is continuous in $\sfw$ and decays exponentially in $\delta_\sfw$ to $0$ such that the following holds. Let $u_j \in U_j$ for all $j \in \mathcal{A}$. Then, for all $\xi \in \C$ with $|a| < a_0$ and $k \in \N$, we have
\begin{align*}
\left|\sum_{u \in \Fix(\sigma^k)} e^{-\xi\tau^\sfw_k(u)} - \sum_{j = 1}^N \mathcal{L}_{-\xi\tau^\sfw}^k \chi_{U_j}(u_j) \right| \ll_\sfw |b|ke^{-\eta_\sfw k}
\end{align*}
where the implicit constant is continuous in $\sfw$.
\end{lemma}

Finally, \cref{thm:SpectralBoundOnTransferOperator_1bNorm} and the two lemmas above are used to show that $\log\zeta_\sfw$ can be written as a series of holomorphic functions which converge uniformly on compact subsets of some vertical strip about the critical line $\delta_\sfw + i\R$; thereby proving \cref{thm:EssentialSpectralGapRuelle}. We also obtain a bound for $\log\zeta_\sfw$ on a half-plane $(\delta_\sfw - \eta_\sfw, +\infty) + i\R$. Moreover, with more complex analysis including the Phragm\'{e}n–Lindel\"{o}f theorem, we further obtain the following: there exists $\alpha \in (0, 1)$ of the form $\alpha = 1 - O(\eta_\sfw/\delta_\sfw) = 1 - O(e^{-c\delta_\sfw})$ for some $c > 0$ such that
\begin{align}
\label{eqn:RuelleZetaFunctionLogLipschitzEstimate}
\frac{\zeta'(\xi)}{\zeta(\xi)} = O_\sfw(|b|^\alpha) \qquad \text{as $|b| \to +\infty$}
\end{align}
uniformly on a half-plane $(\delta_\sfw - \eta_\sfw/4, +\infty) + i\R$, where the implicit constant is continuous in $\sfw$.

\subsection{Proof outline of \cref{thm:SpectralBoundOnTransferOperator} implies \cref{thm:ExponentialPrimeOrbitTheorem}}
Having established an essential spectral gap for the Selberg zeta function, \cref{thm:EssentialSpectralGap}, techniques from analytic number theory can be used to derive a prime orbit theorem with a power saving error term, \cref{thm:ExponentialPrimeOrbitTheorem}. Such a derivation is rather short and completely contained in \cite[\S\,3, pp. 1030--1033]{PS98}, as we now outline.

Define the weighted counting functions $\mathcal{M}_\sfw$ and $\mathcal{M}_{\sfw, 1}$ by
\begin{align*}
\mathcal{M}_{\sfw}(T) &= \sum_{k = 1}^\infty \sum_{\gamma \in \mathcal{P}: e^{k\delta_\sfw \ell_\sfw(\gamma)} \leq T} \delta_\sfw \ell_\sfw(\gamma), \\
\mathcal{M}_{\sfw, 1}(T) &= \int_1^T \mathcal{M}_\sfw(t) \, dt = \sum_{k = 1}^\infty \sum_{\gamma \in \mathcal{P}: e^{k\delta_\sfw \ell_\sfw(\gamma)} \leq T} \delta_\sfw \ell_\sfw(\gamma) \bigl(T - e^{k\delta_\sfw \ell_\sfw(\gamma)}\bigr),
\end{align*}
for all $T \geq 1$. We have the following propositions where \cref{lem:RuelleZetaFunctionIdentity} and the estimate in \cref{eqn:RuelleZetaFunctionLogLipschitzEstimate} is used to prove the first one which is then used to prove the second one.

\begin{proposition}
There exists $\eta_\sfw > 0$ which is continuous in $\sfw$ and decays exponentially in $\delta_\sfw$ to $0$ such that for all $c \in \bigl(1 - \frac{\eta_\sfw}{\delta_\sfw}, 1\bigr)$, we have
\begin{align*}
\mathcal{M}_{\sfw, 1}(T) = \frac{T^2}{2} + \frac{1}{2\pi i} \int_{c - i\infty}^{c + i\infty} \left(-\frac{\delta_\sfw\zeta_\sfw'(\delta_\sfw\xi)}{\zeta_\sfw(\delta_\sfw\xi)}\right) \frac{T^{\xi + 1}}{\xi(\xi + 1)} \, d\xi = \frac{T^2}{2} + O_\sfw\bigl(T^{c + 1}\bigr)
\end{align*}
for all $T \geq 1$, where the contour integral is along a vertical line through $c$ and the implicit constant is continuous in $\sfw$.
\end{proposition}

\begin{proposition}
\label{prop:M_w asymptotic formula}
There exists $\eta_\sfw > 0$ which is continuous in $\sfw$ and decays exponentially in $\delta_\sfw$ to $0$ such that for all $c \in \bigl(1 - \frac{\eta_\sfw}{\delta_\sfw}, 1\bigr)$, we have
\begin{align*}
\mathcal{M}_\sfw(T) = T + O_\sfw\bigl(T^{(c + 1)/2}\bigr) \qquad \text{for all $T \geq 1$}
\end{align*}
where the implicit constant is continuous in $\sfw$.
\end{proposition}

Now define the counting functions $\mathcal{N}_{\sfw, 0}$ and $\mathcal{N}_{\sfw, 1}$ by
\begin{align*}
\mathcal{N}_{\sfw, 0}(T) &= \sum_{k = 1}^\infty \sum_{\gamma \in \mathcal{P}: e^{k\delta_\sfw \ell_\sfw(\gamma)} \leq T} \frac{1}{k}, & \mathcal{N}_{\sfw, 1}(T) &= \sum_{\gamma \in \mathcal{P}: e^{\delta_\sfw \ell_\sfw(\gamma)} \leq T} 1,
\end{align*}
for all $T \geq 1$. We do a Riemann--Stieltjes integration by parts:
\begin{align*}
\mathcal{N}_{\sfw, 0}(T) &= \int_2^T \frac{1}{\log(t)} \, d\mathcal{M}_\sfw(t) + O_\sfw(1) \\
&= \left[\frac{\mathcal{M}_\sfw(t)}{\log(t)}\right]_{t = 2}^{t = T} - \int_2^T \mathcal{M}_\sfw(t) \cdot \left(\frac{1}{\log(t)}\right)' \, dt + O_\sfw(1).
\end{align*}
By \cref{prop:M_w asymptotic formula} and the fact that $\frac{T}{\log(T)} - \int_2^T t \cdot \bigl(\frac{1}{\log(t)}\bigr)' \, dt = \Li(T) + \frac{2}{\log(2)}$ (see \cref{eqn:Li function}), for any choice of $c \in \bigl(1 - \frac{\eta_\sfw}{\delta_\sfw}, 1\bigr)$, we have
\begin{align}
\label{eqn:N_w0 asymptotic}
\mathcal{N}_{\sfw, 0}(T) = \Li(T) + O_\sfw\bigl(T^{(c + 1)/2}/\log(T)\bigr).
\end{align}
Now we easily see that
\begin{align*}
\mathcal{N}_{\sfw, 0}(T) = \mathcal{N}_{\sfw, 1}(T) + \sum_{k = 2}^\infty \frac{1}{k}\mathcal{N}_{\sfw, 1}(T^{1/k}) = \mathcal{N}_{\sfw, 1}(T) + O_\sfw\bigl(T^{1/2}\log(T)\bigr)
\end{align*}
using the crude estimate that
\begin{align*}
\max\bigl\{k \in \N: \bigl\{\gamma \in \mathcal{P}: e^{\delta_\sfw \ell_\sfw(\gamma)} \leq T^{1/k}\bigr\} \neq \varnothing\bigr\} \ll_\sfw \log(T)
\end{align*}
derived from the existence of a smallest primitive period. All the implicit constants above can be taken to be continuous in $\sfw$. Finally, we rearrange \cref{eqn:N_w0 asymptotic} and apply change of variables $T \mapsto e^{\delta_\sfw T}$ to arrive at \cref{thm:ExponentialPrimeOrbitTheorem}.

\section{Local non-integrability condition}
\label{sec:LNIC}
This section is devoted to a crucial ingredient for Dolgopyat's method stated in \cref{pro:LNIC} called the \emph{local non-integrability condition (LNIC)} as in \cite{Sto11}, which is then upgraded to \cref{pro:NIC}. Although \cref{pro:NIC} is stated in reverse Lipschitz form on $\interior(U)$, we obtain it via \cref{pro:LNIC} which requires Lie theory and hence the smooth structure on $G$. Thus, we first resolve technical issues related to the incompatibility of $\limitset$ with the smooth structure on $G$.

\subsection{Smooth extensions of coding maps}
\label{subsec:SmoothExtensions}
Fix $\sfw \in \ww$ throughout this subsection. As alluded to above, due to the fractal nature of $\limitset$, the Markov section does not inherit the smooth structure on $G$. In order to circumvent this issue, we need to enlarge $U$ to an open set and smoothly extend the maps associated to the coding, $\tau^\sfw$, $\mathcal{P}$, and $\sigma$. This type of construction was done in \cite[\S\,5.1]{SW21} based on \cite[Lemma 1.2]{Rue89}. However, even this procedure has many more technical issues in our setting. One problem is that $\mathcal{X} = \Gamma \backslash \bigl(\limitset^{(2)} \times \R\bigr)$ is not naturally situated in a larger smooth manifold with a natural extension of the translation flow. Note that in general the action $\Gamma \curvearrowright \Fboundary^{(2)} \times \R$ is not properly discontinuous (see \cref{rem: properly discontinuous}).
To avoid this problem, we work on the cover $\limitset^{(2)} \times \R$ instead and use the inclusion
\begin{align*}
\limitset^{(2)} \times \R \hookrightarrow \Fboundary^{(2)} \times \R
\end{align*}
which is a Lipschitz embedding into a smooth manifold where the translation flow is still defined. Another problem is that it is difficult to gain control on the $\Gamma$-orbit of the open neighborhood of $U = \bigsqcup_{j \in \mathcal{A}} U_j$ in $\Fboundary^{(2)} \times \R$ (e.g., diameter estimates, mutual disjointness, seperation distance) since we do not have a $\Gamma$-invariant metric on $\Fboundary^{(2)} \times \R$ and $\Gamma \curvearrowright \Fboundary^{(2)} \times \R$ may not be properly discontinuous. To avoid these topological obstructions, we allow dependence of the open neighborhoods on the arbitrarily large length $m \in \N$ of sections of $\sigma^m$ which will appear in \cref{pro:LNIC}. We do not follow \cite[Lemma 1.2]{Rue89} because, although we believe that it is true, it would take significantly more work to properly formulate and establish the eventually expanding/contracting property of the smooth extension of $\sigma^{-1}$ on some fixed open neighborhood of $U$. See \cite{DMS24} for the case that $\Gamma$ is a projective Anosov subgroup of $\PSL_n(\R)$ which gives a construction of a flow-invariant smooth manifold $\mathcal{M} \supset \mathcal{X}$ on which the translation flow is a contact Axiom A flow.

In light of the above discussion, let us first fix a connected fundamental domain $\mathsf{D} \subset \limitset^{(2)} \times \R$ for $\mathcal{X}$, unique isometric lifts
\begin{align*}
\mathsf{R}_j = [\mathsf{U}_j, \mathsf{S}_j] \subset \mathsf{D}
\end{align*}
of $R_j = [U_j, S_j]$ for all $j \in \mathcal{A}$, and write
\begin{align*}
\mathsf{U} &:= \bigsqcup_{j \in \mathcal{A}} \mathsf{U}_j, & \mathsf{R} &:= \bigsqcup_{j \in \mathcal{A}} \mathsf{R}_j.
\end{align*}
Denote by $\tilde{u} \in \mathsf{R}$ the unique lift of $u \in R$. Abusing notation, we denote the lift of cylinders by the same symbol. Since the translation flow is defined on $\limitset^{(2)} \times \R$, the maps $\tau^\sfw$, $\mathcal{P}$, and $\sigma$ have natural lifts $\tau^\sfw: \Gamma\mathsf{R} \to \R$, $\mathcal{P}: \Gamma\mathsf{R} \to \Gamma\mathsf{R}$, and $\sigma: \Gamma\mathsf{U} \to \Gamma\mathsf{U}$, abusing notation. We have corresponding cores $\hat{\mathsf{U}}$ and $\hat{\mathsf{R}}$. For all $u \in \mathsf{R}$ we define $\mathsf{g}(u) \in \Gamma$ to be the unique element such that $\mathcal{P}(u) \in \mathsf{g}(u)\mathsf{R}$. Similarly, for all admissible pairs $(j, k)$, we write $\mathsf{g}^{(j, k)} \in \Gamma$ for the unique element such that
\begin{align*}
\sigma(\interior(\mathsf{U}_j)) \cap \mathsf{g}^{(j, k)}\interior(\mathsf{U}_k) \neq \varnothing.
\end{align*}
Then, $\hat{\mathsf{g}} = \{\mathsf{g}^{(j, k)}: (j, k) \text{ is an admissible pair}\} \subset \Gamma$ is a generating subset. For all admissible sequences $\alpha = (\alpha_0, \alpha_1, \dotsc, \alpha_k)$ for some $k \in \Z_{\geq 0}$, we extend the notation and write
\begin{align}
\label{eqn:BirkhoffProduct}
\mathsf{g}^\alpha := \prod_{j = 0}^{k - 1} \mathsf{g}^{(\alpha_j, \alpha_{j + 1})}
\end{align}
in ascending order if $k > 0$, and $\mathsf{g}^\alpha := e$ if $k = 0$, and $\mathsf{g}^{-\alpha} := (\mathsf{g}^\alpha)^{-1}$.

Let $j \in \mathcal{A}$, $w_j \in \mathsf{U}_j$ be the center, and $\gamma \in \Gamma$. We have a metric on $W^{\Fboundary, \mathrm{ss}}(\gamma w_j)$ (resp. $W^{\Fboundary, \mathrm{wu}}(\gamma w_j)$) which is induced by $d_{\iFboundary}$ (resp. $d_{\Fboundary} \times d_\R$) using coordinate pullbacks and denoted by the same notation. Then, $\bigl(W_{\epsilon_\gamma}^{\Fboundary, \mathrm{wu}}(\gamma w_j), W_{\epsilon_\gamma}^{\Fboundary, \mathrm{ss}}(\gamma w_j)\bigr)$ has a local product structure such that $\bigl[W_{\epsilon_\gamma}^{\Fboundary, \mathrm{su}}(\gamma w_j), W_{\epsilon_\gamma}^{\Fboundary, \mathrm{ss}}(\gamma w_j)\bigr]$ contains $\gamma \mathsf{R}_j$, for some $\epsilon_\gamma > 0$. Now, fix open neighborhoods
\begin{align*}
{}^\gamma\tilde{\mathsf{U}}_j \subset W_{\epsilon_\gamma}^{\Fboundary, \mathrm{su}}(\gamma w_j)
\end{align*}
of $\gamma\mathsf{U}_j$ and define
\begin{align*}
{}^\gamma\tilde{\mathsf{R}}_j = [{}^\gamma\tilde{\mathsf{U}}_j, \gamma\mathsf{S}_j]
\end{align*}
such that $\{{}^\gamma\tilde{\mathsf{R}}_j\}_{j \in \mathcal{A}, \gamma \in \Gamma}$ consists of mutually disjoint rectangles. Define
\begin{align*}
\tilde{\mathsf{U}}_j &:= {}^e\tilde{\mathsf{U}}_j, & \tilde{\mathsf{U}} &:= \bigsqcup_{j \in \mathcal{A}} \tilde{\mathsf{U}}_j, & {}^\Gamma\tilde{\mathsf{U}}_j &:= \bigsqcup_{\gamma \in \Gamma} {}^\gamma\tilde{\mathsf{U}}_j, & {}^\Gamma\tilde{\mathsf{U}} &:= \bigsqcup_{j \in \mathcal{A}, \gamma \in \Gamma} {}^\gamma\tilde{\mathsf{U}}_j,
\end{align*}
and similarly define $\tilde{\mathsf{R}}_j$, $\tilde{\mathsf{R}}$, ${}^\Gamma\tilde{\mathsf{R}}_j$, and ${}^\Gamma\tilde{\mathsf{R}}$. We similarly omit the superscript `$e$' and use the superscript `$\Gamma$' for other sets.

Let $m \in \Z_{\geq 0}$. Recall that the translation flow is defined on $\Fboundary^{(2)} \times \R$ and smooth. Thus, for all $j \in \mathcal{A}$ and $\gamma \in \Gamma$, we obtain compactly contained open neighborhoods
\begin{align*}
{}^\gamma\tilde{\mathsf{U}}_j^{(m)} \subset {}^\gamma\tilde{\mathsf{U}}_j
\end{align*}
of $\gamma\mathsf{U}_j$ which are decreasing in $m$, and corresponding to admissible sequences $\alpha = (\alpha_0, \alpha_1, \dotsc, \alpha_m)$ the natural smooth injective extensions
\begin{align*}
\sigma^{-\alpha}: {}^{\Gamma\mathsf{g}^\alpha}\tilde{\mathsf{U}}_{\alpha_m}^{(m)} \to {}^\Gamma\tilde{\mathsf{U}}_{\alpha_0}
\end{align*}
of $\bigl(\sigma|_{\mathtt{C}[\alpha]}\bigr)^{-1}$ with cylinders
\begin{align*}
{}^\gamma\tilde{\mathtt{C}}[\alpha]^{(m)} := \sigma^{-\alpha}\bigl({}^{\gamma\mathsf{g}^\alpha}\tilde{\mathsf{U}}_{\alpha_m}^{(m)}\bigr) \subset {}^\gamma\tilde{\mathsf{U}}_{\alpha_0};
\end{align*}
we then define
\begin{align*}
{}^\gamma\tilde{\mathsf{R}}_j^{(m)} := [{}^\gamma\tilde{\mathsf{U}}_j^{(m)}, \gamma\mathsf{S}_j]
\end{align*}
and
\begin{align*}
\sigma^\alpha := (\sigma^{-\alpha})^{-1}: {}^\Gamma\tilde{\mathtt{C}}[\alpha]^{(m)} \to {}^{\Gamma\mathsf{g}^\alpha}\tilde{\mathsf{U}}_{\alpha_m}^{(m)}.
\end{align*}

There are also natural smooth extensions $\tau^\sfw_{(j, k)}: {}^\Gamma\tilde{\mathtt{C}}[j, k]^{(1)} \to \R$ for all admissible pairs $(j, k)$. We define the smooth maps $\tau^\sfw_\alpha: {}^\Gamma\tilde{\mathtt{C}}[\alpha]^{(m)} \to \R$ by
\begin{align}
\label{eqn:BirkhoffSum}
\tau^\sfw_\alpha(u) &= \sum_{j = 0}^{m - 1} \tau^\sfw_{(\alpha_j, \alpha_{j + 1})}(\sigma^{(\alpha_0, \alpha_1, \dotsc, \alpha_j)}(u))
\end{align}
for all admissible sequences $\alpha = (\alpha_0, \alpha_1, \dotsc, \alpha_m)$.

Following \cite[\S\,5]{CS23} we can construct a smooth section 
\begin{align*}
\mathsf{F}: {}^\Gamma\tilde{\mathsf{R}} \to G
\end{align*}
in a similar fashion such that:
\begin{enumerate}
\item
\label{itm:SectionPropertyOnUnstableLeafs}
for all $j \in \mathcal{A}$ and $\gamma \in \Gamma$, and $u, u' \in {}^\gamma\tilde{\mathsf{U}}_j$, there exists unique $n^+ \in N^+_\Theta$ such that $\mathsf{F}(u') = \mathsf{F}(u)n^+$;
\item
\label{itm:SectionPropertyOnStableLeafs}    
for all $j \in \mathcal{A}$ and $\gamma \in \Gamma$, $u \in {}^\gamma\tilde{\mathsf{U}}_j$, and $s, s' \in \gamma\mathsf{S}_j$, there exists unique $n^- \in N^-_\Theta$ such that $\mathsf{F}([u, s']) = \mathsf{F}([u, s])n^-$.
\end{enumerate}
Observe that $\mathsf{F}$ can indeed be constructed such that it does not depend on $\sfw$ due to the properties of the compatible Markov sections.

We introduce the first return vector map and holonomy. Although we only need the first return time map, since it does not require excess work, we provide the general definitions and the subsequent fundamental lemmas in anticipation that it will turn out useful elsewhere.

\begin{definition}[First return vector map, Holonomy]
The \emph{first return $\Theta$-vector map} and \emph{$\Theta$-holonomy} are $\Gamma$-invariant maps $\mathsf{K}: \Gamma\mathsf{R} \to \LieA_\Theta$ and $\vartheta: \Gamma\mathsf{R} \to S_\Theta$ respectively that associate to each $u \in \gamma\mathsf{R}$ for some $\gamma \in \Gamma$ the unique elements $\mathsf{K}(u) \in \LieA_\Theta$ and $\vartheta(u) \in S_\Theta$ which satisfy
\begin{align*}
\mathsf{F}(\mathcal{P}(u)) = \mathsf{g}(u)\mathsf{F}(u) a_{\mathsf{K}(u)}\vartheta(u).
\end{align*}
We often drop the prefix ``$\Theta$-''.
\end{definition}

Again, $\mathsf{K}$ and $\vartheta$ are independent of $\sfw$. Similar to previous constructions, we have smooth extensions $\mathsf{K}_{(j, k)}: {}^\Gamma\tilde{\mathtt{C}}[j, k]^{(1)} \to \LieA_\Theta$ and $\vartheta^{(j, k)}: {}^\Gamma\tilde{\mathtt{C}}[j, k]^{(1)} \to S_\Theta$ for all admissible pairs $(j, k)$. Again, for all $m \in \Z_{\geq 0}$ and admissible sequences $\alpha = (\alpha_0, \alpha_1, \dotsc, \alpha_m)$, we define the smooth maps $\mathsf{K}_\alpha: {}^\Gamma\tilde{\mathtt{C}}[\alpha]^{(m)} \to \LieA_\Theta$ and $\vartheta^\alpha: {}^\Gamma\tilde{\mathtt{C}}[\alpha]^{(m)} \to S_\Theta$ as in \cref{eqn:BirkhoffSum} and in the style of \cref{eqn:BirkhoffProduct} for holonomy.

\subsection{Local non-integrability condition}
\label{subsec:LNIC}
We are now ready to prove \cref{pro:LNIC,pro:NIC}. The proof follows the techniques developed in \cite{SW21,CS22}.
The version here is actually more simplified in some parts and we obtain a stronger LNIC because we are dealing with the ``geodesic flow'' (see \cref{rem:Weak_VS_Strong_LNIC}).

We start with a slight generalization of \cite[Definition 6.1]{SW21}. Recall the identifications from \cref{subsec:CompatibleFamilyOfMarkovSections}.

\begin{definition}[Associated sequence in $G$]
Let $\sfw \in \interior\limitcone^*$ with $\|\sfw\| = 1$. Let $z_1 \in \tilde{\mathsf{R}}_1$ be the center. Consider a sequence $(z_1, z_2, z_3, z_4, z_1) \in (\tilde{\mathsf{R}}_1)^5$ such that $z_2 \in \mathsf{S}_1$, $z_4 \in \tilde{\mathsf{U}}_1$, and $z_3 = [z_4, z_2]$. We define an \emph{associated sequence in $G$} to be the unique sequence $(g_1, g_2, \dotsc, g_5) \in G^5$ where
\begin{align*}
g_1 &= \mathsf{F}(z_1), \\
g_2 &= \mathsf{F}(z_2) \in g_1N^-_\Theta \text{ such that } \pi_{\psi_{\sfv(\sfw)}}(g_2S_\Theta) = z_2, \\
g_3 &\in g_2N^+_\Theta \text{ such that } a^{\sfv(\sfw)}_t\bigl(\pi_{\psi_{\sfv(\sfw)}}(g_3S_\Theta)\bigr) = z_3  \text{ for some } t \in (-\underline{\tau}, \underline{\tau}), \\
g_4 &\in g_3N^-_\Theta \text{ such that } a^{\sfv(\sfw)}_t\bigl(\pi_{\psi_{\sfv(\sfw)}}(g_4S_\Theta)\bigr) = z_4 \ \text{ for some } t \in (-\underline{\tau}, \underline{\tau}), \\
g_5 &\in g_4N^+_\Theta \text{ such that } a^{\sfv(\sfw)}_t\bigl(\pi_{\psi_{\sfv(\sfw)}}(g_5S_\Theta)\bigr) = z_1 \text{ for some } t \in (-\underline{\tau}, \underline{\tau}).
\end{align*}
\end{definition}

\begin{remark}
In the above definition, the associated sequence $(g_1, g_2, g_3, g_4, g_5)$ corresponding to each $(z_1, z_2, z_3, z_4, z_1)$ is independent of $\sfw$.
\end{remark}

We continue using the notation in the above definition. Define the subsets
\begin{align*}
(N^+_\Theta)_1 &= \{n^+ \in N^+_\Theta: \mathsf{F}(z_1)n^+ \in \mathsf{F}(\tilde{\mathsf{U}}_1)\} \subset N^+_\Theta, \\
(N^-_\Theta)_1 &= \{n^- \in N^-_\Theta: \mathsf{F}(z_1)n^- \in \mathsf{F}(\mathsf{S}_1)\} \subset N^-_\Theta,
\end{align*}
where the first is open and the second is compact. Now, if the above sequence $(z_1, z_2, z_3, z_4, z_1)$ corresponds to some $n^+ \in (N^+_\Theta)_1$ and some $n^- \in (N^-_\Theta)_1$ such that $\mathsf{F}(z_4) = \mathsf{F}(z_1)n^+$ and $\mathsf{F}(z_2) = \mathsf{F}(z_1)n^-$ respectively, then we can define the map $\Xi: (N^+_\Theta)_1 \times (N^-_\Theta)_1 \to A_\Theta S_\Theta$ by
\begin{align*}
\Xi(n^+, n^-) = g_5^{-1}g_1 \in A_\Theta S_\Theta.
\end{align*}
To view it as a function of the first coordinate for a fixed $n^- \in (N^-_\Theta)_1$, we write $\Xi_{n^-}: (N^+_\Theta)_1 \to A_\Theta S_\Theta$.

Let $z_1 \in \tilde{\mathsf{R}}_1$ be the center. Let $j \in \N$ and $\alpha = (\alpha_0, \alpha_2, \dotsc, \alpha_{j - 1}, 1)$ be an admissible sequence. Then, there exists an element which we denote by $n_\alpha \in (N^-_\Theta)_1$ such that
\begin{align*}
\mathsf{F}(\mathcal{P}^j(\sigma^{-\alpha}(z_1))) = \mathsf{F}(z_1)n_\alpha.
\end{align*}
This is well-defined because $\sigma^{-\alpha}(z_1) \in \mathsf{g}^{-\alpha}\mathtt{C}[\alpha] \subset \mathsf{g}^{-\alpha}\mathsf{U}$.

Let $\sfw \in \interior\limitcone^*$ with $\|\sfw\| = 1$. Let $\pi_{\LieA_\Theta \oplus \LieS_\Theta}: \LieG \to \LieA_\Theta \oplus \LieS_\Theta$ and $\pi_{\LieA_\Theta}: \LieG \to \LieA_\Theta$ be orthogonal projection maps. Let
\begin{align*}
\pi_\sfw:\LieG \to \R\sfw
\end{align*}
denote the orthogonal projection map, i.e., the projection with respect to the decomposition $\LieG \cong \R\sfw \oplus \ker\psi_{\sfv(\sfw)} \oplus \LieA_\Theta^\perp$.
Similarly, let
\begin{align*}
\tilde{\pi}_\sfw: A_\Theta S_\Theta \to \exp(\R\sfw)
\end{align*}
be the Cartesian projection map of $A_\Theta S_\Theta \cong \exp(\R\sfw) \times \exp(\ker\psi_{\sfv(\sfw)}) \times S_\Theta$ onto the $\exp(\R\sfw)$ factor. Note that $(d\tilde{\pi}_\sfw)_e = \pi_\sfw|_{\LieA_\Theta \oplus \LieS_\Theta}$.

Also define $(N^-_\Theta)_{1, \epsilon_e} = \left\{n^- \in N^-_\Theta: \mathsf{F}(z_1)n^- \in \mathsf{F}\left(W_{\epsilon_e}^{\Fboundary, \mathrm{ss}}(z_1)\right)\right\}$ where $\epsilon_e$ is as in \cref{subsec:SmoothExtensions} and $z_1 \in \tilde{\mathsf{R}}_1$ is the center.

In order to derive the LNIC in \cref{pro:LNIC}, we need the following two lemmas regarding $\Xi$. We omit the proofs since they are almost a verbatim repetition of that of \cite[Lemmas 6.2 and 6.3]{SW21}.

\begin{lemma}
\label{lem:BrinPesinInTermsOfHolonomy}
Let $m \in \N$, $\alpha = (\alpha_0, \alpha_1, \dotsc, \alpha_{m - 1}, 1)$ be an admissible sequence, and $n^- = n_\alpha \in (N^-_\Theta)_1$. Let $u \in \tilde{\mathsf{U}}_1^{(m)}$ and $n^+ \in (N^+_\Theta)_1$ such that $\mathsf{F}(u) = \mathsf{F}(z_1)n^+$ where $z_1 \in \tilde{\mathsf{R}}_1$ is the center. Then, we have
\begin{align*}
\Xi(n^+, n^-) = a_{-\mathsf{K}_\alpha(\sigma^{-\alpha}(z_1))}\vartheta^\alpha(\sigma^{-\alpha}(z_1))^{-1} a_{\mathsf{K}_\alpha(\sigma^{-\alpha}(u))}\vartheta^\alpha(\sigma^{-\alpha}(u)).
\end{align*}
In particular, for all $\sfw \in \interior\limitcone^*$ with $\|\sfw\| = 1$, we have
\begin{align*}
\log \tilde{\pi}_\sfw(\Xi(n^+, n^-)) = \bigl(\tau^\sfw_\alpha(\sigma^{-\alpha}(u)) - \tau^\sfw_\alpha(\sigma^{-\alpha}(z_1))\bigr)\sfw.
\end{align*}
\end{lemma}

\begin{lemma}
\label{lem:BrinPesinDerivativeImageIsAdjointProjection}
For all $n^- \in (N^-_\Theta)_1$, we have
\begin{align*}
(d\Xi_{n^-})_e = \pi_{\LieA_\Theta \oplus \LieS_\Theta} \circ \Ad_{n^-}|_{\LieN^+_\Theta} \circ (dh_{n^-})_e,
\end{align*}
and in particular, for all $\sfw \in \interior\limitcone^*$ with $\|\sfw\| = 1$, we have
\begin{align*}
d(\log \circ \tilde{\pi}_\sfw \circ \Xi_{n^-})_e = \pi_\sfw \circ \Ad_{n^-}|_{\LieN^+_\Theta} \circ (dh_{n^-})_e,
\end{align*}
where $h_{n^-}: (N^+_\Theta)_1 \to N^+_\Theta$ is a diffeomorphism onto its image which is also smooth in $n^- \in (N^-_\Theta)_{1, \epsilon_e}$ and satisfies $h_{n^-}(e) = e$ and $h_e = \Id_{(N^+_\Theta)_1}$.
\end{lemma}

Denote $\Pi^- := \{-\alpha: \alpha \in \Pi\}$. For all $\alpha \in \Pi \sqcup \Pi^-$, let $H_\alpha \in \LieA$ be the corresponding root direction, i.e., $\alpha = \langle \cdot, H_\alpha\rangle$. The following lemma is more or less a standard Lie algebra fact. We give a proof for the sake of completeness and for comparison with the analogous statement in \cite[Proposition 4.6]{CS22}. Using loc. cit. more directions can be produced (including those in $\LieM$) whereas our lemma below gives the stronger quantifier ``for all nonzero $x_\alpha \in \LieG_\alpha$'' for each $\alpha \in \Pi \sqcup \Pi^-$ (cf. \cref{rem:Weak_VS_Strong_LNIC}).

\begin{lemma}
\label{lem:BracketofRootSpaces}
Let $\alpha \in \Pi \sqcup \Pi^-$. For all nonzero $x_\alpha \in \LieG_\alpha$, there exists a nonzero $x_{-\alpha} \in \LieG_{-\alpha}$ such that $[x_{-\alpha},x_\alpha] \in \R_{> 0} H_\alpha$.
\end{lemma}

\begin{proof}
Let $\alpha \in \Pi$ and $x_\alpha \in\LieG_\alpha$ be nonzero. We will show that $x_{-\alpha} := \theta(x_\alpha)$ satisfies the lemma. Indeed $x_{-\alpha} \in \LieG_{-\alpha}$ because for all $H \in \LieA$, we have
\begin{align*}
[H, x_{-\alpha}] = \theta[\theta(H), \theta(x_{-\alpha})] = \theta[-H, x_\alpha] = -\alpha(H)\theta(x_{\alpha}) = -\alpha(H)x_{-\alpha}.
\end{align*}
Thus, $[x_{-\alpha},x_\alpha] \in [\LieG_{-\alpha},\LieG_\alpha] \subset \LieG_0 = \LieA\oplus\LieM$. Moreover, we have
\begin{align*}
\theta[x_{-\alpha}, x_\alpha] = [\theta(x_{-\alpha}),\theta(x_\alpha)] = [x_\alpha,x_{-\alpha}] = -[x_{-\alpha},x_\alpha]
\end{align*}
which implies $[x_{-\alpha},x_\alpha] \in \LieP$. Hence, $[x_{-\alpha},x_\alpha] \in (\LieA\oplus\LieM) \cap \LieP = \LieA$.
Recalling the definition of the inner product $\langle \cdot, \cdot \rangle$ and norm $\|\cdot\|$ on $\LieG$ from \cref{subsec:LieTheoreticPreliminaries}, we have
\begin{align*}
\langle H, [x_{-\alpha},x_\alpha]\rangle = B(H, [x_{-\alpha},x_\alpha]) = B([H, x_{-\alpha}], x_\alpha) = -\alpha(H)B(x_{-\alpha}, x_\alpha)
\end{align*}
which implies $[x_{-\alpha},x_\alpha] = -B(x_{-\alpha}, x_\alpha) H_\alpha = \|x_\alpha\|^2 H_\alpha \in \R_{>0} H_\alpha$.
\end{proof}

Repeating an analogous proof for arbitrary $\Theta$, we have the following generalization of \cite[Lemma 2.11]{ELO23} which is itself a generalization of \cite[Proposition 3.12]{Win15}. Alternatively, \cite[Theorem 1.1]{KO23a} (which is a stronger measure theoretic statement but for subvarieties) also suffices for our purposes in the proof of \cref{lem:am_ProjectionOfAdjointImage}.

\begin{lemma}
\label{lem:LimitSetNotInSubmanifold}
For any open subset $\mathcal{O} \subset \Fboundary$ with $\limitset \cap \mathcal{O} \neq \varnothing$, the intersection $\limitset \cap \mathcal{O}$ is not contained in any smooth submanifold of $\Fboundary$ of lower dimension.
\end{lemma}

We write the next lemma in its most general form, although we eventually apply its corollary for $\sfw \in \overline{\ww}$. Let $\sfw \in \interior\limitcone^*$ with $\|\sfw\| = 1$. We define
\begin{align*}
\Theta_\sfw &:= \{\alpha \in \Theta: \sfw \notin \ker\alpha\} \subset \Theta, \\
\LieSimple^\pm_\sfw &:= \bigoplus_{\alpha \in \Theta_\sfw} \LieG_{\mp\alpha} \subset \LieN^\pm_\Theta.
\end{align*}
Since $\Theta_\sfw$ is always nonempty, if $\#\Theta = 1$, then we trivially have $\Theta_\sfw = \Theta$. Morevoer, $\sfw \in \overline{\ww}$ implies $\Theta_\sfw = \Theta$ and $\LieSimple^\pm_\sfw = \LieN^\pm_\Theta$.

\begin{lemma}
\label{lem:am_ProjectionOfAdjointImage}
Let $\sfw \in \interior\limitcone^*$ with $\|\sfw\| = 1$. For all nonzero $n^+ \in \LieSimple^+_\sfw$, there exists $n^- \in (N^-_\Theta)_1$ such that
\begin{align}
\label{eqn:ProjectionOfAdjointNonzero}
\pi_\sfw(\Ad_{n^-}(n^+)) \neq 0.
\end{align}
\end{lemma}

\begin{proof}
Let $\sfw$ and $n^+$ be as in the lemma. Suppose for the sake of contradiction that $\pi_\sfw\bigl(\Ad_{(N^-_\Theta)_1}(n^+)\bigr) = 0$. Without loss of generality, we may assume that $\mathsf{F}(z_1) = e$ at the center $z_1 \in \tilde{\mathsf{R}}_1$. Consider the smooth map
\begin{align*}
L: N^-_\Theta &\to \R \\
n^- &\mapsto \pi_\sfw(\Ad_{n^-}(n^+)).
\end{align*}
Then, we have $dL_e(\hat{n}^-) = \pi_\sfw([\hat{n}^-,n^+])$ for all $\hat{n}^- \in \LieN^-_\Theta$. 

Note that $\pi_{\LieA_\Theta}(H_\alpha) \notin \ker\psi_{\sfv(\sfw)} = \sfw^\perp$ for all $\alpha \in \Theta_\sfw$. Let $x_\alpha$ for some $\alpha \in \Theta_\sfw$ be a nonzero component in the decomposition of $n^+$ according to $\LieSimple^+_\sfw = \bigoplus_{\alpha \in \Theta_\sfw} \LieG_{-\alpha}$. By \cref{lem:BracketofRootSpaces}, there exists $x_{-\alpha} \in \LieG_\alpha \subset \LieN^-_\Theta$ such that $[x_{-\alpha},x_\alpha] \in \R_{>0}H_\alpha$. Using the definition of $\LieSimple^+_\sfw$, we have $\pi_\sfw([x_{-\alpha},x_\alpha]) \ne 0$. Then $dL_e(x_{-\alpha}) = \pi_\sfw([x_{-\alpha},n^+]) = \pi_\sfw([x_{-\alpha},x_\alpha]) \ne 0$ so there is a neighborhood $O \subset N^-_\Theta$ containing $e$ such that $L^{-1}(0) \cap O$ is a smooth submanifold (in fact a subvariety) of $N^-_\Theta$ of lower dimension. However, $(N^-_\Theta)_1 \cap O \subset L^{-1}(0) \cap O$, but on the other hand $((N^-_\Theta)_1)^-$ contains $\ilimitset \cap \mathcal{O}$ for some open set $\mathcal{O} \subset \iFboundary$. It follows via the diffeomorphism $N^-_\Theta \to N^-_\Theta e^-$ that $\ilimitset \cap \mathcal{O}$ is contained in a smooth submanifold of $\iFboundary$ of lower dimension, contradicting \cref{lem:LimitSetNotInSubmanifold}.
\end{proof}

For any normed vector space $(V, \|\cdot\|)$, let $\mathbb{S}(V)$ denote the unit sphere in $V$ centered at $0 \in V$. Since \cref{eqn:ProjectionOfAdjointNonzero} is an open condition, the maps involved are smooth, and $\overline{\ww}$ and $\mathbb{S}(\LieN^+_\Theta)$ are compact, a straightforward compactness argument gives the following corollary.

\begin{corollary}
\label{cor:ProjectionOfAdjointNonzero}
There exist $C_\ww > 0$, $\jj \in \N$, $\delta > 0$, nontrivial elements $n_1^-, n_2^-, \dotsc,\allowbreak n_{\jj}^- \in (N^-_\Theta)_1$, and finite subsets $\mathscr{s}_1, \mathscr{s}_2, \dotsc, \mathscr{s}_{\jj} \subset \mathbb{S}(\LieN^+_\Theta)$ with $\bigcup_{j = 1}^{\jj} B_\delta^{\LieN^+_\Theta}(\mathscr{s}_j) \supset \mathbb{S}(\LieN^+_\Theta)$ such that the following holds. Let $\sfw \in \overline{\ww}$, and $1 \leq j \leq \jj$, and $\eta_j^- \in (N^-_\Theta)_1$ with $d_{N^-_\Theta}(\eta_j^-, n_j^-) \le \delta$. Then, for all $n^+ \in \mathbb{S}(\LieN^+_\Theta) \cap B_{2\delta}^{\LieN^+_\Theta}(\mathscr{s}_j)$, we have
\begin{align*}
\bigl|\pi_\sfw\bigl(\Ad_{\eta_j^-}(n^+)\bigr)\bigr| \geq C_\ww.
\end{align*}
\end{corollary}

\begin{remark}
\label{rem:Weak_VS_Strong_LNIC}
It is important to note that although \cref{cor:ProjectionOfAdjointNonzero} is analogous to \cite[Lemma 6.4]{SW21}, when projected to $\R\sfw$ for both statements, the former effectively has a stronger quantifier ``for all $n^+ \in \mathbb{S}(\LieN^+_\Theta)$''. As a result, \cref{pro:LNIC} also has a stronger quantifier and is hence stronger than the $\R\sfw$ projection of \cite[Proposition 6.5]{SW21}. Note that the stronger proposition is possible only for the ``geodesic flow'' and is ultimately rooted in \cref{lem:BracketofRootSpaces}. Due to this stronger proposition, we do not require the non-concentration property as in \cite{SW21}. At a conceptual level, this is similar to what happens in \cite{Sto11}.
\end{remark}

With \cref{lem:BrinPesinDerivativeImageIsAdjointProjection,cor:ProjectionOfAdjointNonzero} in hand, the proof of \cref{pro:LNIC} is similar to that of \cite[Proposition 6.5]{SW21}. However, a subtle but important point in the proof below is that rather than sticking with derivative bounds, we convert to a reverse Lipschitz version of a bound for a related function defined on $\tilde{\mathsf{U}}_1$. This circumvents the following technical issue: choice of a convex open neighborhood in $\tilde{\mathsf{U}}_1^{(m)}$ depends on $m$ and this is fatal for Dolgopyat's method.

\begin{proposition}[LNIC]
\label{pro:LNIC}
There exist $\varepsilon \in (0, 1)$, $m_0 \in \N$, $\jj \in \N$, and a cylinder $\mathtt{Q} \subset U_1$ of the center $z_1 \in R_1$ such that for all $m \geq m_0$, there exist sections $v_j = \sigma^{-x_j}: U_1 \to U_{x_{j, 0}}$ for some mutually distinct admissible sequences $x_j = (x_{j, 0}, x_{j, 1}, \dotsc, x_{j, m - 1}, 1)$ for all integers $0 \leq j \leq \jj$, such that for all $\sfw \in \overline{\ww}$ and $u, u' \in \mathtt{Q}$, there exists $1 \leq j \leq \jj$ such that
\begin{align*}
|(\tau^\sfw_{m} \circ v_j - \tau^\sfw_{m} \circ v_0)(u) - (\tau^\sfw_{m} \circ v_j - \tau^\sfw_{m} \circ v_0)(u')| \geq \varepsilon d(u, u').
\end{align*}
\end{proposition}

\begin{proof}
We will work on $\tilde{\mathsf{R}}$ to access the smooth structure on $\tilde{\mathsf{U}}$. Fix $C_\ww > 0$, $j_0 \in \N$, and $\delta > 0$ to be the constants provided by \cref{cor:ProjectionOfAdjointNonzero}. Let $z_1 \in \tilde{\mathsf{R}}_1$ be the center. Define the diffeomorphism $\psi: \tilde{\mathsf{U}}_1 \to (N^+_\Theta)_1$ such that $\mathsf{F}(u) = \mathsf{F}(z_1) \psi(u)$ for all $u \in \tilde{\mathsf{U}}_1$. Fix $C_1 = \max(\|(d\psi_{z_1})\|_{\mathrm{op}}, \|d(\psi^{-1})_e\|_{\mathrm{op}}) > 0$. Recall the diffeomorphism $h_{n^-}: (N^+_\Theta)_1 \to h_{n^-}((N^+_\Theta)_1)$ from \cref{lem:BrinPesinDerivativeImageIsAdjointProjection} which is smooth in $n^- \in (N^-_\Theta)_{1, \epsilon_e}$ and satisfies $h_{n^-}(e) = e$ and $h_e = \Id_{(N^+_\Theta)_1}$. Without loss of generality, we may assume that the size $\hat{\delta}$ of the Markov section $\mathcal{R}$ is sufficiently small so that $\frac{(dh_{n^-})_e(n^+)}{\|(dh_{n^-})_e(n^+)\|} \subset B_{2\delta}^{\LieN^+_\Theta}(n_0^+)$ for all $n^- \in (N^-_\Theta)_1$, and $n^+ \in \mathbb{S}(\LieN^+_\Theta) \cap B_\delta^{\LieN^+_\Theta}(n_0^+)$, and $n_0^+ \in \mathbb{S}(\LieN^+_\Theta)$. By compactness of $(N^-_\Theta)_1$, we may take $C_2 = \sup_{n^- \in (N^-_\Theta)_1}\max\bigl(\|(dh_{n^-})_e\|_{\mathrm{op}}, \bigl\|d\bigl(h_{n^-}^{-1}\bigr)_e\bigr\|_{\mathrm{op}}\bigr) > 0$. Fix $\varepsilon = \frac{C_\ww}{3C_1C_2}$ and $\varepsilon' = \frac{\varepsilon}{C_1C_2}$. Since $(\pi_\sfw \circ \Ad_e)|_{\LieN^+_\Theta} = 0$ for all $\sfw \in \overline{\ww}$ and $\overline{\ww}$ is compact, decreasing $\delta > 0$ if required, we have the trivial upper bound $\bigl\|(\pi_\sfw \circ \Ad_{n^-})|_{\LieN^+_\Theta}\bigr\|_{\mathrm{op}} \leq \varepsilon'$ for all $n^- \in N^-_\Theta$ with $d_{N^-_\Theta}(n^-, e) \leq \delta$.

Let us construct the appropriate admissible sequences. Fix $n_0^- = e \in (N^-_\Theta)_1$ and $n_1^-, n_2^-, \dotsc, n_{\jj}^- \in (N^-_\Theta)_1$ to be the elements provided by \cref{cor:ProjectionOfAdjointNonzero}. Using the Markov property and the topological mixing property of the transition matrix $T$ (see \cref{subsec:SymbolicDynamics}), we can fix $m_0 \in \N$ such that, given any $m \geq m_0$, there exists distinct $\eta_0^-, \eta_1^-, \dotsc, \eta_{\jj}^- \in (N^-_\Theta)_1$ with $d_{N^-_\Theta}(\eta_j^-, n_j^-) < \delta$ for all $0 \leq j \leq \jj$ such that, writing $\mathsf{F}(z_1)\eta_j^- = \mathsf{F}(s_j)$, the following is satisfied: the point $s_j$ is the image of some point in $\hat{\mathsf{U}}$ under the $m$\textsuperscript{th} iterate $\mathcal{P}^m$ of the Poincar\'{e} first return map, i.e., $\mathcal{P}^m(u_j) = s_j$ for some $u_j \in \hat{\mathsf{U}}_{x_{j, 0}}$ and $x_{j, 0} \in \mathcal{A}$. Then the associated trajectory of the translation flow of $u_j$ through the Markov section gives corresponding admissible sequences $x_j = (x_{j, 0}, x_{j, 1}, \dotsc, x_{j, m - 1}, 1)$ and sections $v_j = \sigma^{-x_j}: \tilde{\mathsf{U}}_1^{(m)} \to {}^{\mathsf{g}^{-x_j}}\tilde{\mathsf{U}}_{x_{j, 0}}$ for all $0 \leq j \leq \jj$. Note that by definition, $\eta_j^- = n_{x_j}$ for all $0 \leq j \leq \jj$.

Let $\sfw \in \overline{\ww}$ for the rest of the proof. Let $1 \leq j \leq \jj$. We define the map
\begin{align*}
\varphi_j &:= \tau^\sfw_{x_j} \circ v_j - \tau^\sfw_{x_0} \circ v_0 : \tilde{\mathsf{U}}_1^{(m)} \to \R.
\end{align*}
We can relate $\varphi_j$ to $\Xi$ via the following observation obtained using definitions and \cref{lem:BrinPesinInTermsOfHolonomy}:
\begin{align}
\label{eqn:BP_and_Xi_Relation}
(\varphi_j(u) - \varphi_j(z_1))\sfw = \bigl(\log \circ \tilde{\pi}_\sfw \circ \Xi_{n_{x_j}}\bigr)(\psi(u)) - \bigl(\log \circ \tilde{\pi}_\sfw \circ \Xi_{n_{x_0}}\bigr)(\psi(u))
\end{align}
for all $u \in \tilde{\mathsf{U}}_1^{(m)}$. Now, a crucial point to observe is that although the left hand side of \cref{eqn:BP_and_Xi_Relation} may not be defined on $\tilde{\mathsf{U}}_1$ (since $\tilde{\mathsf{U}}_1^{(m)} \subset \tilde{\mathsf{U}}_1$ may be strictly smaller), the right hand side of \cref{eqn:BP_and_Xi_Relation} \emph{is well-defined} on $\tilde{\mathsf{U}}_1$. Therefore, for the time being, we work with the map
\begin{align*}
\varrho_j := \bigl(\log \circ \tilde{\pi}_\sfw \circ \Xi_{n_{x_j}} - \log \circ \tilde{\pi}_\sfw \circ \Xi_{n_{x_0}}\bigr) \circ \psi : \tilde{\mathsf{U}}_1 \to \R\sfw.
\end{align*}
Let $Z \in \mathbb{S}(\T_{z_1}(\tilde{\mathsf{U}}_1))$. Taking the differential at $z_1$ and evaluating at $Z$ gives
\begin{align*}
(d\varrho_j)_{z_1}(Z) = d\bigl(\log \circ \tilde{\pi}_\sfw \circ \Xi_{n_{x_j}}\bigr)_e ((d\psi)_{z_1}(Z)) - d\bigl(\log \circ \tilde{\pi}_\sfw \circ \Xi_{n_{x_0}}\bigr)_e ((d\psi)_{z_1}(Z)).
\end{align*}
Let us write $W_j = \bigl(dh_{n_{x_j}}\bigr)_e((d\psi)_{z_1}(Z))$ and $W_0 = \bigl(dh_{n_{x_0}}\bigr)_e((d\psi)_{z_1}(Z))$ for ease of notation. Applying \cref{lem:BrinPesinDerivativeImageIsAdjointProjection} and using the reverse triangle inequality gives
\begin{align*}
\|(d\varrho_j)_{z_1}(Z)\| \geq \bigl\|\pi_\sfw\bigl(\Ad_{n_{x_j}}(W_j)\bigr)\bigr\| - \bigl\|\pi_\sfw\bigl(\Ad_{n_{x_0}}(W_0)\bigr)\bigr\|.
\end{align*}
Fix $\mathscr{s}_1, \mathscr{s}_2, \dotsc, \mathscr{s}_{\jj} \subset \mathbb{S}(\LieN^+_\Theta)$ to be the finite subsets with $\bigcup_{j = 1}^{\jj} B_\delta^{\LieN^+_\Theta}(\mathscr{s}_j) \supset \mathbb{S}(\LieN^+_\Theta)$ provided by \cref{cor:ProjectionOfAdjointNonzero}. Then, \emph{there exists $1 \leq j_Z \leq \jj$} depending on $Z$ such that $\frac{(d\psi)_{z_1}(Z)}{\|(d\psi)_{z_1}(Z)\|} \in B_\delta^{\LieN^+_\Theta}(\mathscr{s}_{j_Z})$. Note that our initial estimates ensure that $\frac{W_{j_Z}}{\|W_{j_Z}\|} \in B_{2\delta}^{\LieN^+_\Theta}(\mathscr{s}_{j_Z})$. Applying \cref{cor:ProjectionOfAdjointNonzero} again, we can bound the first term from below:
\begin{align*}
\bigl\|\pi_\sfw\bigl(\Ad_{n_{x_{j_Z}}}(W_{j_Z})\bigr)\bigr\| \geq C_\ww\|W_{j_Z}\| \geq \frac{C_\ww}{C_1C_2} \geq 3\varepsilon
\end{align*}
and we can bound the second term trivially from above:
\begin{align*}
\bigl\|\pi_\sfw\bigl(\Ad_{n_{x_0}}(W_0)\bigr)\bigr\| \leq \varepsilon'\|W_0\| \leq \varepsilon'C_1C_2 \leq \varepsilon.
\end{align*}
As a result, we obtain the lower bound
\begin{align*}
|(d\varrho_{j_Z})_{z_1}(Z)\| \geq 2\varepsilon \qquad \text{for all $Z \in \mathbb{S}(\T_{z_1}(\tilde{\mathsf{U}}_1))$}.
\end{align*}
Finally, by smoothness of $\varrho_j$ for all $1 \leq j \leq \jj$ and compactness of $\mathbb{S}(\T_{z_1}(\tilde{\mathsf{U}}_1))$, there exists a uniform open neighborhood $\mathcal{U} \subset \tilde{\mathsf{U}}_1$ of $z_1$ such that the following holds: for all $u \in \mathcal{U}$ and $Z \in \mathbb{S}(\T_{z_1}(\tilde{\mathsf{U}}_1))$, there exists $1 \leq j \leq \jj$ such that
\begin{align*}
\|(d\varrho_j)_u(Z)\| \geq \varepsilon.
\end{align*}
We may assume that $\mathcal{U}$ is \emph{convex} by shrinking it to an open ball centered at $z_1$ if necessary. Using convexity of $\mathcal{U}$, we then immediately deduce the following: for all $u, u' \in \mathcal{U}$, there exists $1 \leq j \leq \jj$ such that
\begin{align*}
\|\varrho_j(u) - \varrho_j(u')\| \geq \varepsilon d(u, u').
\end{align*}
Finally, restricting to a cylinder $\mathtt{Q} \subset \mathcal{U} \cap \mathsf{U}_1$ and going back to the maps $\varphi_j$ for all $1 \leq j \leq \jj$, we conclude the following: for all $u, u' \in \mathtt{Q}$, there exists $1 \leq j \leq \jj$ such that
\begin{align*}
|\varphi_j(u) - \varphi_j(u')| \geq \varepsilon d(u, u') \qquad \text{for all $u, u' \in \mathtt{Q}$}.
\end{align*}
\end{proof}

We can immediately upgrade the cylinder $\mathtt{Q}$ from \cref{pro:LNIC} to $\interior(U)$ in the following corollary using the topological mixing property of the transition matrix $T$ (see \cref{subsec:SymbolicDynamics}). The mutual disjointness in \cref{pro:NIC} follows from mutual distinctness of the admissible sequences in \cref{pro:LNIC}.

\begin{proposition}
\label{pro:NIC}
There exist $\varepsilon \in (0, 1)$, $m_0 \in \N$, and $\jj \in \N$ such that for all $m \geq m_0$, there exists a set of Lipschitz sections $\{v_j: U \to U\}_{j = 0}^{\jj}$ of $\sigma^m$ such that for all $\sfw \in \overline{\ww}$ and $u, u' \in \interior(U)$, there exists $1 \leq j \leq \jj$ such that
\begin{align*}
|(\tau^\sfw_{m} \circ v_j - \tau^\sfw_{m} \circ v_0)(u) - (\tau^\sfw_{m} \circ v_j - \tau^\sfw_{m} \circ v_0)(u')| \geq \varepsilon d(u, u').
\end{align*}
Moreover, $v_0(U), v_1(U), \dotsc, v_{\jj}(U)$ are mutually disjoint.
\end{proposition}

Fix $\varepsilon$, $m_0$, and $\jj$ to be the ones provided by \cref{pro:NIC} henceforth.

\section{Dolgopyat operators and the proof of \texorpdfstring{\cref{thm:DolgopyatsMethod}}{\autoref{thm:DolgopyatsMethod}}}
\label{sec:ProofOfTheorem}
In this section we carry out Dolgopyat's method to prove \cref{thm:DolgopyatsMethod}. Many of the techniques used here go all the way back to Dolgopyat \cite{Dol98}. Since it is sensitive to the setting at hand, some of the arguments require careful treatment in each instance. Over the years, it has been made cleaner and more efficient in some settings such as by Avila--Gou\"{e}zel--Yoccoz \cite{AGY06} and Stoyanov \cite{Sto11}.

Thanks to \cref{rem:Weak_VS_Strong_LNIC}, we are able to follow Stoyanov's strategy using cylinders and the new distance $D$. One especially important advantage of Stoyanov's version of Dolgopyat's method is that we do not require the Federer/doubling property. This is not just a convenience---there is no appropriate metric to use on $\LieN^+_\Theta$ which is also compatible with $U$ and an analogue of the doubling property is not known yet for PS measures in higher rank to the best of the authors' knowledge. The reader may also find it useful to consult \cite{OW16,Sar19,Sar23} where many full proofs are provided in a similar framework.

We fix $\sfw \in \ww$ for the rest of the section. Recall that this implies $\Theta_\sfw = \Theta$ and hence, \cref{pro:NIC} applies in this section.

\subsection{Preliminary lemmas and constants}
\label{subsec:PreliminaryLemmasAndConstants}
We obtain the following eventually expanding/contracting property, which also implies the same for the new distance function $D$, from the metric Anosov property of the translation flow in \cref{thm:TranslationFlowIsMetricAnosov}, and a compactness argument for the lower bound. Another lemma follows from it as in \cite[Proposition 3.3]{Sto11}.

\begin{lemma}
\label{lem:SigmaHyperbolicity}
There exist $c_0 \in (0, 1)$ and $\kappa_1 > \kappa_2 > 1$ such that for all cylinders $\mathtt{C} \subset U$ with $\len(\mathtt{C}) = j \in \N$, and $u, u' \in \mathtt{C}$, we have both
\begin{gather*}
c_0 \kappa_2^j d(u, u') \leq d(\sigma^j(u), \sigma^j(u')) \leq {c_0}^{-1}\kappa_1^j d(u, u') \\
c_0 \kappa_2^j D(u, u') \leq D(\sigma^j(u), \sigma^j(u')) \leq {c_0}^{-1}\kappa_1^j D(u, u').
\end{gather*}
\end{lemma}

\begin{lemma}
\label{lem:CylinderDiameterBound}
There exist $p_0 \in \N$ and $\rho \in (0, 1)$ such that for all cylinders $\mathtt{C} \subset U$ with $\len(\mathtt{C}) = l \in \Z_{\geq 0}$ and subcylinders $\mathtt{C}', \mathtt{C}'' \subset \mathtt{C}$ with $\len(\mathtt{C}') = l + 1$ and $\len(\mathtt{C}'') = l + p_0$, we have
\begin{align*}
\diam(\mathtt{C}'') \leq \rho\diam(\mathtt{C}) \leq \diam(\mathtt{C}').
\end{align*}
\end{lemma}

We fix constants $c_0$, $\kappa_1$, $\kappa_2$, $p_0$, and $\rho$ provided by \cref{lem:SigmaHyperbolicity,lem:CylinderDiameterBound} henceforth. We use these bounds extensively without comments. We also fix $p_1 \in \N$ such that
\begin{align}
\label{eqn:Constantp1}
\frac{1}{2} - 2\rho^{p_1 - 1} \geq \frac{1}{16}.
\end{align}

The following is a Lasota--Yorke \cite{LY73} type of lemma whose proof is similar to that of \cite[Lemma 5.4]{Sto11} and \cite[Lemma 3.9]{OW16}. Its proof gives the explicit constant $A_0 > 2c_0^{-1}e^{\frac{T_\sfw}{c_0(\kappa_2 - 1)}}\max\bigl(1, \frac{T_\sfw}{\kappa_2 - 1}\bigr) > 2$ which we fix henceforth.

\begin{lemma}
\label{lem:PreliminaryLogLipschitz}
There exists $A_0 > 0$ such that for all $|a| < a_0'$, $|b| > 1$, and $k \in \N$, we have
\begin{enumerate}
\item\label{itm:PreliminaryLogLipschitzProperty1}	if $h \in \mathcal{C}_B(U)$ (resp. $h \in \tilde{\mathcal{C}}_B(U)$) for some $B > 0$, then $\mathcal{L}_{a, \sfw}^k(h) \in \mathcal{C}_{B'}(U)$ (resp. $\mathcal{L}_{a, \sfw}^k(h) \in \tilde{\mathcal{C}}_{B'}(U)$) for $B' = A_0\left(\frac{B}{\kappa_2^k} + 1\right)$;
\item\label{itm:PreliminaryLogLipschitzProperty2}	if $H \in C(U)$ and $h \in B(U)$ satisfy
\begin{align*}
\|H(u) - H(u')\|_2 \leq Bh(u)D(u, u')
\end{align*}
for all $u, u' \in U_j$, for all $j \in \mathcal{A}$, for some $B > 0$, then for all $j \in \mathcal{A}$, for all $u,u' \in U_j$, we have
\begin{align*}
\left\|\mathcal{L}_{\xi, \sfw}^k(H)(u) - \mathcal{L}_{\xi, \sfw}^k(H)(u')\right\|_2 \leq A_0\left(\frac{B}{\kappa_2^k}\mathcal{L}_{a, \sfw}^k(h)(u) + |b|\mathcal{L}_{a, \sfw}^k\|H\|(u)\right)D(u, u')
\end{align*}
for all $u, u' \in U$.
\end{enumerate}
\end{lemma}

Now, we fix $b_0 = 1$ and some other significant positive constants
\begin{align}
\label{eqn:ConstantE}
E &> 2A_0 > 4, \\
\label{eqn:Constantepsilon1}
\epsilon_1 &< \min\left(\hat{\delta}, \frac{\pi c_0(\kappa_2 - 1)}{2T_\sfw}\right), \\
\label{eqn:Constantm}
m &\geq m_0 \text{ such that } \kappa_2^m > \max\left(8A_0, \frac{4E\epsilon_1\rho^{p_1}}{c_0}, \frac{4 \cdot 128E}{c_0 \varepsilon \rho}\right), \\
\label{eqn:Constantmu}
\mu &< \min\left(\frac{2E\epsilon_1c_0 \rho^{p_0p_1 + 1}}{\kappa_1^m}, \frac{1}{4}, \frac{1}{16 \cdot 16e^{2mT_\sfw}}\left(\frac{\varepsilon\rho\epsilon_1}{64}\right)^2\right).
\end{align}
Having fixed $m \geq m_0$, we also fix the corresponding set of Lipschitz sections $\{v_j: U \to U\}_{j = 0}^{\jj}$ of $\sigma^m$ provided by \cref{pro:NIC}.

\subsection{Construction of Dolgopyat operators}
\label{subsec:ConstructionOfDolgopyatOperators}
Let $|b| > b_0$. We define the set $\{\mathtt{C}_1(b), \mathtt{C}_2(b), \dotsc, \mathtt{C}_{c_b}(b)\}$ for some $c_b \in \N$ consisting of maximal cylinders $\mathtt{C} \subset U$ with $\diam(\mathtt{C}) \leq \frac{\epsilon_1}{|b|}$ so that $U = \bigcup_{l = 1}^{c_b} \overline{\mathtt{C}_l(b)}$. We define the set $\{\mathtt{D}_1(b), \mathtt{D}_2(b), \dotsc, \mathtt{D}_{d_b}(b)\}$ for some $d_b \in \N$ consisting of subcylinders $\mathtt{D} \subset \mathtt{C}_l(b)$ with $\len(\mathtt{D}) = \len(\mathtt{C}_l(b)) + p_0p_1$ and $1 \leq l \leq c_b$. We define the index set $\Xi(b) = \{0, 1, \dotsc, \jj\} \times \{1, 2, \dotsc, d_b\}$. We define $\mathtt{X}_{j, k}(b) = \overline{v_j(\mathtt{D}_k(b))}$ which satisfies $\mathtt{X}_{j, k}(b) \cap \mathtt{X}_{j', k'}(b) = \varnothing$ for all distinct $(j, k), (j', k') \in \Xi(b)$. For all $J \subset \Xi(b)$, we define the function
\begin{align*}
\beta_J = \chi_{U} - \mu \sum_{(j, k) \in J} \chi_{\mathtt{X}_{j, k}(b)}.
\end{align*}
We record a number of basic facts derived from \cref{lem:SigmaHyperbolicity,lem:CylinderDiameterBound}.

\begin{itemize}
\item For all $1 \leq l \leq c_b$ and $(j, k) \in \Xi(b)$, we have the diameter bounds:
\begin{gather}
\label{eqn:DiameterBoundC_l}
\rho\frac{\epsilon_1}{|b|} \leq \diam(\mathtt{C}_l(b)) \leq \frac{\epsilon_1}{|b|} \\
\label{eqn:DiameterBoundD_k}
\rho^{p_0p_1 + 1}\frac{\epsilon_1}{|b|} \leq \diam(\mathtt{D}_k(b)) \leq \rho^{p_1}\frac{\epsilon_1}{|b|} 
\\
\label{eqn:DiameterBoundX_jk}
\frac{\epsilon_1c_0\rho^{p_0p_1 + 1}}{|b|\kappa_1^m} \leq \diam(\mathtt{X}_{j, k}(b)) \leq \frac{\epsilon_1\rho^{p_1}}{|b|c_0\kappa_2^m}.
\end{gather}
\item For all $J \subset \Xi(b)$, we have $\beta_J \in L_D(U)$ with Lipschitz constant (cf. \cite[Lemma 5.2]{Sto11}):
\begin{align}
\label{eqn:LipschitzConstantbeta_J}
\Lip_D(\beta_J) \leq \frac{\mu}{\min_{(j, k) \in J} \diam(\mathtt{X}_{j, k}(b))} \leq \frac{\mu|b|\kappa_1^m}{\epsilon_1c_0\rho^{p_0p_1 + 1}}.
\end{align}
\item For all $1 \leq l \leq c_b$, if $u, u' \in \mathtt{C}_l(b)$, then
\begin{align}
\label{eqn:DBoundOnsigma^m1C_l}
D(v_j(u), v_j(u')) \leq \frac{\epsilon_1}{|b|c_0\kappa_2^m}.
\end{align}
\end{itemize}

\begin{definition}
For all $\xi \in \C$ with $|a| < a_0'$ and $|b| > b_0$, and $J \subset \Xi(b)$, we define the \emph{Dolgopyat operators} $\mathcal{N}_{a, J}^{\sfw}: L_D(U) \to L_D(U)$ by
\begin{align*}
\mathcal{N}_{a, J}^{\sfw}(h) = \mathcal{L}_{a}^m(\beta_J h)
\end{align*}
for all $h \in L_D(U)$.
\end{definition}

\begin{definition}
For all $|b| > b_0$, a subset $J \subset \Xi(b)$ is said to be \emph{dense} if for all $1 \leq l \leq c_b$, there exists $(j, k) \in J$ such that $\mathtt{D}_k(b) \subset \mathtt{C}_l(b)$. Denote $\mathcal{J}(b)$ to be the set of all dense subsets of $\Xi(b)$.
\end{definition}

\subsection{Proof of \texorpdfstring{\cref{thm:DolgopyatsMethod}}{\autoref{thm:DolgopyatsMethod}}}
\Cref{itm:DolgopyatProperty1,itm:LogLipschitzDolgopyat} in \cref{thm:DolgopyatsMethod} are derived from \cref{lem:PreliminaryLogLipschitz} using estimates from \cref{eqn:ConstantE,eqn:Constantm,eqn:Constantmu}. We omit the proofs of since they are almost identical to those in \cite{Sto11,OW16,Sar19,Sar23}.

Similarly, we also omit the proof of \cref{itm:DolgopyatProperty2} in \cref{thm:DolgopyatsMethod} and refer the reader to \cite{Sto11,Sar19,Sar23}, in particular \cite[\S\,13.1]{Sar23}, where full proofs are provided. However, as alluded to previously, we emphasize that the proof of \cref{itm:DolgopyatProperty2} uses the Gibbs property of the measure $\nu_U^{\sfw}$ given as follows: there exist $c_1^U, c_2^U > 0$ such that
\begin{align}
\label{eqn:PropertyOfGibbsMeasures}
c_1^U e^{-\delta_\sfw \tau^\sfw_k(u)} \leq \nu_U^{\sfw}(\mathtt{C}) \leq c_2^U e^{-\delta_\sfw \tau^\sfw_k(u)}
\end{align}
for all $u \in \mathtt{C}$ and cylinders $\mathtt{C}$ with $\len(\mathtt{C}) = k \in \mathbb Z_{\geq 0}$. The Gibbs property is automatically satisfied since it comes from an equilibrium state and replaces the role of the Federer/doubling property. The outcome of the proof of \cref{itm:DolgopyatProperty2} is the following:
\begin{itemize}
\item we obtain a continuous function $\eta: \R_{>0} \to \R_{>0}$ (of $\delta_\sfw$) which increases to $1$ at an exponential rate (due to the Gibbs property) for \cref{thm:DolgopyatsMethod};
\item we obtain a continuous function $a_0: \R_{>0} \to \R_{>0}$ (of $\delta_\sfw$) which decays exponentially to $0$ (due to the Gibbs property) for \cref{thm:DolgopyatsMethod}.
\end{itemize}

For the sake of completeness, we include the proof of \cref{itm:DominatedByDolgopyat} in \cref{thm:DolgopyatsMethod} where the crucial LNIC from \cref{pro:NIC} is used. Here we use $\Theta_\sfw = \Theta$.

\begin{lemma}
\label{lem:LNIC_Output}
Let $|b| > b_0$. Suppose $\mathtt{D}_k(b), \mathtt{D}_{k'}(b) \subset \mathtt{C}_l(b)$ for some $1 \leq k, k' \leq d_b$ and $1 \le l \le c_b$ such that $d(u_0,u_0') \ge \frac{1}{2}\diam(\mathtt{C}_l(b))$ for some $u_0 \in \mathtt{D}_k(b)$ and $u_0' \in \mathtt{D}_{k'}(b)$. Then, there exists $1 \le j \le \jj$ such that
\begin{align*}
\frac{\varepsilon\rho\epsilon_1}{16} \leq |b|\cdot|(\tau^\sfw_m\circ v_j-\tau^\sfw_m\circ v_0)(u) - (\tau^\sfw_m\circ v_j - \tau^\sfw_m \circ v_0)(u')| \leq \pi
\end{align*}
for all $u \in \mathtt{D}_k(b)$ and $u' \in \mathtt{D}_{k'}(b)$.
\end{lemma}

\begin{proof}
Let $u_0$, $u_0'$, $u$, and $u'$ be as in the lemma. The upper bound always holds on $\mathtt{C}_l(b)$ by \cref{eqn:Constantepsilon1}. Recall that $\sfw \in \ww$ and hence $\Theta_\sfw = \Theta$ in this section. Thus, the lemma follows from \cref{pro:NIC} and the estimate
\begin{align*}
d(u, u') \geq d(u_0, u_0') - d(u, u_0) - d(u', u_0') \geq \frac{\rho\epsilon_1}{2|b|} - 2\frac{\rho^{p_1}\epsilon_1}{|b|} \geq \frac{\rho\epsilon_1}{16|b|}
\end{align*}
using \cref{eqn:Constantp1}.
\end{proof}

Now, for all $\xi \in \C$ with $|a| < a_0'$ and $|b| > b_0$, for all integers $1 \leq j \leq \jj$, for all $H \in L(U,\C)$ and for all $h \in \mathcal{C}_{E|b|}(U)$, we define the functions $\chi_{j,0}^{[\xi, H, h]}, \chi_{j,1}^{[\xi, H, h]}: U \to \R$ by
\begin{align*}
\chi_{j,0}^{[\xi, H, h]}(u) 
&=\frac{\bigl|e^{(\tau_m^{\sfw, (a)} + ib\tau^\sfw_m)(v_j(u))} H(v_j(u)) + e^{(\tau_m^{\sfw, (a)} + ib\tau^\sfw_m)(v_0(u))} H(v_0(u))\bigr|}{e^{\tau_m^{\sfw, (a)}(v_j(u))}h(v_j(u)) + (1 - \mu)e^{\tau_m^{\sfw, (a)}(v_0(u))}h(v_0(u))}, \\
\chi_{j,1}^{[\xi, H, h]}(u) 
&= \frac{\bigl|e^{(\tau_m^{\sfw, (a)} + ib\tau^\sfw_m)(v_j(u))} H(v_j(u)) + e^{(\tau_m^{\sfw, (a)} + ib\tau^\sfw_m)(v_0(u))}H(v_0(u))\bigr|}{(1 - \mu)e^{\tau_m^{\sfw, (a)}(v_j(u))}h(v_j(u)) + e^{\tau_m^{\sfw, (a)}(v_0(u))}h(v_0(u))}
\end{align*}
for all $u \in U$. We need some lemmas to deduce \cref{itm:DominatedByDolgopyat} in \cref{thm:DolgopyatsMethod}. For a proof of the following lemma, see for instance \cite[Lemma 3.17]{OW16}.

\begin{lemma}
\label{lem:HTrappedByh}
Let $|b| > b_0$. Suppose $H \in L(U,\C)$ and $h \in \mathcal{C}_{E|b|}(U)$ satisfy \cref{itm:DominatedByh,itm:LogLipschitzh} in \cref{thm:DolgopyatsMethod}. Then for all $(j, k) \in \Xi(b)$, we have
\begin{align*}
\frac{1}{2} \leq \frac{h(v_j(u))}{h(v_j(u'))} \leq 2 \qquad \text{for all $u, u' \in \mathtt{D}_k(b)$}
\end{align*}
and also either of the alternatives
\begin{alternative}
\item\label{alt:HLessThan3/4h}	$|H(v_j(u))| \leq \frac{3}{4}h(v_j(u))$ for all $u \in \mathtt{D}_k(b)$,
\item\label{alt:HGreaterThan1/4h}	$|H(v_j(u))|\geq \frac{1}{4}h(v_j(u))$ for all $u \in \mathtt{D}_k(b)$.
\end{alternative}
\end{lemma}

For any $w_1,w_2 \in \C-\{0\}$, let $\angle(w_1,w_2) \in [0,\pi]$ denote the angle between $w_1$ and $w_2$ viewed as vectors in $\C \cong \R^2$. The following lemma is a stronger version of the triangle inequality proven by elementary trigonometry.

\begin{lemma}
\label{lem:StrongTriangleInequality}
Suppose $w_1, w_2 \in \C-\{0\}$ such that $\angle(w_1, w_2) \geq \alpha$ and $\frac{|w_1|}{|w_2|} \leq L$ for some $\alpha \in [0, \pi]$ and $L \geq 1$. Then we have 
\begin{align*}
|w_1 + w_2| \leq \left(1 - \frac{\alpha^2}{16L}\right)|w_1| + |w_2|.
\end{align*}
\end{lemma}

\begin{lemma}
\label{lem:chiLessThan1}
Let $\xi \in \C$ with $|a| < a_0'$ and $|b| > b_0$. Suppose $H \in L(U, \C)$ and $h \in \mathcal{C}_{E|b|}(U)$ satisfy \cref{itm:DominatedByh,itm:LogLipschitzh} in \cref{thm:DolgopyatsMethod}. For all integers $1 \leq l \leq c_b$, there exists $(j, k) \in \Xi(b)$ such that $\mathtt{D}_k(b) \subset \mathtt{C}_l(b)$ and such that $\chi_{j,0}^{[\xi, H, h]}(u) \leq 1$ or $\chi_{j,1}^{[\xi, H, h]}(u) \leq 1$ for all $u \in \mathtt{D}_k(b)$.
\end{lemma}

\begin{proof}
Let $\xi$, $H$, $h$, and $l$ be as in the lemma. Suppose \cref{alt:HLessThan3/4h} in \cref{lem:HTrappedByh} holds for some $(j, k)\in \Xi(b)$. Then it is a straightforward calculation to check that $\chi_{j, 1}^{[\xi, H, h]}(u) \leq 1$ for all $u \in \mathtt{D}_k(b)$, using \cref{eqn:Constantmu}. Otherwise, \cref{alt:HGreaterThan1/4h} in \cref{lem:HTrappedByh} holds for all $(j, k) \in \Xi(b)$. Choose $1 \leq k, k' \leq d_b$ satisfying the hypotheses in \cref{lem:LNIC_Output} and $1 \le j \le \jj$ satisfying the conclusion of \cref{lem:LNIC_Output}. Let $u \in \mathtt{D}_k(b)$ and $u' \in \mathtt{D}_{k'}(b)$. Note that $|H(v_\ell(u))|, |H(v_\ell(u'))| > 0$ for all $\ell \in \{0, j\}$. We would like to apply \cref{lem:StrongTriangleInequality} but first we need to establish bounds on relative angle and relative size. We start with the former. For all $\ell \in \{0, j\}$, let $u_\ell \in \{u, u'\}$ such that $|H(v_\ell(u_\ell))| = \min(|H(v_\ell(u))|, |H(v_\ell(u'))|)$. Then recalling the hypotheses for $H$ and $h$, and \cref{eqn:DBoundOnsigma^m1C_l,eqn:Constantm}, for all $\ell \in \{0, j\}$, we have
\begin{align*}
\frac{|H(v_\ell(u)) - H(v_\ell(u'))|}{\min\left(|H(v_\ell(u))|, |H(v_\ell(u'))|\right)} &\leq \frac{E|b|h(v_\ell(u_\ell))D(v_\ell(u), v_\ell(u'))}{|H(v_\ell(u_\ell))|} 
\\
&\leq 4E|b| \cdot \frac{\epsilon_1}{|b|c_0\kappa_2^m} \leq \frac{\varepsilon\rho\epsilon_1}{128}.
\end{align*}
Using elementary geometry, the above shows that $\sin(\angle(H(v_\ell(u)), H(v_\ell(u')))) \leq \frac{\varepsilon\rho\epsilon_1}{128}$ with $\angle(H(v_\ell(u)), H(v_\ell(u'))) \in [0, \frac{\pi}{2})$, for all $\ell \in \{0, j\}$ and hence, $$\angle(H(v_\ell(u)), H(v_\ell(u'))) \leq 2\sin(\angle(H(v_\ell(u)), H(v_\ell(u')))) \leq \frac{\varepsilon\rho\epsilon_1}{64}$$ 
for all $\ell \in \{0, j\}$. For notational convenience, we define 
$\varphi: U \to \R$ by 
$$\varphi(w) = b(\tau^\sfw_m\circ v_j(w) - \tau^\sfw_m\circ v_0(w)) \qquad \text{for all $w \in U$}.$$
By \cref{lem:LNIC_Output}, we have $\frac{\varepsilon\rho\epsilon_1}{16} \leq |\varphi(u) - \varphi(u')| \leq \pi$. The second inequality is to ensure that we take the correct branch of angle in the following calculations. We will use these bounds to obtain a lower bound for $\angle(V_j(u), V_0(u))$ or $\angle(V_j(u'), V_0(u'))$ where we define
\begin{align*}
V_\ell(w) = e^{(\tau_m^{\sfw, (a)} + ib\tau^\sfw_m)(v_\ell(w))}H(v_\ell(w)) \qquad \text{for all $w \in U$ and $\ell \in \{0, j\}$}.
\end{align*}
Using the triangle inequality and previously computed bounds, we have
\begin{align*}
& \angle\left(V_j(u),V_0(u)\right) = \angle\bigl(e^{i\varphi(u)} H(v_j(u)), H(v_0(u))\bigr) \\
\geq{}&\angle\bigl(e^{i\varphi(u)} H(v_j(u)), e^{i\varphi(u')} H(v_j(u))\bigr) - \angle\bigl(e^{i\varphi(u')} H(v_j(u)), e^{i\varphi(u')} H(v_j(u'))\bigr) \\
&{}- \angle(H(v_0(u)), H(v_0(u'))) - \angle\bigl(e^{i\varphi(u')} H(v_j(u')), H(v_0(u'))\bigr)\\
={}&|\varphi(u) - \varphi(u')| - \angle\left(H(v_j(u)), H(v_j(u'))\right) - \angle(H(v_0(u)), H(v_0(u'))) \\
&{}- \angle(V_j(u'),V_0(u')) \\
\ge{}& \frac{\varepsilon\rho\epsilon_1}{16} - \frac{\varepsilon\rho\epsilon_1}{64} - \frac{\varepsilon\rho\epsilon_1}{64} - \angle(V_j(u'),V_0(u'))
\end{align*}
and hence,
\begin{align*}
\angle\left(V_j(u),V_0(u)\right) + \angle\left(V_j(u'),V_0(u')\right) \ge \frac{\varepsilon\rho\epsilon_1}{32}
\end{align*}
for all $u \in \mathtt{D}_k(b)$ and $u' \in \mathtt{D}_{k'}(b)$. Without loss of generality, we may assume that $\angle(V_j(u), V_0(u)) \geq \frac{\varepsilon\rho\epsilon_1}{64}$ for all $u \in \mathtt{D}_k(b)$, which establishes the required bound on relative angle. For the bound on relative size, let $(\ell, \ell') \in \{(0, j), (j, 0)\}$ such that $h(v_\ell(u_0)) \leq h(v_{\ell'}(u_0))$ for some $u_0 \in \mathtt{D}_k(b)$. Then by \cref{lem:HTrappedByh}, we have
\begin{align*}
\frac{|V_\ell(u)|}{|V_{\ell'}(u)|} & = \frac{e^{\tau_m^{\sfw, (a)}(v_\ell(u))}\big|H(v_\ell(u))\big|}{e^{\tau_m^{\sfw, (a)}(v_{\ell'}(u))}\big|H(v_{\ell'}(u))\big|} \leq \frac{4e^{\tau_m^{\sfw, (a)}(v_\ell(u)) - \tau_m^{\sfw, (a)}(v_{\ell'}(u))}h(v_\ell(u))}{h(v_{\ell'}(u))} \leq 16e^{2mT_\sfw}
\end{align*}
for all $u \in \mathtt{D}_k(b)$, which establishes the required bound on relative size. Now applying \cref{lem:StrongTriangleInequality,eqn:Constantmu} and $|H| \leq h$ on $|V_j(u) + V_{0}(u)|$ gives $\chi_{j, 0}^{[\xi, H, h]}(u) \leq 1$ or $\chi_{j, 1}^{[\xi, H, h]}(u) \leq 1$ for all $u \in \mathtt{D}_k(b)$.
\end{proof}

One then derives the following; for a proof, see for instance \cite[Lemma 13.8]{Sar23}.

\begin{lemma}
\label{lem:DominatedByDolgopyat}
There exists $a_0 > 0$ such that for all $\xi \in \C$ with $|a| < a_0$ and $|b| > b_0$, if $H \in L(U, \C)$ and $h \in \mathcal{C}_{E|b|}(U)$ satisfy \cref{itm:DominatedByh,itm:LogLipschitzh} in \cref{thm:DolgopyatsMethod}, then there exists $J \in \mathcal{J}(b)$ such that
\begin{align*}
\left|\mathcal{L}_{\xi, \sfw}^m(H)(u)\right| \leq \mathcal{N}_{a, J}^{\sfw}(h)(u)
\end{align*}
for all $u \in U$.
\end{lemma}

\bibliographystyle{amsalpha}
\bibliography{References}
\end{document}